\documentclass[11pt]{amsart}
\usepackage{url}
\usepackage{hyperref,amssymb}
\hypersetup{colorlinks=true,linkcolor=blue,
urlcolor=cyan,citecolor=magenta} 
\usepackage[hyperpageref]{backref}
\usepackage{frcursive,calrsfs,mathrsfs}
\usepackage{xcolor} 
\newtheorem{theorem}{Theorem}[section]
\newtheorem{lemma}[theorem]{Lemma}
\newtheorem{corollary}[theorem]{Corollary}
\newtheorem{conjecture}[theorem]{Conjecture}
\newtheorem{proposition}[theorem]{Proposition}
\newtheorem{example}[theorem]{Example}

\newtheorem{definition}[theorem]{Definition}
\newtheorem{definitions}[theorem]{Definitions}
\newtheorem{remark}[theorem]{Remark}

\newtheorem{hypothesis}[theorem]{Hypothesis}
\numberwithin{equation}{section}
\usepackage[english]{babel}
\usepackage{url}
\usepackage{hyperref,amssymb}
\hypersetup{colorlinks=true,linkcolor=blue,
urlcolor=cyan,citecolor=magenta} 
\usepackage[hyperpageref]{backref}
\usepackage{frcursive,calrsfs,mathrsfs}
\usepackage{xcolor} 
\newcommand{\Q}{{\mathbb{Q}}}
\newcommand{\Z}{{\mathbb{Z}}}
\newcommand{\F}{{\mathbb{F}}}

\newcommand{\ds}{\displaystyle}

\newcommand{\ov}{\overline}
\newcommand{\wt}{\widetilde}
\newcommand{\ft}{\footnotesize}
\newcommand{\ns}{\normalsize}

\newcommand{\BH}{{\bf H}}
\newcommand{\CH}{{\mathcal H}}
\newcommand{\BZ}{{\bf Z}}
\newcommand{\BE}{{\bf E}}
\newcommand{\BF}{{\bf F}}
\newcommand{\CE}{{\mathcal E}}

\newcommand{\CK}{{\mathcal K}}

\newcommand{\CF}{{\mathcal F}}
\newcommand{\CI}{{\mathcal I}}
\newcommand{\CX}{{\mathcal X}}
\newcommand{\CY}{{\mathcal Y}}
\newcommand{\X}{{\rm X}}
\newcommand{\CT}{{\mathcal T}}
\newcommand{\T}{{\rm T}}
\newcommand{\CR}{{\mathcal R}}
\newcommand{\CW}{{\mathcal W}}
\newcommand{\CG}{{\mathcal G}}
\newcommand{\order}{\raise0.8pt \hbox{${\scriptstyle \#}$}}
\newcommand{\lien}{\mathrel{\mkern-4mu}}
\newcommand{\too}{\relbar\lien\rightarrow}

\newcommand{\plus}{\ds\mathop{\raise 0.5pt \hbox{$\bigoplus$}}\limits}
\newcommand{\prd}{\ds\mathop{\raise 1.0pt \hbox{$\prod$}}\limits}
\newcommand{\sm}{\ds\mathop{\raise 1.0pt \hbox{$\sum$}}\limits}
\newcommand{\ffrac}[2]{\hbox{\ft $\displaystyle\frac{#1}{#2}$}}
\newcommand{\bbinom}[2]{\hbox{$\binom{#1}{#2}$}}
\newcommand{\Gal}{{\rm Gal}}
\newcommand{\Norm}{{\bf N}}
\newcommand{\J}{{\bf J}}

\newcommand{\Nu}{\hbox{\Large $\nu$}}
\newcommand{\Lbda}{{\bf \Lambda}}

\newcommand{\Ker}{{\rm Ker}}
\newcommand{\ram}{{\rm ram}}
\newcommand{\nr}{{\rm nr}}
\newcommand{\pr}{{\rm pr}}
\newcommand{\gen}{{\rm gen}}
\newcommand{\ar}{{\rm ar}}
\newcommand{\alg}{{\rm alg}}
\newcommand{\lc}{{\rm lc}}
\newcommand{\bp}{{\rm bp}}
\newcommand{\cyc}{{\rm cy}}
\newcommand{\lgm}{{\rm lg}}
\newcommand{\tor}{{\rm tor}}
\newcommand{\rk}{{\rm rk}}
\newcommand{\Hom}{{\rm H}}
\newcommand{\N}{{\raise-1.8pt \hbox{\tiny $N$}}}
\newcommand{\Ll}{{\hbox{\ft $L$}}}
\newcommand{\Kk}{{\hbox{\ft $K$}}}
\newcommand{\limproj}{\mathop{\lim_{\longleftarrow}}}
\newcommand{\limind}{\mathop{\lim_{\longrightarrow}}}
\def\plus{\displaystyle\mathop{\raise 2.0pt \hbox{$\bigoplus $}}\limits}
\def\prd{\displaystyle\mathop{\raise 2.0pt \hbox{$\prod$}}\limits}
\def\sm{\displaystyle\mathop{\raise 2.0pt \hbox{$\sum$}}\limits}

\title[Algebraic norms and capitulation of $p$-class groups]
{Algebraic norms and capitulation of $p$-class \\ groups in ramified cyclic $p$-extensions}

\author[Georges Gras]{Georges Gras}

\address{Villa la Gardette, 4 chemin de Ch\^ateau Gagni\`ere, F-38520 Le 
Bourg d'Oisans}
\email{g.mn.gras@wanadoo.fr}

\urladdr{\url{http://orcid.org/0000-0002-1318-4414}}



\begin{document}

\date{February 5, 2023}

\begin{abstract} 
We examine the phenomenon of capitulation of the $p$-class group $\CH_K$ 
of a real number field $K$ in totally ramified cyclic $p$-extensions $L/K$ of degree 
$p^\N$. Using a property of the algebraic norm 
$\nu_{\!L/K}$, we show that the kernel of capitulation is in relation with the 
``complexity'' of $\CH_L$ measured via its exponent $p^{e(L)}$ and the length 
$m(L)$ of its Galois filtration $\{\CH_L^i\}_{i \geq 0}$. 
We prove that a {\it sufficient condition} of 
capitulation is given by $e(L) \in [1, N-s(L)]$ if $m(L) \in  [p^{s(L)}, p^{s(L)+1}-1]$ 
for $s(L) \in [0, N-1]$ (Theorem \ref{main1}); this improves the case of ``stability'' 
$\order \CH_L = \order \CH_K$ (i.e., $m(L) = 1$, $s(L)=0$, $e(L) = e(K)$) (Theorem 
\ref{main2}). Numerical examples (with PARI programs), showing most often  
capitulation of $\CH_K$ in $L $, are given, taking the simplest ones, $L \subset K(\mu_\ell^{})$, 
with primes $\ell \equiv 1 \pmod {2 p^\N}$, over cubic fields with $p=2$ and real 
quadratic fields with $p=3$. Some conjectures on the existence of non-zero 
densities of such $\ell$'s are proposed (Conjectures \ref{conjcap}, \ref{conjprobas}). 
Capitulation property of other arithmetic invariants is examined.
\end{abstract}

\maketitle

\tableofcontents

\section{Introduction} 

\subsection{Statement of the main results} 
Let $L/K$ be any cyclic $p$-extension of number fields, 
of degree $p^\N$, $N \geq 1$, of Galois group $G =: \langle \sigma \rangle$, 
and let $\CH_K$, $\CH_L$ be the $p$-class groups of $K$, $L$, respectively. 
For a finite group $A$, let $\rk_p (A)\! := \dim_{\F_p}(A/A^p)$ be the 
$p$-rank of $A$. 

We assume, in all the sequel, that $L/K$ is {\it totally ramified}.

Let $\J_{L/K} : \CH_K \to \CH_L^G$ be the transfer map (or extension of 
classes) and let $\Norm_{L/K}  : \CH_L \to \CH_K$ be the arithmetic norm
induced by $\Norm_{L/K}({\mathfrak P}) := {\mathfrak p}^f$, for all prime ideal 
${\mathfrak P}$ of $L$ and the prime ideal ${\mathfrak p}$ of $K$ under 
${\mathfrak P}$ with residue degree $f$.

We know that $\Norm_{L/K}(\CH_L)$ is the subgroup of $\CH_K$ which corresponds,
by the Artin map of class field theory, to $\Gal(H_K^\nr/L \cap H_K^\nr )$ where $H_K^\nr$ 
is the $p$-Hilbert class field of $K$; thus, by assumption of
total ramification, $\Norm_{L/K}(\CH_L) = \CH_K$.

Let $\Nu_{\!L/K} := \sm_{i=0}^{p^N-1} \sigma^i$, be the algebraic norm in $L/K$.
It is immediate to see that $\Nu_{\!L/K} = \J_{L/K} \circ \Norm_{L/K}$, whence
$\Nu_{\!L/K}(\CH_L) = \J_{L/K}(\CH_K)$ giving the key principle that
$\CH_K$ capitulates in $L$ if and only if $\Nu_{\!L/K}(\CH_L)=1$.

\smallskip
We can state, from using this criterion:

\begin{theorem}\label{main1}
Let $L/K$ be any totally ramified cyclic $p$-extension of number fields, 
of degree $p^\N$, $N \geq 1$, of Galois group $G =: \langle \sigma \rangle$.
Let $p^{e(\Ll)}$ be the exponent of $\CH_L$ and let $m(\Ll)$ be the minimal 
integer such that $(\sigma-1)^{m(\Ll)}$ annihilates $\CH_L$. 

(i) A sufficient condition for the capitulation of $\CH_K$ in $L$, is  
$e(\Ll) \in[1, N - s(\Ll)]$ if $m(\Ll) \in [p^{s(\Ll)}, p^{s(\Ll)+1}-1]$ for $s(\Ll) \in [0, N-1]$.

(ii) A class $h \ne 1$ of $\CH_K$ capitulates in $L$ if, in the writing $h =: \Norm_{L/K}(h')$, 
$h'$ is of order $p^e$ and annihilated by $(\sigma - 1)^m$ such that 
$e \in[1, N - s]$ if $m \in [p^s, p^{s+1}-1]$ for $s \in [0, N-1]$.
\end{theorem}

\begin{theorem}\label{main2}
Let $L/K$ be any totally ramified cyclic $p$-extension of number fields, 
of degree $p^\N$, $N \geq 1$, of Galois group $G =: \langle \sigma \rangle$:

(i) Let $K_n$, $n \in [0, N]$, be the subfield of $L$ of degree $p^n$ over $K$ and set 
$G_n := \Gal(K_n/K)$. Then $\order \CH_{K_1} = \order \CH_K$ implies the following 
properties: 

\quad $\bullet$ $\CH_{K_n} = \CH_{K_n}^{G_n}
\ds \mathop{\simeq}^{\hbox{\tiny$\Norm_{K_n/K}$}} \CH_K$, for all $n \in [0, N]$.

\quad $\bullet$ For any $e \leq N$, the subgroup $\CH_K[p^e]$ of $\CH_K$, of 
classes annihilated by~$p^e$, capitulates in $K_e$; whence $\CH_K$
capitulates in $K_{e(\Kk)}$ if $e(\Kk) \leq N$.

(ii) If $\rk_p(\CH_{K_1}) = \rk_p(\CH_K)$, then $\rk_p(\CH_{K_n}) = \rk_p(\CH_K)$, 
for all $n \in [0, N]$.
\end{theorem}

\begin{remark}\label{nzero}
Properties (i), (ii) of stability in Theorem \ref{main2} may occur only from 
some layer $K_{n_0}$, giving $\order \CH_{K_{n_0+1}} = \order \CH_{K_{n_0}}$. 
Indeed, we will see that the minimal level $n_0$ (if any) may depend on the structure
of $\CH_K$ when the exponent and the $p$-rank of $\CH_K$ are large.
To get the statement of stability in the first layer, it suffices  to replace the base field $K$ by 
$K' = K_{n_0}$, $L/K$ by $L/K'$ and $\CH_K$ by $\CH_{K'}$; then, when $N':=N-n_0$ 
is large enough, the capitulation of $\CH_{K'}$ implies, a fortiori, that of $\CH_K$.
See a detailed example of this phenomenon in \S\,\ref{20887}.
\end{remark}

The proof of Theorem \ref{main1} will be given by Corollary \ref{maincoro} to 
Theorem \ref{mainthm}. Then Theorem \ref{main2} comes from our previous work
\cite[Theorem 3.1 \& Section 6, \S\,(b)]{Gras2022$^a$} 
generalizing \cite{Fuku1994,KrSch1995,Band2007,LOXZ2022,MiYa2021}; it
corresponds, in the statement of  Theorem \ref{main1},
to the case $m(\Kk_n) = 1$ (from $\CH_{K_n} = \CH_{K_n}^{G_n}$), 
whence $s(\Kk_n)=0$ and the condition $e(\Kk) \leq N$ (since $e(\Ll) = e(\Kk)$ from the 
isomorphisms $\CH_{K_n} \ds \mathop{\simeq}^{\hbox{\tiny$\Norm_{K_n/K}$}} \CH_K$ 
given by the arithmetic norms); this case is called the $p$-class 
groups stability in the tower $L = \bigcup_{n = 0}^\N K_n$. 

In practice, the knowledge of $m(\Ll)$ determines the unique $s(\Ll) \geq 0$ such
that $m(\Ll) \in [p^{s(\Ll)}, p^{s(\Ll)+1}-1]$ and one must check if $s(\Ll) \in [0, N-1]$;
once this holds, one must have 
$e(\Ll) \leq N-s(\Ll)$ to get the capitulation. See the definition of the filtration
in \S\,\ref{filtration} and its computation with PARI \cite{PARI} in 
\S\,\ref{progfiltr}.\footnote{\,PARI programs may be copy and past, 
from any pdf file. We only give excerpts of the very numerous 
numerical results which can be found again from running these programs. 
They excerpts give an overview suggesting the high frequency 
of capitulations in the simplest cyclic $p$-extensions $L/K$, $L \subset K(\mu_\ell^{})$,
$\ell \equiv 1 \pmod {2p^\N}$. Detailed examples may be found 
in Subsections \ref{firstexamples1}, \ref{mainexamples2}.}

The claim on the
$p$-ranks for the case of $\Z_p$-extensions was given by Fukuda \cite{Fuku1994}, 
then found again by Bandini \cite{Band2007}. It holds for any cyclic totally ramified 
$p$-extension as we have explained in \cite[\S\,6\,(b)]{Gras2022$^a$}.

Theorems \ref{main1}, \ref{main2} express that, if the ``complexity'' of $\CH_L$ is not too 
important, then, in a not very intuitive way, $\CH_K$ capitulates in $L$. Conversely, 
if one knows that capitulation is impossible (e.g., minus $p$-class groups in extensions 
of CM-fields or some invariants attached to abelian $p$-ramification theory), then the 
complexity of these invariants strictly increases with $n$ (see Theorem \ref{nocap}).

\subsection{The different aspects of capitulation and aims of the paper}\label{history}
The general problem of capitulation\,\footnote{\,I recently learned (from a Lemmermeyer 
text) that the word {\it capitulation} was coined by Arnold Scholz. It is possible that
this term may be received as incongruous; a solution is to consider that a non
principal ideal ${\mathfrak a}$ is a troublemaker with respect to elementary arithmetic,
in which case, the terminology is perfectly understandable.} of $\CH_K$ in~$L$  
has been studied in a very large number of publications, first in the purpose
of the factorization problem for Dedekind rings as exposed in Martin \cite{Mart2011}.
It is impossible to give a complete bibliography, but one may cite, among many 
other contributions, subsequent to the pioneering works of Hilbert--Scholz--Taussky
(see Kisilevsky \url{https://doi.org/10.2140/PJM.1997.181.219} for prehistory):
\cite{Kisil1970,Tera1971,Bond1981,HeSc1982,Schm1985,GrJa1985,Jaul1986,Jaul1988,
Iwas1989,Suzu1991,Maire1996,Gras1997,Maire1998,Kuri1999,Thie2000,GrWe2000,
KoMo2000,GrWe2003,JauMi2006,Gonz2006,Vali2008,Bosc2009,Bem2012,Mayer2014,
AZT2016,AZTM2016,Bisw2016,AZT2017,Jaul2019$^b$,Jaul2022}, in which the reader 
will find more history and references. 

Except some Iwasawa's theory results on capitulation \cite{Iwas1973,Iwas1989},
most of these papers are related to the Artin--Furtw\"angler theorem
and its generalizations on capitulation in the Hilbert class field $H_K^\nr$
(or to the Hilbert Theorem $94$ in cyclic sub-extensions of $H_K^\nr/K$ 
as given in Miyake \cite{Miya1989}, Suzuki \cite{Suzu1991}), 
which is not our purpose since, on the contrary, 
we will study {\it totally ramified cyclic $p$-extensions} $L/K$ and more precisely the 
simplest tamely ramified $p$-extensions $L \subset K(\mu_\ell^{})$, 
$\ell \equiv 1\! \pmod {2 p^\N}$ prime, $[L : K] = p^\N$, which, surprisingly, 
are often capitulation fields of $\CH_K$ when $K$ is totally real or for non totally 
imaginary base fields. 

Many classical articles give cohomological expressions of the capitulation in 
terms of global units as the fact that, in the non-ramified case, $\Ker(\J_{L/K})$ 
is isomorphic to a subgroup of $\Hom^1(G,\BE_L)$, where $\BE_L$ is the group 
of units of $L$ and $G = \Gal(L/K)$ (see, e.g., Jaulent \cite[Chap.\,III, \S\,1]{Jaul1986}, 
\cite{Jaul1988}, then Bembom \cite{Bem2012} for more comments and references). 

Similarly, using sets of places $S$, $T$ and tamely ramified Galois 
extensions, Maire \cite[Th\'eor\`eme 4.1]{Maire1996} describes injective maps
$\CH_{L,(\ell)}/\J_{L/K}(\CH_{K,(\ell)}) \hookrightarrow \Hom^2(G,\BE_{L,(\ell)})$,
in the context $L \subset K(\mu_\ell^{})$, where $\CH_{K,(\ell)}$, $\CH_{L,(\ell)}$ 
are the ray-class groups modulo $(\ell)$ and $\BE_{L,(\ell)}$ is the group of units 
congruent to $1$ modulo~$(\ell)$.

But the aspect ``global units'' 
is more difficult since the behavior of the unit groups in $L/K$ is less known compared
to that of $p$-class groups, even if there are some links; indeed, we have the following 
classical ``exact sequence of capitulation'' obtained from the map which associates, with 
the invariant class of the ideal ${\mathfrak a}$, a unit $\varepsilon := \Norm_{L/K}(\alpha)$ 
from the relation ${\mathfrak a}^{\sigma - 1} =: (\alpha)$, $\alpha \in L^\times$:
\begin{equation}\label{suite}
1 \to \J_{L/K}(\CH_K) \cdot \CH_L^\ram  \to \CH_L^{G} 
\to \BE_K\cap \Norm_{L/K}(L^\times)/\Norm_{L/K}(\BE_L) \to 1, 
\end{equation}

\noindent
where $\CH_{L}^\ram$ is generated by the classes of the 
ramified prime ideals of~$L$; in the right term, if $\BE_K\cap \Norm_{L/K}(L^\times)$
depends on easier local norm considerations, $\Norm_{L/K}(\BE_L)$ is in general unknown. 

On the contrary, $\J_{L/K}(\CH_K)$, $\CH_{L}^\ram$, are subgroups of $\CH_L^G$ 
and the order of this group is known from the Chevalley--Herbrand formula 
\cite[pp. 4002--405]{Chev1933}. For generalizations of this formula, see Gras 
\cite{Gras1978} for isotopic components of $\CH_L^G$ in the abelian semi-simple 
case, Jaulent \cite[Chapitre III, p.\,167]{Jaul1986} with ramification and decomposition, 
Lemmermeyer \cite{Lemm2013} in the spirit of the previous work of Jaulent and some 
papers of Gonz\'alez-Avil\'es \cite{Gonz2006}; then our general higher fixed points 
formulas \cite{Gras1973,Gras2017$^a$} or its idelic translation \cite{LiYu2020} by Li--Yu, allow 
the algorithmic computation of $\CH_L$ from its filtration (see Subsection \ref{CheHer}). 

In \cite{Gras2022$^a$}, giving extensive numerical PARI computations, we have 
proposed a conjecture, whose main consequence should be an obvious and immediate 
proof of the real abelian Main Conjecture 
``$\CH_{K,\varphi} = (\CE_{K,\varphi} : \CF_{\!K,\varphi} )$'' 
(where  $\CE := \BE \otimes \Z_p$ and $\CF := \BF \otimes \Z_p$)
in terms of indices of Leopoldt's cyclotomic units and $\varphi$-components using 
$p$-adic characters $\varphi$ of $K$, in the semi-simple case\,\footnote{The complete 
statement being the following \cite[Section 1.4, then Theorem 4.6]{Gras2022$^b$}: 
Assume that $K/\Q$ is a real cyclic extension, of prime-to-$p$ degree. 
Let $\ell \equiv 1 \pmod{2 p^\N}$ be a 
prime totally inert in $K/\Q$ and let $L \subset K(\mu_\ell^{})$ be the 
subfield of degree $p^\N$ over $K$. Then, if $\CH_K$ capitulates 
in $L$, the ``Main Conjecture'' on the $p$-adic components 
$\CH_{K,\varphi}$, of $\CH_K$, holds (i.e., 
$\order \CH_{K,\varphi} = (\CE_{K,\varphi} : \CF_{\!K,\varphi})$ for all 
irreducible $p$-adic character $\varphi$ of $K$).}.

We can improve the scope of this conjecture in various directions taking into 
account the new criterion of Theorem \ref{main1} using algebraic norms
(see Remark \ref{nzero} about the definition of the level $n_0$):

\begin{conjecture}\label{conjcap}
Let $K$ be any real number field and let $\CH_K$ be its $p$-class group,
of exponent $p^{e(\Kk)}$. For any prime number $\ell \equiv 1 \pmod {2 p^\N}$, $N \geq e(\Kk)$,
let $K_n$ be the sub-extension of $K(\mu_\ell^{})$, of degree $p^n$ over $K$, $n \in [0,N]$:

$\bullet$ There exist infinitely many $\ell$'s such that $\CH_K$ capitulates in $K(\mu_\ell^{})$.

$\bullet$ There exist infinitely many $\ell$'s such that the capitulation of $\CH_K$ in 
$K(\mu_\ell^{})$ is due, for some $n \in [1,N]$, to: $e(\Kk_n) \in[1, n - s(\Kk_n)]$ if 
$m(\Kk_n) \in [p^{s(\Kk_n)}, p^{s(\Kk_n)+1}-1]$ for $s(\Kk_n) \in [0, n-1]$. 

$\bullet$ For $N\gg 0$, the case 
$\order \CH_{K_{n_0+1}} = \order \CH_{K_{n_0}}$ of stability, for some $n_0 < N$,
occurs for infinitely many $\ell$'s.
\end{conjecture}

\begin{remark}
This restriction to the family of $p$-extensions $L/K$, $L \subset K(\mu_\ell^{})$,
is an\-other point of view with respect to the case of abelian capitulations obtained in Gras
\cite{Gras1997} (1997), Kurihara \cite{Kuri1999} (1999), Bosca \cite{Bosc2009} (2009), 
Jaulent \cite{Jaul2019$^b$,Jaul2022} (2019/22). Indeed, 
all techniques in these papers need to built a finite set of abelian $p$-extensions $L_k$ 
of $\Q$, ramified at various primes, requiring many local arithmetic conditions
existing from Chebotarev theorem, whose compositum with $K$ gives a capitulation 
field of $\CH_K$; the method must apply to any abelian field $K$ (of suitable signature), 
of arbitrary increasing degree, obtained in an iterative process giving, for instance, that 
the maximal real subfield of $\Q\big(\bigcup_{f>0} \mu_f^{} \big)$ is principal (see in
\cite{Bosc2009} the most general statements).
\end{remark}

\section{Complexity of \texorpdfstring{$\CH_L$}{Lg} versus capitulation of 
\texorpdfstring{$\CH_K$}{Lg}}\label{JrondN}

Let $L/K$ be a totally ramified cyclic $p$-extension of degree 
$p^\N$, $N \geq 1$, of Galois group $G =: \langle \sigma \rangle$. Let
$\Nu_{\!L/K} := \sum_{i=0}^{p^\N-1} \sigma^i$,  
be the algebraic norm in $L/K$. From the law of decomposition of
an unramified prime ideal ${\mathfrak q}$ of $K$, 
we get, for ${\mathfrak Q} \mid {\mathfrak q}$ in $L$ and the
decomposition group $D$ of ${\mathfrak Q}$ (of order $f$),
$({\mathfrak q})_L = \prod_{\ov \sigma \in G/D} {\mathfrak Q}^{\ov \sigma}$, thus
$\Nu_{\!L/K}({\mathfrak Q}) = \prod_{i=0}^{p^\N-1} {\mathfrak Q}^{\sigma^i}
= ({\mathfrak q})_L^f = ({\mathfrak q}^f)_L =( \Norm_{L/K}({\mathfrak Q}))_L$;  
whence the relation:
\begin{equation}\label{algebraicnorm}
\Nu_{\!L/K}(\CH_L) = \J_{L/K} \circ \Norm_{L/K}(\CH_L) = \J_{L/K}(\CH_K).
\end{equation} 

Thus, $\CH_K$ capitulates in $L$ if and only if $\Nu_{\!L/K}(\CH_L) = 1$. 
So, the action of the algebraic norm characterizes the capitulation (complete
or incomplete) and it is clear that the result mainly depends on the $\Z_p[G]$
structure of $\CH_L$ which is expressed by means of the canonical associated 
filtration $\{\CH_L^i\}_{i \geq 0}$ that we are going to recall, from \cite{Gras2017$^a$}, 
improved english translation of  \url{https://doi.org/10.2969/jmsj/04630467}.

\subsection{Filtration of \texorpdfstring{$\CH_L$}{Lg} in the totally ramified case}
Let $L/K$ be a cyclic $p$-extension  of degree $p^\N$, $N \geq 1$, and Galois group 
$G = \langle \sigma \rangle$.
To avoid technical writings, assume that any prime ideal ${\mathfrak l}$ of $K$, 
ramified in $L/K$, is totally ramified, and that there are $r \geq 1$ 
such prime ideals. We do not assume $L/\Q$ Galois.

\subsubsection{Filtration and higher Chevalley--Herbrand formulas} \label{CheHer}
The generalizations of the Chevalley--Herbrand formula is based on the corresponding 
filtration $\{\CH_L^i\}_{i \geq 0}$ defined as follows:
\begin{equation*}
\CH_L^0 =1,\ \ \ \  \CH_L^1  := \CH_L^G, \ \ \ \ 
\CH_L^{i+1}/\CH_L^i := (\CH_L/\CH_L^i)^G,\  i \geq 0,
\end{equation*}

\noindent
up to $i=m(\Ll) := \min \{m \geq 0,\ \CH_L^{ (\sigma - 1)^m} = 1\}$, 
for which $\CH_L^{m(\Ll)} = \CH_L$. 

Denote by $\CI_L^i$ a $\Z[G]$-module 
of ideals of $L$, of finite type, generating $\CH_L^i$, with $\CI_L^0 = 1$,
$\CI_L^{i+1} \supseteq \CI_L^i,\, \forall i \geq 0$; $\CI_L^i$ is defined 
up to the group of principal ideals of $L$, thus $\Norm_{L/K}(\CI_L^i)$ is defined
up to $\big(\Norm_{L/K}(L^\times) \big)$. Note that the above invariants are constant
for $i \geq m(\Ll)$ and that $m(\Kk) = 1$ for $\CH_K \ne 1$ (otherwise $m(\Kk) = 0$).

This filtration has, for all $i \geq 0$, the following properties \cite[Theorem 3.6]{Gras2017$^a$}:
\begin{equation}\label{filtration}
\left\{ \begin{aligned}
&(i) \ \ \ \order\CH_L^1 = \order \CH_K \times 
\frac{p^{N (r-1)}}{(\BE_K : \BE_K \cap \Norm_{L/K}(L^\times))},\\
&(ii) \ \  \order (\CH_L^{i+1}/\CH_L^i) =  \frac{\order \CH_K}{\order \Norm_{L/K}(\CH_L^i)}
\!\times \!\frac{p^{N (r-1)}}{(\Lbda_K^i : \Lbda_K^i \cap \Norm_{L/K}(L^\times))}, \\
&\hspace{3.5cm} \Lbda_K^i := \{x \in K^\times,\ (x) \in \Norm_{L/K}(\CI_L^i) \} ,  \\
&(iii) \  \CH_L^i = \{h \in \CH_L,\ h^{(\sigma-1)^i} = 1\}, \\ 
&(iv) \ \ \order (\CH_L^{i+1}/\CH_L^i) \leq \order \CH_L^1, \\
&(v) \ \ \  \order \CH_L = \hbox{$\prod_{i=0}^{m(\Ll)-1}$} \order (\CH_L^{i+1}/\CH_L^i)
\leq (\order \CH_L^1)^{m(\Ll)}.
\end{aligned} \right.
\end{equation}

\noindent
The $\Lbda_K^i$'s are subgroups of $K^\times$ containing $\BE_K$, with 
$\Lbda_K^0 = \BE_K$ in (i). In particular, any $x \in \Lbda_K^i$ is local norm
in $L/K$ at all the non-ramified places. So, for any $(x) = \Norm_{L/K}({\mathfrak A})$,
${\mathfrak A} \in \CI_L^i$, which is also local norm at the ramified places, then
$x = \Norm_{L/K}(y)$, $y \in L^\times$ (Hasse's norm theorem) and there 
exists an ideal ${\mathfrak B}$ of $L$ such that 
${\mathfrak A} = (y) {\mathfrak B}^{\sigma-1}$; this constitutes an algorithm by
addition of the ${\mathfrak B}$'s to $\CI_L^i$ to get $\CI_L^{i+1}$ then 
$\CH_L^{i+1}$.\,\footnote{\,For explicit class field theory, Hasse's norm theorem, 
norm residue symbols, product formula, see, e.g., \cite[Theorem II.6.2, 
Definition II.3.1.2, Theorems II.3.1.3, 3.4.1]{Gras2005}.}
Since $\Lbda_K^0 = \BE_K$ is a $\Z$-module of finite type, this algorithm allows
to construct $\Lbda_K^i$ of finite type for all $i$, with $\Lbda_K^i \subseteq \Lbda_K^{i+1}$
(indeed, $\Norm_{L/K}(\CI_L^i)$ is of finite type, there is a finite number of relations of
principality between the generators and $\Lbda_K^i/\Lbda_K^i \cap \Norm_{L/K}(L^\times)$
is annihilated by $p^\N$).

The $i$-sequence $\order (\CH_L^{i+1}/\CH_L^i)$ is 
decreasing, from $\order \CH_L^1$ up to $1$, because of the injective maps
$\CH_L^{i+1}/\CH_L^i \hookrightarrow \CH_L^i/\CH_L^{i-1}
\hookrightarrow \cdots \hookrightarrow \CH_L^1$,
due to the action of $\sigma - 1$, giving the inequality in (v).

The first (resp. second) factor in (ii) is called the class (resp. norm) factor.

\subsubsection{Properties of the class and norm factors}
Since ramified places $v$ of $K$ are assumed to be totally ramified 
in $L/K$, their inertia groups $I_v(L/K)$ in $L/K$ are isomorphic to $G$. Let $\omega_{L/K}$ 
be the map which associates with $x \in \Lbda_K^i$ the family of Hasse's norm symbols 
$\big( \frac{x \, ,\, L/K}{v}\big) \in I_v(L/K)$. Since $x$ is local norm at 
the unramified places, $\omega_{L/K}(\Lbda_K^i)$ is contained (product formula) in:
$$\Omega_{L/K} := \big \{ (\tau_v)_{v} \in \hbox{$\bigoplus_{v}$} I_v(L/K), 
\ \, \hbox{$\prod_{v}$} \tau_v= 1 \big\} \simeq G^{r-1}; $$ 
then $\order\omega_{L/K}(\Lbda_K^i) = 
(\Lbda_K^i : \Lbda_K^i \cap \Norm_{L/K}(L^\times))$ divides $p^{N (r-1)}$. 

Let $K_n \subseteq L$ be of degree $p^n$ over $K$ and set $G_n := \Gal(K_n/K) 
=: \langle \sigma_n \rangle$, $n \in [0, N]$.
All of the above applies to the $K_n$'s;
denote by $\Lbda_K^i(n) \subset K^\times$ the invariants corresponding 
to $K_n/K$ instead of $L/K$; so $\Lbda_K^i(n) =
\{x \in K^\times,\ \, (x) \in \Norm_{K_n/K}(\CI_{K_n}^i) \}$, where 
$\CI_{K_n}^i$ represents $\CH_{K_n}^i$; so $\Lbda_K^0(n) = \BE_K$.
For $n=0$ and $i \geq 1$, $\CI_{\!K}^i$ generates $\CH_K$
and $\Lbda_K^i(0) = \{x \in K^\times, \ (x) \in \CI_{K}^i \}$
contains $\Lbda_K^0(0) = \BE_K$ and is given by relations
between elements of $\CI_{\!K}^i$; we have $\CI_K^0 = 1$.

\begin{lemma} \label{inclusion} 
For any $i$ fixed we may assume  $\Lbda_K^i(n+1) \subseteq 
\Lbda_K^i(n),\, \forall n \in [0, N-1]$.
\end{lemma}

\begin{proof}
We have $\Norm_{K_{n+1}/K_n} (\CH_{K_{n+1}}^i)  \subseteq \CH_{K_n}^i$; so, for
any ideal ${\mathfrak A}_{n+1} \in \CI_{\!K_{n+1}}^i$, one may write
$\Norm_{K_{n+1}/K_n}({\mathfrak A}_{n+1}) = (\alpha_n) \, {\mathfrak A}_n$, where
$\alpha_n \in K_n^\times$ and ${\mathfrak A}_n \in \CI_{\!K_n}^i$, in which case
{\it modifying} the definition of $\CI_{\!K_n}^i$ modulo principal ideals of $K_n$, 
one may assume $\Norm_{K_{n+1}/K_n} (\CI_{\!K_{n+1}}^i ) \subseteq  
\CI_{\!K_n}^i$ whence $\Norm_{K_{n+1}/K} (\CI_{\!K_{n+1}}^i ) 
\subseteq \Norm_{K_{n}/K} (\CI_{\!K_n}^i)$; this modifies $\Lbda_K^i(n)$ 
modulo $\Norm_{K_n/K}(K_n^\times)$ which does not modify
$\order \omega_{K_n/K}(\Lbda_K^i(n))$.
Using the process from the top, we obtain
$\Lbda_K^i(N) \subseteq \Lbda_K^i(N-1) \subseteq \cdots \subseteq 
\Lbda_K^i(1) \subseteq \Lbda_K^i(0)$. For $i=0$, $\Lbda_K^0(n) = \BE_K$,
for all $n \geq 0$
\end{proof}

\begin{lemma}\label{increasing}
For $i \geq 0$ fixed, the integers $\order  \big( \CH_{K_n}^{i+1} / \CH_{K_n}^i \big)$ 
define an increasing $n$-sequence from $\order \big( \CH_K^{i+1} / \CH_K^i \big)=1$ 
up to $\order \big( \CH_L^{i+1} / \CH_L^i \big)$; the $\order  \CH_{K_n}^i$'s define an 
increasing $n$-sequence from $\order \CH_K^i = \CH_K$ up to $\order  \CH_L^i$.
\end{lemma}

\begin{proof} Consider, for $i \geq 1$ fixed and $n \geq 0$, the two factors of the 
$n$-sequence:
$$\order  \big( \CH_{K_n}^{i+1} / \CH_{K_n}^i \big)=
\ffrac{\order  \CH_{K}}{\order \Norm_{K_n/K}( \CH_{K_n}^i) } \times \ffrac{p^{n (r -1)} }
{\order \omega_{K_n/K}(\Lbda_K^i(n))}. $$
As $\Norm_{K_{n+1}/K} (\CH_{K_{n+1}}^i) \subseteq \Norm_{K_{n}/K} (\CH_{K_n}^i)$,
$\ds p^{c_{K_n}^i} := \frac{\order \CH_{K}}{\order \Norm_{K_n/K}( \CH_{K_n}^i)}$ defines
an increasing $n$-sequence from $1$ up to $p^{c_L^i} \mid \order \CH_K$.
The factor
$p^{\rho_{K_n}^i} := \ffrac{p^{n (r -1)}}{\order \omega_{K_n/K} (\Lbda_K^i(n))}$
defines an increasing $n$-sequence from $1$ up to $p^{\rho_L^i}$ since,
from Lemma \ref{inclusion}:
$$p^{\rho_{K_{n+1}}^i \!\!- \rho_{K_n}^i}\! =\! p^{r-1} 
\frac{\order \omega_{K_n/K}(\Lbda_K^i(n))}
{\order \omega_{K_{n+1}/K}(\Lbda_K^i(n+1))} \geq p^{r-1} 
\frac{\order \omega_{K_n/K}(\Lbda_K^i (n))}
{\order \omega_{K_{n+1}/K}(\Lbda_K^i (n))} ; $$ 

\noindent
so, in the restriction 
$\Omega_{K_{n+1}/K} \too \hspace{-0.52cm} \too \Omega_{K_{n}/K}$
(whose kernel is of order $p^{r-1}$ due to the total ramification of each place),
the image of $\omega_{K_{n+1}/K}(\Lbda_K^i (n))$ is 
$\omega_{K_n/K}(\Lbda_K^i (n))$ because of the properties 
of Hasse's symbols, whence $p^{\rho_{K_{n+1}}^i \!\!- \rho_{K_n}^i} \geq 1$ 
and the result for the $n$-sequence 
$p^{\rho_{K_n}^i}$, with maximal value $p^{\rho_L^i}$. 
For $i=0$, $\CH_{K_n}^0=1$ for all $n \geq 0$.
The first claim of the lemma
holds for the $n$-sequence $\order \big( \CH_{K_n}^{i+1} / \CH_{K_n}^i \big)$; 
for $n=N$, one gets the formula $\order  \big( \CH_L^{i+1} / \CH_L^i \big) = 
p^{c_L^i} \!\cdot p^{\rho_L^i}$.

Assuming, by induction on  $i \geq 0$, that the $n$-sequence $\order \CH_{K_n}^i$ is
increasing, the property follows for the $n$-sequence $\order \CH_{K_n}^{i+1}$.
\end{proof}

\begin{remark}
The $n$-sequence $m(\Kk_n)$ is an 
increasing sequence from $1$ (if $\CH_K \ne 1$, $0$ otherwise) up to $m(\Ll)$.
The $\order \CH_{K_n}$'s
define an increasing $n$-sequence from $\order \CH_K$ up to $\order \CH_L$ since
$\order \CH_{K_{n+1}} \geq \order \CH_{K_{n+1}}^{\Gal(K_{n+1}/K_n)} =
\order \CH_{K_n} \, \ffrac{p^{r-1}}{\omega_{K_{n+1}/K_n}(\BE_{K_n})}
\geq \order \CH_{K_n}$. 
The integers $e(\Kk_n)$ and $\rk_p(\CH_{K_n})$ define increasing $n$-sequences.
\end{remark}

The interest of this filtration is that standard probabilities may apply at each 
level $n$ to the algorithm computing $\CH_{K_n}^{i+1}$ from $\CH_{K_n}^i$ by 
means of the factors $p^{c_{K_n}^i}$ and $p^{\rho_{K_n}^i}$, giving plausible 
heuristics in the spirit of the works of Koymans-Pagano \cite{KoPa2022}, Smith \cite{Smith2022},
leading to a considerable generalization of pioneering works as that of Morton \cite{Mort1982}, 
Gerth III \cite{Gerth1986} and many others, the theory over $\Q$ giving generalizations of
the well-known R\'edei matrices. 
Indeed, let $x \in \Lbda_K^i(n)$ be such that $(x) = \Norm_{K_n/K}({\mathfrak A})$ 
(when it holds for ${\mathfrak A} \in \CI_{K_n}^i$), and let $x = \Norm_{K_n/K}(y)$, when 
$\omega_{K_n/K}(x)=1$ (depen\-ding on Hasse's symbols), so that ${\mathfrak A} 
= (y) {\mathfrak B}^{\sigma-1}$ giving ${\mathfrak B} \in \CI_{K_n}^{i+1}$.
So we propose the following conjecture on the algorithmic evolution of the class and norm 
factors, respectively:

\begin{conjecture}\label{conjprobas}
Let $L/K$ be any cyclic $p$-extension of degree $p^\N$, $N \geq 1$, of Galois group $G$; 
we assume that $L/K$ is ramified at $r \geq 1$ places of $K$, totally ramified in $L/K$.
Then, the orders of each of the two factors (class and norm) in the $i$-sequence 
$\order (\CH_L^{i+1}/\CH_L^i)$, follow binomial laws as $i$ increases, based on the
following probabilities:

$\bullet$ Let $c \in \CH_K$; the probability, for an ideal ${\mathfrak C}$ 
of $L$, that the $p$-class of the ideal $\Norm_{L/K}({\mathfrak C})$ equals $c$, 
is $\ffrac{1}{\order \CH_K}$.

$\bullet$ Let $\gamma \in  G^{r-1}$; the probability, for $x \in K^\times$
local norm at all the non-ramified places, that $\omega_{L/K}(x) = \gamma$, 
is $\ffrac{1}{p^{N(r-1)}}$.
\end{conjecture}

\subsubsection{Program computing the filtrations 
\texorpdfstring{$\{\CH_{K_n}^i\}_{i \geq 1}$}{Lg}}\label{progfiltr}

The following program may be used for the calculation of the Galois structure of the 
$\CH_{K_n}$'s, $K_n \subseteq L \subset K(\mu_\ell^{})$, {\it whatever the base field 
$K$ of prime-to-$p$ degree $d$}, given, as usual, by means of a monic polynomial of 
$\Z[x]$ (the condition $p \nmid d$ simplifies the computation of a generator $\sigma_n$ 
(in ${\sf S}$) of $\Gal(K_n/K)$). Let $r(\Kk_n) := \rk_\Z(\BH_{K_n})$
(in ${\sf rKn}$).

For this, one must indicate the prime $p$ in ${\sf p}$, the number ${\sf Nn}$ of layers 
$K_n$ considered, the polynomial ${\sf PK}$ defining $K$, a prime ${\sf ell}$ congruent 
to $1$ modulo $2p^\N$, ${\sf N \geq Nn}$, and a value ${\sf mKn}$ for computing the 
${\sf h_j^i := h_j^{(S - 1)^i}}$ for ${\sf 1 \leq i \leq mKn}$ and ${\sf 1 \leq j \leq rKn}$, 
where the ${\sf h_j}$'s are the $r(\Kk_n)$ generators of the whole class group $\BH_{K_n}$ 
given by PARI (in ${\sf CKn=Kn.clgp}$), and where ${\sf S}$ is chosen of order $p^n$ in 
${\sf G=nfgaloisconj(Kn)}$ by testing the orders.\,\footnote{\, Warning: 
for some class group computations, PARI uses random primes in some analytic 
contexts, so that the generators given by ${\sf Kn.clgp}$ may vary; but the 
corresponding matrices of exponents are ``equivalent''. For this observation, 
run the programs several times.}

So $\CH_{K_n}^i = \{h \in \CH_{K_n},\  h^{(\sigma_n-1)^i} = 1)\}$, $1 \leq i \leq m(\Kk_n)$
(see \eqref{filtration}\,(iii)). PARI works with independent generators ${\sf h_j}$ 
of $\BH_{K_n}$, of orders ${\sf e_j}$ (given in ${\sf CKn[2]}$); thus, for any
data ${\sf [e_1,\ldots,e_{rKn}]}$ given by ${\sf bnfisprincipal(Kn,Y)[1]}$
for an ideal ${\sf Y}$ whose class is ${\sf h= \prod_j h_j^{e_j}}$,
the program gives, instead,
the list ${\sf E:= [\ov e_1,\ldots, \ov e_{rKn}]}$ defining the $p$-class of
${\sf Y}$ (in $\CH_{K_n}$) from ${\sf \ov h = \prod_j \ov h_j^{\ov e_j}}$,
${\sf \ov e_j = lift(Mod(e_j,p^{n_j}))}$, where ${\sf p^{n_j}}$ is the 
$p$-part of the order of ${\sf h_j}$; this does not modify the Galois structure 
of $\CH_{K_n}$ and the outputs are more readable.
So, the ideal ${\sf Y}$ is $p$-principal if and only if ${\sf E = [0, \ldots , 0]}$.
The outputs are written under the form ${\sf \ov h_j^{(\sigma-1)^i} 
= [\ov e_1,\ldots, \ov e_{rKn}]}$ instead of ${\sf \ov h_j^{(\sigma-1)^i} = 
\ov h_1^{\ov e_1} \cdots \ov h_{rKn}^{\ov e_{rKn}}}$.

The invariant $m(\Kk_n)$ is obtained (for ${\sf mKn}$ large enough) 
for the first $i$ giving zero matrices in the test of principality of the ${\sf h_j^i}$'s. 

Below, we take as example the cyclic cubic field of conductor ${\sf f=703}$, 
for which, using the structure of $\BZ$-module, $\BZ=\Z[\exp(\frac{2 i \pi}{3})]$,
we have $\CH_K \simeq \BZ/2\BZ$; taking ${\sf ell=97}$, ${\sf mKn=3}$,
the results are given for ${\sf n=1}$ and ${\sf n=2}$ and ${\sf r}$ is the number 
of prime ideals of $K$ dividing $\ell$:

\ft\begin{verbatim}
PROGRAM COMPUTING THE h_j^[(S-1)^i]:
{p=2;Nn=2;PK=x^3+x^2-234*x-729;ell=97;mKn=3;K=bnfinit(PK,1);CK0=K.clgp;
r=matsize(idealfactor(K,ell))[1];print("p=",p," Nn=",Nn," PK=",PK,
" ell=",ell," mKn=",mKn," CK0=",CK0[2]," r=",r);
for(n=1,Nn,Qn=polsubcyclo(ell,p^n);Pn=polcompositum(PK,Qn)[1];
Kn=bnfinit(Pn,1);CKn=Kn.clgp;dn=poldegree(Pn);
print("CK",n,"=",CKn[2]);rKn=matsize(CKn[2])[2];
\\Search of a generator S of Gal(Kn/K):
G=nfgaloisconj(Kn);Id=x;for(k=1,dn,Z=G[k];ks=1;while(Z!=Id,
Z=nfgaloisapply(Kn,G[k],Z);ks=ks+1);if(ks==p^n,S=G[k];break));
\\Computation of the image of CKn by (S-1)^i:
for(j=1,rKn,X=CKn[3][j];Y=X;for(i=1,mKn,YS=nfgaloisapply(Kn,S,Y);
T=idealpow(Kn,Y,-1);Y=idealmul(Kn,YS,T);B=bnfisprincipal(Kn,Y)[1];
\\computation in Ehij of the modified exponents of B:
Ehij=List;for(j=1,rKn,c=B[j];w=valuation(CKn[2][j],p);c=lift(Mod(c,p^w)); 
listput(Ehij,c,j));print("h_",j,"^[","(S-1)^",i,"]=",Ehij));print()))}
\end{verbatim}\ns
\ft\begin{verbatim}
p=2  Nn=2  f=703  PK=x^3+x^2-234*x-729  ell=97  mKn=3  CK0=[6,2]  r=1
CK1=[6,2,2,2]=[2,2,2,2]
h_1^[(S-1)^1]=[1,1,0,0]       h_2^[(S-1)^1]=[1,1,0,0]
h_3^[(S-1)^1]=[0,0,1,1]       h_4^[(S-1)^1]=[0,0,1,1]
h_1^[(S-1)^2]=[0,0,0,0]       h_2^[(S-1)^2]=[0,0,0,0]
h_3^[(S-1)^2]=[0,0,0,0]       h_4^[(S-1)^2]=[0,0,0,0]
h_1^[(S-1)^3]=[0,0,0,0]       h_2^[(S-1)^3]=[0,0,0,0]
h_3^[(S-1)^3]=[0,0,0,0]       h_4^[(S-1)^3]=[0,0,0,0]
CK2=[12,4,2,2]=[4,4,2,2]
h_1^[(S-1)^1]=[0,2,1,1]       h_2^[(S-1)^1]=[0,2,1,0]
h_3^[(S-1)^1]=[2,2,0,0]       h_4^[(S-1)^1]=[2,0,0,0]
h_1^[(S-1)^2]=[0,2,0,0]       h_2^[(S-1)^2]=[2,2,0,0]
h_3^[(S-1)^2]=[0,0,0,0]       h_4^[(S-1)^2]=[0,0,0,0]
h_1^[(S-1)^3]=[0,0,0,0]       h_2^[(S-1)^3]=[0,0,0,0]
h_3^[(S-1)^3]=[0,0,0,0]       h_4^[(S-1)^3]=[0,0,0,0]
CK3=[12,4,2,2]=[4,4,2,2]
\end{verbatim}\ns

This gives $m(\Kk_1)=2$, $\CH_{K_1}^{\sigma_1-1} = \langle h_1h_2, h_3h_4 \rangle$,
$\CH_{K_1}^{G_1} = \langle h_1h_2^{-1}, h_3h_4^{-1} \rangle \simeq \BZ/2\BZ$.
Then $m(\Kk_2)=3$, $\CH_{K_2}^{\sigma_2 - 1} = 
\langle h_2^2h_3h_4, h_2^2 h_3,h_1^2h_2^2, h_1^2 \rangle$
$= \langle h_1^2, h_2^2 , h_3,  h_4\rangle \simeq (\BZ/2\BZ)^2$, 
whence $\CH_{K_2}^{G_2} =(\CH_{K_2})^2 \simeq \BZ/2\BZ$. Then
$\CH_{K_2}^{(\sigma_2 - 1)^2} = \langle h_1^2, h_2^2 \rangle
\simeq (\CH_{K_2})^2 \simeq \BZ/2\BZ$, $\CH_{K_2}^2
= \langle  h_1^2, h_2^2, h_3,  h_4 \rangle \simeq (\BZ/2\BZ)^2$.
These computations will be pursued to obtain a partial capitulation 
in $K_1$ and a complete capitulation in $K_2$ (stability from $K_2$).

\subsection{Classes, units and capitulation kernel}\label{firstexamples1}
In this subsection, we restrict ourselves to the case $L \subseteq K(\mu_\ell^{})$, 
$\ell \equiv 1 \pmod {2 p^\N}$, with $K$ real of prime-to-$p$ degree $d$, so 
that $L/K$ is totally ramified at all the $r$ prime ideals ${\mathfrak l} \mid \ell$ of $K$. 
Chevalley--Herbrand's formula $\order\CH_L^G = \order \CH_K \times \ffrac{p^{N (r-1)}}
{\order \big (\BE_K/\BE_K \cap \Norm_{L/K}(L^\times)\big)}$ and exact sequence of 
capitulation \eqref{suite} lead to the relation:
\begin{equation}\label{general}
\order \big( \J_{L/K}(\CH_K) \, \CH_L^\ram \big) \times
\order \big(\BE_K /\Norm_{L/K}(\BE_L) \big) = \order \CH_K \times p^{N \,(r-1)}.
\end{equation}

Since $L_0 \subset \Q(\mu_\ell^{})$ is $p$-principal (indeed, 
$\CH_{L_0}^{\Gal(L_0/\Q)} = 1$) and since $p \nmid d$, the primes 
${\mathfrak L}_i \mid {\mathfrak l}_i$ of $L$, $i=1,\ldots , r$, fulfill a  
relationship of the form ${\mathfrak L}_1 \cdots {\mathfrak L}_r = (\alpha_0)_L^{}$, 
$\alpha_0^{} \in L_0^\times$, so that $\rk_p(\CH_L^\ram) \leq r-1$. 
Then $(\CH_L^\ram)^{p^\N} = \J_{L/K}(\CH_K^\ram)$,
$\Norm_{L/K}(\CH_L^\ram) = \CH_K^\ram$, where $\CH_K^\ram 
\subseteq \CH_K$ is generated by the $p$-classes of the ${\mathfrak l}_i$'s.
One verifies easily that, in $L = \bigcup_{n=0}^\N K_n$, 
$$\hbox{$\order \CH_{K_n}^\ram$, \ \ \ 
$\order \big(\BE_K /\Norm_{K_n/K}(\BE_{K_n}) \big), \ \ \ \ffrac{p^{N (r-1)}}
{\order \big (\BE_K/\BE_K \cap \Norm_{L/K}(L^\times)\big)}$} $$ 

\noindent
(see Lemma \ref{increasing} for the last one) define increasing $n$-sequences and 
that  $\order \J_{K_n/K}(\CH_{K})$ is decreasing. This suggests that for some 
$\ell$'s, with $N \gg 0$, there is capitulation of $\CH_K$ in $K_{n_0}$ 
and, from \eqref{general}, relations of the form:
$$\order \CH_{K_n}^\ram = p^{a\, n + a_0}, \ \ \ 
\order \big(\BE_K /\Norm_{K_n/K}(\BE_{K_n}) \big) = p^{b \, n + b_0},\ \, \forall \, n \geq n_0, $$
with $a + b = r-1$ and $a_0 + b_0 = v_p(\order \CH_K)$.

\subsubsection{Case $r=1$}
The case $r=1$ is particular since, whatever $n$, all the factors are 
finite in the relation \eqref{general} which becomes:
$$\order \big( \J_{L/K}(\CH_K) \, \CH_L^\ram \big) \times
\order \big(\BE_K /\Norm_{L/K}(\BE_L) \big) = \order \CH_K, $$
giving, possibly, stationary $n$-sequences from some layer $K_{n_0}$ up to $L$.
The case $r=1$ supposes that $\ell$ does not split in $K$ and we can assume, 
for instance, that $K/\Q$ is cyclic of prime-to-$p$ degree, in which case
$\CH_L^\ram = 1$, whence:
$$\order \J_{L/K}(\CH_K) \times \order \big(\BE_K /\Norm_{L/K}(\BE_L) \big)
= \order \CH_K, $$
so that complete capitulation in $L$ is equivalent to:
\begin{equation}\label{equalorders}
 \order \big(\BE_K /\Norm_{L/K}(\BE_L) \big) = \order \CH_K,
\end{equation}
with stationary $n$-sequences from some layer $K_{n_0}$.

Of course, only orders coincide in \eqref{equalorders} since structures may be very 
different; we will consider two examples of this phenomenon. We then compute
$\Nu_{K_n/K}(\CH_{K_n})$ by means of the $\Nu_{K_n/K}(h_j)$, $1 \leq j \leq r(K_n)$.

\begin{remark}
An remarkable fact, in a diophantine viewpoint, is that, when the class
of ${\mathfrak a}$ capitulates in some $K_n$, the writing of the generator
$\alpha \in K_n^\times$, of the extended ideal $({\mathfrak a})_{K_n}$, 
on the $\Q$-basis ${\sf Kn.zk}$ of the field $K_n$, needs most often 
oversized coefficients, and increasing with $n$ (several thousand digits 
and, often, PARI proves the principality without giving these coefficients).
If the reader wishes to verify this fact, it suffices to add the instruction
${\sf print(bnfisprincipal(Kn,Y))}$ giving the whole data for the ideal ${\sf Y}$
considered.
\end{remark}

\begin{example}
{\rm We consider the cubic field of conductor $31923$, with $p=2$, $\ell = 257$ ($N = 7$).

\ft\begin{verbatim}
p=2  f=31923  PK=x^3-10641*x+227008  CK0=[6,2,2,2]  ell=257  r=1
CK1=[18,6,2,2,2,2]=[2,2,2,2,2,2]
h_1^[(S-1)^1]=[0,0,0,0,0,0]   h_2^[(S-1)^1]=[0,0,0,0,0,0]  
h_3^[(S-1)^1]=[0,1,0,0,1,0]   h_4^[(S-1)^1]=[0,0,0,0,1,1]  
h_5^[(S-1)^1]=[0,0,0,0,0,0]   h_6^[(S-1)^1]=[0,0,0,0,0,0]
h_1^[(S-1)^2]=[0,0,0,0,0,0]   h_2^[(S-1)^2]=[0,0,0,0,0,0]  
h_3^[(S-1)^2]=[0,0,0,0,0,0]   h_4^[(S-1)^2]=[0,0,0,0,0,0]  
h_5^[(S-1)^2]=[0,0,0,0,0,0]   h_6^[(S-1)^2]=[0,0,0,0,0,0]
norm in K1/K of the component 1 of CK1:[0,0,0,0,0,0]
norm in K1/K of the component 2 of CK1:[0,0,0,0,0,0]
norm in K1/K of the component 3 of CK1:[0,1,0,0,1,0]
norm in K1/K of the component 4 of CK1:[0,0,0,0,1,1]
norm in K1/K of the component 5 of CK1:[0,0,0,0,0,0]
norm in K1/K of the component 6 of CK1:[0,0,0,0,0,0]
Incomplete capitulation, m(K1)=2, e(K1)=1
CK2=[36,12,2,2,2,2]=[4,4,2,2,2,2]
h_1^[(S-1)^1]=[2,2,0,0,0,0]   h_2^[(S-1)^1]=[2,2,1,1,1,0]
h_3^[(S-1)^1]=[0,2,1,1,1,0]   h_4^[(S-1)^1]=[0,2,0,0,1,1]  
h_5^[(S-1)^1]=[0,2,1,1,0,1]   h_6^[(S-1)^1]=[2,0,1,1,0,1]
h_1^[(S-1)^2]=[0,0,0,0,0,0]   h_2^[(S-1)^2]=[0,2,0,0,0,0]
h_3^[(S-1)^2]=[0,2,0,0,0,0]   h_4^[(S-1)^2]=[2,2,0,0,0,0]
h_5^[(S-1)^2]=[2,0,0,0,0,0]   h_6^[(S-1)^2]=[2,0,0,0,0,0]
norm in K2/K of the component 1 of CK2:[0,0,0,0,0,0]
norm in K2/K of the component 2 of CK2:[0,0,0,0,0,0]
norm in K2/K of the component 3 of CK2:[0,0,0,0,0,0]
norm in K2/K of the component 4 of CK2:[0,0,0,0,0,0]
norm in K2/K of the component 5 of CK2:[0,0,0,0,0,0]
norm in K2/K of the component 6 of CK2:[0,0,0,0,0,0]
Complete capitulation, m(K2)=3, e(K2)=2
\end{verbatim}\ns

The data for $K_2$ give $\CH_{K_2}^{G_2} = \langle h_1h_5h_6, \,h_2h_4 h_6 \rangle 
\simeq \BZ/4\BZ$, since the two independent generators $h_1$, $h_2$, are of order $4$.
The exact sequence \eqref{suite} reduces to the isomorphism of $\BZ$-modules 
$\CH_{K_2}^{G_2} \simeq \BE_K /\Norm_{K_2/K}(\BE_{K_2})$, which 
are of order $\order \CH_K = 16$, but are not isomorphic to $\CH_K$.
We have, from the relation \eqref{equalorders} and since $\BE_K \simeq \BZ$, 
the isomorphisms of $\BZ$-modules
$\CH_{K_2}^{G_2} \simeq \BE_K /\Norm_{K_2/K}(\BE_{K_2})\simeq  \BZ/4\BZ$
and $\CH_K \simeq \BZ/2\BZ \times \BZ/2\BZ$.}
\end{example}

\begin{example}
{\rm We consider a quadratic field with $p=3$, $\ell = 19$ inert in $K$ ($N=2$).
Let $K=\Q(\sqrt {32009})$ for which $\CH_K \simeq \Z/3\Z \times \Z/3\Z$. 
The general Program gives an incomplete capitulation in $K_1$, then a complete 
capitulation in $K_2$:

\ft\begin{verbatim}
PK=x^2-32009  CK0=[3,3]  ell=19  r=1
CK1=[9,3]
h_1^[(S-1)^1]=[3,0]        h_2^[(S-1)^1]=[3,0]
h_1^[(S-1)^2]=[0,0]        h_2^[(S-1)^2]=[0,0]
norm in K1/K of the component 1 of CK1:[3,0]
norm in K1/K of the component 2 of CK1:[0,0]
Incomplete capitulation, m(K1)=2, e(K1)=2
CK2=[9,3]
h_1^[(S-1)^1]=[0,0]        h_2^[(S-1)^1]=[3,0]
h_1^[(S-1)^2]=[0,0]        h_2^[(S-1)^2]=[0,0]
norm in K2/K of the component 1 of CK2:[0,0]
norm in K2/K of the component 2 of CK2:[0,0]
Complete capitulation, m(K2)=2, e(K2)=2
\end{verbatim}\ns

The data for $K_2$ give $\CH_{K_2}^{G_2} = \langle h_1 \rangle \simeq \Z/9\Z$. 
Since $\J_{K_2/K}(\CH_K)=1$, $\CH_{K_2}^\ram = 1$ and $\BE_K \simeq \Z$,
the relations \eqref{general} \eqref{equalorders} become, in $K_2/K$, the 
isomorphism {of cyclic groups}
$\CH_{K_2}^{G_2} \simeq \BE_K/ \Norm_{K_2/K}(\BE_{K_2})  \simeq \Z/9\Z$,
while $\CH_K \simeq  \Z/3\Z \times \Z/3\Z$.}
\end{example}

\subsubsection{Case $r > 1$}
In this case, an heuristic is that there is no obstruction
about $\CH_{K_n}^\ram$, $\BE_K /\Norm_{K_n/K}(\BE_{K_n})$ as 
$\Z_p[G_n]$-modules of standard $p$-ranks, except a bounded 
exponent which may increase as soon as the orders of the modules increase 
in \eqref{general}, regarding $N$. Under complete capitulation, one gets:
\begin{equation}\label{cas2}
\order \CH_{K_n}^\ram \times \order \big(\BE_K /\Norm_{K_n/K}(\BE_{K_n}) \big) 
= \order \CH_K \times p^{n \,(r-1)},\ \, \forall \, n \geq n_0.
\end{equation}

\begin{example}\label{1951}
{\rm Consider $K_3$ in the following example with $p=2$, $\ell = 17$ totally split 
in the cyclic cubic field $K$ of conductor $f=1951$, and complete capitulation of $\CH_K$ in $K_1$
(the others layers are computed for checking); 
we have $\CH_K \simeq \BZ/2\BZ$ and $\CH_{K_3} \simeq \BZ/8\BZ \times \BZ/4\BZ$:

\ft\begin{verbatim}
p=2  f=1951  PK=x^3+x^2-650*x-289  CK0=[2,2]  ell=17  r=3
CK1=[4,4,2,2]
h_1^[(S-1)^1]=[2,0,0,0]   h_2^[(S-1)^1]=[0,2,0,0] 
h_3^[(S-1)^1]=[0,0,0,0]   h_4^[(S-1)^1]=[0,0,0,0]
h_1^[(S-1)^2]=[0,0,0,0]   h_2^[(S-1)^2]=[0,0,0,0] 
h_3^[(S-1)^2]=[0,0,0,0]   h_4^[(S-1)^2]=[0,0,0,0]
norm in K1/K of the component 1 of CK1:[0,0,0,0]
norm in K1/K of the component 2 of CK1:[0,0,0,0]
norm in K1/K of the component 3 of CK1:[0,0,0,0]
norm in K1/K of the component 4 of CK1:[0,0,0,0]
Complete capitulation, m(K1)=2, e(K1)=2
CK2=[4,4,4,4]
h_1^[(S-1)^1]=[2,0,0,0]   h_2^[(S-1)^1]=[0,2,0,0] 
h_3^[(S-1)^1]=[2,0,2,0]   h_4^[(S-1)^1]=[0,2,0,2]
h_1^[(S-1)^2]=[0,0,0,0]   h_2^[(S-1)^2]=[0,0,0,0] 
h_3^[(S-1)^2]=[0,0,0,0]   h_4^[(S-1)^2]=[0,0,0,0]
norm in K2/K of the component 1 of CK2:[0,0,0,0]
norm in K2/K of the component 2 of CK2:[0,0,0,0]
norm in K2/K of the component 3 of CK2:[0,0,0,0]
norm in K2/K of the component 4 of CK2:[0,0,0,0]
Complete capitulation, m(K2)=2, e(K1)=2
CK3=[8,8,4,4]
h_1^[(S-1)^1]=[0,0,2,0]   h_2^[(S-1)^1]=[2,2,2,0]
h_3^[(S-1)^1]=[2,0,0,0]   h_4^[(S-1)^1]=[6,2,0,2]
h_1^[(S-1)^2]=[4,0,0,0]   h_2^[(S-1)^2]=[0,4,0,0]
h_3^[(S-1)^2]=[0,0,0,0]   h_4^[(S-1)^2]=[0,0,0,0]
norm in K3/K of the component 1 of CK3:[0,0,0,0]
norm in K3/K of the component 2 of CK3:[0,0,0,0]
norm in K3/K of the component 3 of CK3:[0,0,0,0]
norm in K3/K of the component 4 of CK3:[0,0,0,0]
Complete capitulation, m(K3)=3, e(K3)=3
\end{verbatim}\ns

We compute $\CH_{K_1}^\ram$, $\CH_{K_2}^\ram$ and $\CH_{K_3}^\ram$, 
adding the following instructions after a program running:

\ft\begin{verbatim}
For K1:
A=component(idealfactor(Kn,ell),1)[3];bnfisprincipal(Kn,A)[1]=[2,0,0,0]
\\Checking the principality of A^2 with components of a generator:
A2=idealpow(Kn,A,2);bnfisprincipal(Kn,A2)=[[0,0,0,0],[-51,0,-2,2,0,0]]
For K2:
B=component(idealfactor(Kn,ell),1)[3];bnfisprincipal(Kn,B)[1]=[0,0,2,2] 
\\Checking the principality of B^2 with components of a generator:
B2=idealpow(Kn,B,2);bnfisprincipal(Kn,B2)=[[0,0,0,0],
[-199260,-90688,102100,13880,4054,-14216,-8292,8559,-6223,-3433,3403,7557]]
For K3:
C=component(idealfactor(Kn,ell),1)[3];bnfisprincipal(Kn,C)[1]=[2,2,0,2] 
C2=idealpow(Kn,C,2);bnfisprincipal(Kn,C2)[1]=[4,4,0,0]
\\Checking the principality of C^4 with components of a generator:
C4=idealpow(Kn,C,4);bnfisprincipal(Kn,C4)=[[0,0,0,0], 
[-57074733,49681698,-55181004,32125541,42753200,-11450554,20535876,
-4037958, -4486534,-2178833,-1875179,3883122,-1527899,1665071,4332070,
2101150,1108465,-1251165, -1504106,445954,-292536,-677913,157262,-159406]]
\end{verbatim}\ns

The previous data and formula \eqref{cas2} give:
\begin{equation*}
\left\{ \begin{aligned}
\CH_{K_1} & \simeq \BZ/4\BZ \times \BZ/2\BZ,  \ \CH_{K_1}^{\sigma_1 - 1} = 
\langle h_1^2, h_2^2 \rangle \simeq \BZ/2\BZ, \\
\CH_{K_1}^{G_1} & \simeq \BZ/2\BZ \times \BZ/2\BZ, \ \CH_{K_1}^\ram 
\simeq \BZ/2\BZ, \  \BE_K/\Norm_{K_1/K}(\BE_{K_1}) \simeq \BZ/4\BZ, 
\end{aligned} \right.
\end{equation*}

\noindent
since $\BE_K$ is a free $\BZ$-module of rank $1$ and $\BE_K/\Norm_{K_1/K}(\BE_{K_1})$
of order $16$.
\begin{equation*}
\left\{ \begin{aligned}
\CH_{K_2} & \simeq \BZ/4\BZ \times \BZ/4\BZ,  \ \CH_{K_2}^{\sigma_2 - 1} = \CH_{K_2}^2
\simeq \BZ/2\BZ \times \BZ/2\BZ, \\
\CH_{K_2}^{G_2} & \simeq \BZ/2\BZ \times \BZ/2\BZ,  \  \CH_{K_2}^\ram \simeq \BZ/2\BZ, \  
\BE_K/\Norm_{K_2/K}(\BE_{K_2}) \simeq \BZ/4\BZ, 
\end{aligned} \right.
\end{equation*}
\begin{equation*}
\left\{ \begin{aligned}
\CH_{K_3} & \simeq \BZ/8\BZ \times \BZ/4\BZ,  \ \CH_{K_3}^{\sigma_3 - 1} = \CH_{K_3}^2
\simeq \BZ/4\BZ \times \BZ/2\BZ, \\
\CH_{K_3}^{G_3} & \simeq \BZ/4\BZ \times \BZ/2\BZ,  \  \CH_{K_3}^\ram \simeq \BZ/4\BZ, \  
\BE_K/\Norm_{K_3/K}(\BE_{K_3}) \simeq \BZ/4\BZ, 
\end{aligned} \right.
\end{equation*}

Chevalley--Herbrand's formula $\order \CH_{K_n}^{G_n} \!= 
\order \CH_K \times \ffrac{4^n}
{(\BE_K : \BE_K \cap \Norm_{K_n/K}(K_n^\times))}$ gives:
$$\BE_K/\BE_K \cap \Norm_{K_1/K}(K_1^\times) = 1 \ \  \&\ \  
\BE_K/\BE_K \cap \Norm_{K_n/K}(K_n^\times) \simeq \BZ/2\BZ\ \  \
\hbox{for $n \in \{2,3\}$}.$$ }
\end{example}

\subsubsection{Conclusion about orders versus structures} 
One may ask (for instance in the cubic case with $p=2$ to simplify), what happens in the 
tower $L = K(\mu_{\ell}^{})$ if $\CH_K$ has a large $p$-rank and/or a large exponent and if 
we suppose the capitulation of $\CH_K$ (necessarily in a larger layer $K_{n_0}$) ? 
The exact sequence \eqref{suite} looks like:
$$1 \to \BZ/2^{R_n}\BZ \too \CH_{K_n}^{G_n} \simeq \BZ/2^{x_n}\BZ \times \BZ/2^{y_n}\BZ \too 
\BZ/2^{E_n}\BZ \to 1, \ \  R_n, E_n \geq 0, $$
since $\CH_{K_n}^\ram \simeq \BZ/2^{R_n}\BZ$ and 
$\BE_K \cap \Norm_{K_n/K}(K_n^\times)/\Norm_{K_n/K}(\BE_{K_n}) 
\simeq \BZ/2^{E_n}\BZ$ are $\BZ$-monogenic. 
The Chevalley--Herbrand formula becomes $4^{x_n+y_n} = 4^H \times 4^{n \,\rho_n}$,
where $\order \CH_K = 4^H$ ($\rho_n = 0$ for all $n$, if $r=1$, otherwise $\rho_n$ increases 
up to some limit $\rho_N^{}$).

So, $\CH_{K_n}^{G_n}$ is at most the product of two $\BZ$-monogenic
components, possibly of large orders since $x_n+y_n = R_n+E_n = H + n \rho_n$.
In the case $r=1$ where $R_n = \rho_n = 0$, $\CH_{K_n}^{G_n}\simeq
\BE_K/\Norm_{K_n/K}(\BE_{K_n})$ (with orders $\order \CH_K$).

The philosophy of such examples is that whatever the structure of $\CH_K$,
there is no obstruction for the relations between {\it orders and structures} of invariants 
associated to the $\CH_{K_n}$'s by means of the exact sequence \eqref{suite}
and the Chevalley--Herbrand formula, invariants whose {\it algebraic structures} have 
canonical limitations, especially in terms of $p$-ranks (more precisely, 
$\BE_K/\Norm_{K_n/K}(\BE_{K_n})$ as monogenic $\Z_p[\Gal(K/\Q)]$-module
and $\CH_{K_n}^\ram$ of $p$-rank bounded by the number of ramified places
except one).

In the case of a real abelian base field $K$, this is typical of the Main Conjecture 
philosophy due to the analytic framework giving only orders and not precise 
structures (see \cite {Gras2022$^b$}).

\subsection{Decomposition of the algebraic norm 
\texorpdfstring{$\Nu_{\!L/K} \in \Z[G]$}{Lg}}
Let $G$ be the cyclic group of order $p^\N$ and let $\sigma$ be 
a generator of $G$. Put $x := \sigma - 1$; then:
$$\Nu_{\!L/K} = \sm_{i=0}^{p^\N - 1} \sigma^i = \sm_{i=0}^{p^\N-1} (x+1)^i =
\ffrac{(x+1)^{p^\N} -1}{x} = \sm_{i=1}^{p^\N} \bbinom{p^\N}{i} x^{i-1}. $$

We have the following elementary property which is perhaps known in Iwasawa's 
theory, but we have not found suitable references; see however Jaulent
\cite[IV.2\,(b)]{Jaul1986}, Washington \cite[\S\,13.3]{Wash1997} or 
Bandini--Caldarola \cite{BaCa2016,Cald2020} for classical computations
in the Iwasawa algebra $\Z_p[[T]]$:

\begin{theorem}\label{mainthm}
The algebraic norm $\Nu_{\!L/K} = \sm_{i=0}^{p^\N-1} \sigma^i \in \Z[G]$ 
is, for all $k \in [1, p^\N-1]$, of the form
$\Nu_{\!L/K} = (\sigma -1)^k \cdot A_k(\sigma -1, p)+p^{f(k)} \cdot B_k(\sigma -1, p)$, 
$A_k$, $B_k \in \Z[\sigma -1, p]$, where $f(k)  = N- s$ if $k \in [p^s, p^{s+1}-1]$ for $s \in [0, N-1]$.
\end{theorem}

\begin{proof}
From:
\begin{equation*}
\Nu_{\!L/K} = \bbinom{p^\N}{1} + x\bbinom{p^\N}{2}+ \cdots + x^{k-1}\bbinom{p^\N}{k} + 
x^k \Big[\bbinom{p^\N}{k+1} +x \bbinom{p^\N}{k+2} + \cdots + x^{p^\N-1-k} \bbinom{p^\N}{p^\N} \Big], 
\end{equation*}

\noindent
we deduce that:
$A_k(x, p) = \bbinom{p^\N}{k+1} + x \bbinom{p^\N}{k+2} +
\cdots +x^{p^\N-1-k} \bbinom{p^\N}{p^\N}$. 

The computation of $B_k(x, p)$ depends on the $p$-adic valuations
of the $\bbinom{p^\N}{j}$, $j \in [1,k]$. To find the maximal factor $p^{f(k)}$
dividing all the coefficients of the polynomial
$\bbinom{p^\N}{1} + x \bbinom{p^\N}{2} + \cdots + x^{k-1}\bbinom{p^\N}{k}$,   
in other words, to find the $p$-part of:
$$\gcd \Big (\bbinom{p^\N}{1}, \bbinom{p^\N}{2},\ldots , \bbinom{p^\N}{k} \Big),$$
we consider $s \in [0, N-1]$. 

Let $v$ be the $p$-adic valuation map.

\begin{lemma} 
One has $v\big (\binom{p^\N}{p^s} \big) = N-s,\, \forall s \in [0, N-1]$.
\end{lemma}

\begin{proof}
We have $\binom{p^\N}{p^s} = \ffrac{p^\N!}{p^s! \cdot (p^\N-p^s)!}$; then, using the well-known
formula: 
$$v(m!) = \ffrac{m - S(m)}{p-1}, \ \, m \geq 1, $$ 
where $S(m)$ is the sum of the digits in the
writing of $m$ in base $p$, we get: 
$$\hbox{$v(p^\N!)=\ffrac{p^\N-1}{p-1}$, \ \ $v(p^s!)=\ffrac{p^s-1}{p-1}$,\ \
$v((p^\N-p^s)!)=\ffrac{p^\N-p^s - (p-1) (N-s)}{p-1}$,}$$ 
since $p^\N-p^s$ may be written $p^s (p^{{}_{N-s}} - 1)$ with:
$$p^{{}_{N-s}} -1 =1 (p-1) +p (p-1)+ p^2(p-1) + \cdots + p^{{}_{N-s-2}}(p-1) + p^{{}_{N-s-1}}(p-1), $$
giving $N-s$ times the digit $p-1$. Whence, for all $s \in [0, N-1]$:

\vspace{0.2cm}
$v\big (\bbinom{p^\N}{p^s} \big) = 
\ffrac{1}{p-1}\big(p^\N-1 -(p^s-1) - (p^\N-p^s - (p-1) (N-s)) \big) = N-s$. 
\end{proof}

\begin{lemma} 
For $k \in [p^s+1,p^{s+1}-1]$, $s \in [0, N-1]$, we have $v\big (\binom{p^\N}{k} \big) \geq N-s$.
\end{lemma}

\begin{proof}
Consider $\binom{p^\N}{k}\binom{p^\N}{p^s}^{-1}$, $k \in [p^s+1, p^{s+1}-1]$
to check that its valuation is non-negative (the interval is empty for $p=2$, $s=0$, 
so, for $p=2$ we assume implicitly $s >0$).
We have:
\begin{equation*}
\begin{aligned}
\frac{\binom{p^\N}{k}}{\binom{p^\N}{p^s}} & = 
\frac{p^\N!}{k! \, (p^\N-k)!} \times \frac{p^s! \, (p^\N-p^s)!}{p^\N!} 
 = \frac{p^s!}{k!} \times  \frac{(p^\N-p^s)!}{(p^\N-k)!} \\
& = \frac{1}{(p^s+1)(p^s+2)\cdots (p^s + (k-p^s))} \times  \frac{(p^\N-p^s)!}{(p^\N-k)!}  \\
& = \frac{(p^\N-k+1)(p^\N-k+2) \cdots (p^\N-k + (k-p^s)) }{(p^s+1)(p^s+2)\cdots (p^s + (k-p^s))}.
\end{aligned}
\end{equation*}

Put $k=p^s + h$, $h \in [1,p^s(p-1)-1]$; then we can write:
\begin{equation*}
\begin{aligned}
\frac{\binom{p^\N}{k}}{\binom{p^\N}{p^s}} \!= &
\frac{[p^\N-(p^s+h)+1] [p^\N-(p^s+h)+2] \cdots [p^\N-(p^s+h)+h]}
{[p^s+1][p^s+2] \cdots [p^s+h]} \\
= & \frac{[p^\N-(p^s+h)+1] [p^\N-(p^s+h)+2] \cdots [p^\N-(p^s+h)+h]}
{[p^s+h] [p^s+h -1]\cdots [p^s+h-(h-1)]} \\
 = & \frac{[p^\N-(p^s+h)+1]}{[(p^s+h)-1]}
\frac{[p^\N-(p^s+h)+2]}{[(p^s+h)-2]}\cdots
\frac{[p^\N-(p^s+h)+(h-1)]}{[(p^s+h)-(h-1)]} \\
& \times \frac{p^\N-p^s}{p^s+h}.
\end{aligned}
\end{equation*}

We remark that each factor of the form 
$\ffrac{p^\N- [(p^s+h) - j]}{[(p^s+h)-j]}$ is a $p$-adic unit for
$j \in [1,h-1]$; indeed, one sees that $(p^s+h)-j \leq p^{s+1} - 2$
with $s+1 \leq N$, whence $v_p((p^s+h)-j) \leq N-1$.

Now, consider the remaining factor $\ffrac{p^\N-p^s}{p^s+h} = \ffrac{p^s (p^{N-s} - 1)}{p^s+h}
= \ffrac{p^s}{p^s+h}$, 
up to a $p$-adic unit since $s \in [0, N-1]$. As $h \leq p^s(p-1)-1$, one can put $h = \lambda p^u$, 
$p \nmid \lambda$, $u \leq s$; the case $u < s$ is obvious and gives a positive 
valuation; if $u=s$, the relation $h \leq p^s(p-1)-1$ implies $\lambda \leq p-2$, 
thus $p^s+h = p^s (1+ \lambda)$ with $1+ \lambda \leq p-1$ and $\ffrac{p^s}{p^s+h}$ 
is, in this case, a $p$-adic unit, whence the lemma.
\end{proof}

This leads to the expression of $f(k)$ on $ \bigcup_{s=0}^{N-1} 
[p^s, p^{s+1}-1] = [1, p^\N-1]$ and to the proof of the theorem.
\end{proof}

The following corollary, proving Theorems \ref{main1}, \ref{main2} and generalizations, 
is of easy use in practice; we assume, to simplify, that $L/K$ is totally ramified:

\begin{corollary}\label{maincoro}
Let $L/K$ be any totally ramified cyclic $p$-extension of degree $p^\N$, $N \geq 1$, 
of Galois group $G = \langle \sigma \rangle$. Let $m(\Ll)$ be the minimal integer 
such that $(\sigma - 1)^{m(\Ll)}$ annihilates $\CH_L$ and let $p^{e(\Ll)}$ be the 
exponent of $\CH_L$. 

{\bf a}) Then a sufficient condition of complete capitulation of $\CH_K$ in $L$ is that
$e(\Ll) \in[1, N - s(\Ll)]$ if $m(\Ll) \in [p^{s(\Ll)}, p^{s(\Ll)+1}-1]$ for $s(\Ll) \in [0, N-1]$.

A class $h \ne 1$ of $\CH_K$ capitulates in $L$ as soon as $h =: \Norm_{L/K}(h')$, 
where $h'$ is of order $p^e$ and annihilated by $(\sigma - 1)^m$ such that 
$e \in[1, N - s]$ if $m \in [p^s, p^{s+1}-1]$ for $s \in [0, N-1]$.

{\bf b}) For $t \geq 1$, put $\ov \CH_{\!\! L} := \CH_L/\CH_L^{p^t}$ 
and let $\ov m(\Ll)$, $\ov e(\Ll)$, be the corresponding
parameters for which $\ov m(\Ll) \leq m(\Ll)$ and $\ov e(\Ll) \leq e(\Ll)$; 
then $\ov \CH_{\!\! K}$ capitulates in $L$ as soon as
$\ov e(\Ll) \in[1, N - \ov s(\Ll)]$ if $\ov m(\Ll) \in [p^{{\ov s}(\Ll)}, p^{\ov s(\Ll)+1}-1]$ 
for $\ov s(\Ll) \in [0, N-1]$. 

{\bf c}) Stability of the $\ov \CH_{\!\! K_n}$'s  in the tower $L = \bigcup_{n = 0}^\N K_n$ 
holds as soon as $\order \ov \CH_{\!\! K_1} = \order \ov \CH_{\!\! K}$ and leads to
$\ov m(\Kk_n) = 1$, $\ov s(\Kk_n)=0$, $\ov e(\Kk_n) = \ov e(\Kk)$ for all $n$.
\end{corollary}

\begin{proof}
If $h' \in \CH_L$, $\Nu_{\!L/K}(h') = \J_{L/K} (\Norm_{L/K}(h')) = \big( h'^{(\sigma-1)^k} \big)^{A} 
\times \big( h'^{p^{f(k)}} \big)^{B},\, \forall k \in [1, p^\N-1]$. Thus, $\J_{L/K}$ is non-injective 
as soon as $h$ fulfills the conditions stated in ({\bf b}) with $k=m \in [p^s, p^{s+1}-1]$,
$s \in [0, N-1]$, and $f(k) = N-s \geq e$. 
For ({\bf a}) on the triviality of $\Nu_{\!L/K}(\CH_L)$, it suffices that 
$m=m(\Ll)$, $s=s(\Ll)$ and $e=e(\Ll)$ be solution. One obtains Theorem \ref{main1}.

The case ({\bf c}) of quotients is immediate; their capitulation is, a priori, ``easier''
and means that any ideal ${\mathfrak a}$ of $K$ becomes of the form 
$({\mathfrak a})_L = (\alpha) \cdot {\mathfrak A}^{p^t}$, where ${\mathfrak A}$ 
is an ideal of $L$ and $\alpha \in L^\times$. The case $t = 1$ of stability 
property gives the stability of the $p$-ranks.
\end{proof}

In other words, if the length $m(\Ll)$ of the filtration is not too large as well as 
the exponent $p^{e(\Ll)}$ of $\CH_L$, then we obtain $\Nu_{\!L/K}(h) = 1$ for 
all $h \in \CH_L$ (or at least for some), whence complete (or partial) capitulation 
of $\CH_K$ in $L$.

Another way to interpret this result is to say that if $N$ is large enough and
if the Galois complexity of the $p$-class groups $\CH_{K_n}$
does not increase too much, then $\CH_K$ capitulates in $L$. 
For this, we may introduce the following definition:

\begin{definition} \label{smooth}
Let $L/K$ be a cyclic $p$-extension totally ramified of degree $p^\N>1$ 
and Galois group $G =: \langle \sigma \rangle$ (we do not assume that 
$L/\Q$ is Galois). Let $\CH_K$,
$\CH_L$ be the $p$-class groups of $K$, $L$, respectively. Denote
by $p^{e(\Ll)}$ the exponent of $\CH_L$ and by $m(\Ll)$ the length of the
filtration $\big\{ \CH_L^i \big\}_{i \geq 0}$ (i.e., the least integer $m(\Ll)$ such 
that $(\sigma - 1)^{m(L)}$ annihilates $\CH_L$).
We will say that $L/K$ is of {\it smooth complexity} when $e(\Ll) \leq N - s(\Ll)$ 
if $m(\Ll) \in [p^{s(\Ll)}, p^{s(\Ll)+1}-1]$ for $s(\Ll) \in [0, N-1]$.
\end{definition}

\subsubsection{The Furtw\"{a}ngler property}
This property is the strong equality: 
$$\Ker_{\CH_L}(\Norm_{L/K}) = \CH_L^{\sigma - 1}; $$

\noindent 
it is equivalent to say that the genus field of $L/K$ is $L H_K^\nr$, then
obviously equivalent to $\order \CH_L^G = \order \CH_K$, thus to the triviality 
of the norm factor in the Chevalley--Herbrand formula (e.g., case $r = 1$). 
We have discovered some 
applications in Bembom's thesis, about the Galois structure of $\CH_L$ via its filtration
and the problem of capitulation (see, e.g., \cite[\S\,2.6, Theorem 2.6.3; \S\,2.8, 
Theorem 2.8.9]{Bem2012}).
Under the Furtw\"{a}ngler property, Nakayama's Lemma gives immediately:
\begin{proposition}\label{Bembom}
Let $L/K$ be a totally ramified cyclic $p$-extension such that 
$\order \CH_L^G = \order \CH_K$. From $\CH_K = \bigoplus_j \langle h_j \rangle$ 
set $\CH'_L:= \bigoplus_{} \langle h'_j \rangle$, $h'_j \in \CH_L$ be such 
that $h_j = \Norm_{L/K}(h'_j),\, \forall i$. Then $\CH_L$ is generated 
by $\CH'_L$ as $\Z_p[\sigma - 1]$-module and any $h' \in \CH_L$ 
writes $h'=\prod_j h'_j{}^{\omega_j}$, where $\omega_j = 
\sum_{i=0}^{m(L)-1} a_{i,j} (\sigma - 1)^i$, $a_{i,j} \in \Z_p$.
\end{proposition}

\subsubsection{Program giving the decompositions of 
\texorpdfstring{$\Nu_{\!L/K}$}{Lg}}
The following program put  $\Nu_{\!L/K}$ in the form 
$P(x,p) = x^k A(x,p) + p^{f(k)} B(x,p)$, $1 \leq k \leq p^\N-1$, 
where $x=\sigma - 1$ and $p$ are considered as indeterminate variables.
This is necessary to have universal expressions for any $\CH_L$ as
$\Z_p[G]$-module; in other words, we do not reduce modulo $p$ the 
coefficients of $A$ and $B$.
One must note that, except the cases $k=1$, $B=1$ and $k=p^\N-1$, $A=1$,
$A$ and $B$ are not invertible in the group algebra, which 
allows improvements of the standard reasoning of annihilation with 
$( (\sigma - 1)^k$, $p^{f(k)})$.
One must precise the numerical prime number $p$ in ${\sf Prime}$ 
and $N$ in ${\sf N}$:

\ft\begin{verbatim}
PROGRAM OF DECOMPOSITION OF THE ALGEBRAIC NORM
{Prime=2;N=3;P=0;for(i=1,Prime^N,C=binomial(Prime^N,i);
v=valuation(C,Prime);c=C/Prime^v;P=P+c*p^v*x^(i-1));print("P=",P);
for(k=1,Prime^N-1,B=lift(Mod(P,x^k));w=valuation(B,p);print();
print("P=x^",k,".A+p^",w,".B");print("A=",(P-B)/x^k);print("B=",B/p^w))}
\end{verbatim}\ns
\ft\begin{verbatim}
Case p=2, N=1, P=x+p$
P=x+p
A=1    B=1
\end{verbatim}\ns
\ft\begin{verbatim}
Case p=3, N=1, P=x^2+p*x+p
P=x^1.A+p^1.B              P=x^2.A+p^1.B
A=x+p     B=1              A=1     B=x+1
\end{verbatim}\ns
\ft\begin{verbatim}
Case p=2, N=2, P=x^3+p^2*x^2+3*p*x+p^2
P=x^1.A+p^2.B              P=x^2.A+p^1.B           P=x^3.A+p^1.B
A=x^2+p^2*x+3*p  B=1       A=x+p^2  B=3*x+p        A=1  B=p*x^2+3*x+p
\end{verbatim}\ns

For the next examples, we only write ${\sf P}$.

\ft\begin{verbatim}
Case p=3, N=2:
P=x^8+p^2*x^7+4*p^2*x^6+28*p*x^5+14*p^2*x^4+14*p^2*x^3+28*p*x^2+4*p^2*x+p^2
\end{verbatim}\ns
\ft\begin{verbatim}
Case p=2, N=3:
P=x^7+p^3*x^6+7*p^2*x^5+7*p^3*x^4+35*p*x^3+7*p^3*x^2+7*p^2*x+p^3
\end{verbatim}\ns
\ft\begin{verbatim}
Case p=2, N=4:
P=x^15+p^4*x^14+15*p^3*x^13+35*p^4*x^12+455*p^2*x^11+273*p^4*x^10
  +1001*p^3*x^9+715*p^4*x^8+6435*p*x^7+715*p^4*x^6+1001*p^3*x^5
  +273*p^4*x^4+455*p^2*x^3+35*p^4*x^2+15*p^3*x+p^4
\end{verbatim}\ns

Thus, as soon as $\Nu_{\!L/K} = 
(\sigma -1)^k \times A_k(\sigma -1, p) + p^{f(k)}\times B_k(\sigma -1, p)$, 
with $k \geq m(\Ll)$ and $f(k) \geq e(\Ll)$, then $\CH_K$ capitulates in $L$. 
But the reciprocal does not hold and capitulation may occur whatever 
$m(\Ll)$ and $e(\Ll)$. 
Indeed, if $m(\Ll) \in [p^{s(L)}, p^{s(L)+1}-1]$ for $s(\Ll) \in [0, N-1]$, 
with $e(\Ll) > N - s(\Ll)$ instead of $e(\Ll) \leq N - s(\Ll)$, reducing ${\sf P}$ 
modulo the ideal $\big(x^{m(\Ll)}, p^{e(\Ll)}\big)$, the capitulation of $\CH_K$ is 
equivalent to the annihilation of $\CH_L$ by an explicit polynomial of the form:
$$\sm_{j=1}^{j+h} a_j \, p^{e_j} x^{m_j},\ \ p \nmid a_j,\ e_j \in [0, e(\Ll)-1],
\ m_j \in [0, m(\Ll)-1] , $$
which may hold by accident depending on the numerical data of the filtration; we 
give such a case in Example \ref{pasnecessaire}\,(ii). We will also give examples 
where the above property of $\Nu_{L/K}$ applies, apart from the obvious case of stability.

\subsection{Illustrations using the decomposition 
of \texorpdfstring{$\Nu_{\!L/K}$}{Lg}}\label{mainexamples2}

In what follows, the cyclic $p$-extensions $L/K$ are always totally ramified.

\begin{example}\label{pasnecessaire}
{\rm Let's begin with an example where the structure of $\CH_{K_n}$ grows sufficiently 
with $n$, giving no capitulation of $\CH_K$ up to $n=3$; next, another
choice of $\ell$ leads to capitulation of $\CH_K$, but where Theorem \ref{main1} does not apply.

(i) We consider the cyclic cubic field $K$ of conductor $f=703$, $p=2$
and $\ell = 17$. Then $\CH_K \simeq \BZ/2\BZ$ and we 
get, from Program \ref{cubic} (several hours for the level $n=3$):

\ft\begin{verbatim}
p=2  Nn=3  f=703  PK=x^3+x^2-234*x-729  mKn=3  CK0=[6,2]  ell=17  r=1
CK1=[12,4]=[4,4]
h_1^[(S-1)^1]=[0,2]      h_2^[(S-1)^1]=[2,2]
h_1^[(S-1)^2]=[0,0]      h_2^[(S-1)^2]=[0,0]
h_1^[(S-1)^3]=[0,0]      h_2^[(S-1)^3]=[0,0]
norm in K1/K of the component 1 of CK1:[2,2]
norm in K1/K of the component 2 of CK1:[2,0]
CK2=[24,8]=[8,8]
h_1^[(S-1)^1]=[0,2]      h_2^[(S-1)^1]=[6,6]
h_1^[(S-1)^2]=[4,4]      h_2^[(S-1)^2]=[4,0]
h_1^[(S-1)^3]=[0,0]      h_2^[(S-1)^3]=[0,0] 
norm in K2/K of the component 1 of CK2:[4,4]
norm in K2/K of the component 2 of CK2:[4,0]
CK3=[48,16]=[16,16]
h_1^[(S-1)^1]=[8,10]     h_2^[(S-1)^1]=[10,14]
h_1^[(S-1)^2]=[4,12]     h_2^[(S-1)^2]=[12,8]
norm in K3/K of the component 1 of CK3:[8,8]
norm in K3/K of the component 2 of CK3:[8,0]
\end{verbatim}\ns

For $n=1$, we have $m(\Kk_1) = 2$, $s(\Kk_1)=1$, $n-s(\Kk_1) = 0$ and $e(\Kk_1)=2 >0$.

For $n=2$, we have $m(\Kk_2) = 3$, $s(\Kk_2)=1$, $n-s(\Kk_2) = 1$ and $e(\Kk_2)=3 >1$.

 For $n=3$, we have $m(\Kk_3) = 4$, $s(\Kk_3)=2$, $n-s(\Kk_3) = 1$ and $e(\Kk_3)=4 >1$. 

In this case the complexity (non smooth) increases due to successive exponents $2,4,8,16$. 
Nevertheless $\ov \CH_{\!\!K} = \CH_K/\CH_K^2$ capitulates in $K_1$ because 
$\ov m(\Kk_1) = 1$ since
${\sf h_1^{[(S-1)^1]}=[0,2]}$, ${\sf h_2^{[(S-1)^1]}=
[2,2]}$ ($\ov s(\Kk_1) = 0$), then $\ov e(\Kk_1)=1 \leq n-\ov s(\Kk_1)$.

(ii) Changing $\ell=17$ into $\ell = 97$ ($N=4$) gives complete capitulation in $K_2$
because of a smooth complexity, but higher than a stability in $K_2/K$ (a stability 
begins at $n_0=2$):

\ft\begin{verbatim}
p=2  Nn=2  f=703  PK=x^3+x^2-234*x-729  CK0=[6,2]  ell=97  r=1
CK1=[6,2,2,2]=[2,2,2,2]
h_1^[(S-1)^1]=[1,1,0,0]      h_2^[(S-1)^1]=[1,1,0,0]  
h_3^[(S-1)^1]=[0,0,1,1]      h_4^[(S-1)^1]=[0,0,1,1]
h_1^[(S-1)^2]=[0,0,0,0]      h_2^[(S-1)^2]=[0,0,0,0]
h_3^[(S-1)^2]=[0,0,0,0]      h_4^[(S-1)^2]=[0,0,0,0]
norm in K1/K of the component 1 of CK1:[1,1,0,0]
norm in K1/K of the component 2 of CK1:[1,1,0,0]
norm in K1/K of the component 3 of CK1:[0,0,1,1]
norm in K1/K of the component 4 of CK1:[0,0,1,1]
No capitulation, m(K1)=2, e(K1)=1
CK2=[12,4,2,2]=[4,4,2,2]
h_1^[(S-1)^1]=[0,2,1,1]      h_2^[(S-1)^1]=[0,2,1,0]
h_3^[(S-1)^1]=[2,2,0,0]      h_4^[(S-1)^1]=[2,0,0,0]
h_1^[(S-1)^2]=[0,2,0,0]      h_2^[(S-1)^2]=[2,2,0,0]
h_3^[(S-1)^2]=[0,0,0,0]      h_4^[(S-1)^2]=[0,0,0,0]
h_1^[(S-1)^3]=[0,0,0,0]      h_2^[(S-1)^3]=[0,0,0,0]
h_3^[(S-1)^3]=[0,0,0,0]      h_4^[(S-1)^3]=[0,0,0,0]
norm in K2/K of the component 1 of CK2:[0,0,0,0]
norm in K2/K of the component 2 of CK2:[0,0,0,0]
norm in K2/K of the component 3 of CK2:[0,0,0,0]
norm in K2/K of the component 4 of CK2:[0,0,0,0]
Complete capitulation, m(K2)=3, e(K2)=2
CK3=[12,4,2,2]=[4,4,2,2] (a day of computation)
\end{verbatim}\ns

For $n=2$, ${\sf P=x^3 + 2^2*x^2 + 3*2*x + 2^2}$, $m(\Kk_2) = 3$, $s(\Kk_2)=1$ and 
$e(\Kk_2) = 2 > n-s(\Kk_2) = 1$; so the property deduced from the use of $\Nu_{K_2/K}$ 
does not hold. Then, modulo the ideal $(x^3, 2^2)$, it follows that ${\sf 3*2*x}$ 
must annihilate $\CH_{K_2}$, which is confirmed by the data since 
$h_j^{2 (\sigma_2-1)} =1$, $1 \leq j \leq 4$. The stability from $K_2$ gives a smooth
complexity of the extension $L/K$ and confirm the capitulation.

So, Theorem \ref{main1} gives a sufficient condition of capitulation, not necessary, but
some information remains when the condition is not fulfilled.}
\end{example}

\begin{example}{\rm
We consider the cyclic cubic field $K$ of conductor $f=1777$, $p=2$ and 
$\ell = 17$. Then $\CH_K \simeq \BZ/4\BZ$ is of exponent $4$.

\ft\begin{verbatim}
p=2  Nn=3  f=1777  PK=x^3+x^2-592*x+724  CK0=[4,4]  ell=17  r=3
CK1=[8,8]
h_1^[(S-1)^1]=[0,0]     h_2^[(S-1)^1]=[0,0]       
norm in K1/K of the component 1 of CK1:[2,0]
norm in K1/K of the component 2 of CK1:[0,2]
No capitulation, m(K1)=1, e(K1)=3
CK2=[8,8]
h_1^[(S-1)^1]=[0,0]     h_2^[(S-1)^1]=[0,0] 
norm in K2/K of the component 1 of CK2:[4,0]
norm in K2/K of the component 2 of CK2:[0,4]
Incomplete capitulation, m(K2)=1, e(K2)=3
CK3=[8,8]
h_1^[(S-1)^1]=[0,0]     h_2^[(S-1)^1]=[0,0] 
norm in K3/K of the component 1 of CK3:[0,0]
norm in K3/K of the component 2 of CK3:[0,0]
Complete capitulation, m(K3)=1, e(K3)=3
\end{verbatim}\ns

In this case, the stability from $K_1$ implies necessarily the capitulation in $K_3$ 
(so the third computation is for checking). Moreover, for all $n \geq 1$, $\CH_{K_n}$ 
is annihilated by $\sigma - 1$ and $\CH_{K_n} = \CH_{K_n}^{G_n}$ 
as expected from Theorem \ref{main2}\,(i) and given by Program \ref{progfiltr}.
Whence $m(\Kk_n) = 1$, $s(\Kk_n)=0$, 
$e(\Kk_n) = 3$ giving $e(\Kk_n) \leq n - s(\Kk_n)$ only from $n=3$.
Note that $\CH_K/\CH_K^2$ capitulates in $K_1$ and that, considering
the quotients $\ov \CH_{\!\!K_n} := \CH_{K_n}/\CH_{K_n}^4$, then $\ov \CH_{\!\!K} =
\CH_K/\CH_K^4$ capitulates only from $K_2$.}
\end{example}

\begin{example}{\rm
(i) We consider the quadratic field $\Q(\sqrt {142})$ with $p=3$ and various
primes $\ell \equiv 1 \pmod 3$, $\ell \not\equiv 1 \pmod 9$, so that
$N=1$, $L=K_1$, and the sufficient conditions of Theorem \ref{main1} are
$e(\Kk_1) = 1 \leq 1-s(\Kk_1)$, whence $s(\Kk_1)=0$ and $m(\Kk_1) \in [1, 2]$:

\ft\begin{verbatim}
p=3  PK=x^2-142  N=1  CK0=[3]  ell=13  r=2 
CK1=[3,3]
h_1^[(S-1)^1]=[0,0]   h_2^[(S-1)^1]=[0,0]
h_1^[(S-1)^2]=[0,0]   h_2^[(S-1)^2]=[0,0]
norm in K1/K of the component 1 of CK1:[0,0]
norm in K1/K of the component 2 of CK1:[0,0]
Complete capitulation, m(K1)=1, e(K1)=1
\end{verbatim}\ns
\ft\begin{verbatim}
p=3  PK=x^2-142  N=1  CK0=[3]  ell=1123  r=2 
CK1=[21,3]=[3,3]
h_1^[(S-1)^1]=[1,2]   h_2^[(S-1)^1]=[1,2]
h_1^[(S-1)^2]=[0,0]   h_2^[(S-1)^2]=[0,0]
norm in K1/K of the component 1 of CK1:[0,0]
norm in K1/K of the component 2 of CK1:[0,0]
Complete capitulation, m(K1)=2, e(K1)=1
\end{verbatim}\ns
\ft\begin{verbatim}
p=3  PK=x^2-142  N=1  CK0=[3]  ell=208057  r=2
CK1=[3,3,3,3]
h_1^[(S-1)^1]=[2,2,0,0]   h_2^[(S-1)^1]=[1,1,0,0]
h_3^[(S-1)^1]=[0,1,1,1]   h_4^[(S-1)^1]=[2,1,2,2]
h_1^[(S-1)^2]=[0,0,0,0]   h_2^[(S-1)^2]=[0,0,0,0]
h_3^[(S-1)^2]=[0,0,0,0]   h_4^[(S-1)^2]=[0,0,0,0]
norm in K1/K of the component 1 of CK1:[0,0,0,0]
norm in K1/K of the component 2 of CK1:[0,0,0,0]
norm in K1/K of the component 3 of CK1:[0,0,0,0]
norm in K1/K of the component 4 of CK1:[0,0,0,0]
Complete capitulation, m(K1)=2, e(K1)=1
\end{verbatim}\ns

(ii) For $p=5$ and $N=1$ the conditions become 
$e(\Kk_1) = 1 \leq 1-s(\Kk_1)$, whence $s(\Kk_1)=0$, with $m(\Kk_1) \in [1, 4]$,
which offers more possibilities:

\ft\begin{verbatim}
p=5  PK=x^2-401  N=1  CK0=[5]  ell=1231  r=2
CK1=[5,5]
h_1^[(S-1)^1]=[0,0]    h_2^[(S-1)^1]=[0,0]
h_1^[(S-1)^2]=[0,0]    h_2^[(S-1)^2]=[0,0]
norm in K1/K of the component 1 of CK1:[0,0]
norm in K1/K of the component 2 of CK1:[0,0]
Complete capitulation, m(K1)=1, e(K1)=1
\end{verbatim}\ns
\ft\begin{verbatim}
p=5  PK=x^2-401  N=1  CK0=[5]  ell=1741  r=1
CK1=[5,5]
h_1^[(S-1)^1]=[3,3]    h_2^[(S-1)^1]=[2,2]
h_1^[(S-1)^2]=[0,0]    h_2^[(S-1)^2]=[0,0]
norm in K1/K of the component 1 of CK1:[0,0]
norm in K1/K of the component 2 of CK1:[0,0]
Complete capitulation, m(K1)=2, e(K1)=1
\end{verbatim}\ns
\ft\begin{verbatim}
p=5  PK=x^2-401  CK0=[5]  ell=4871  r=1
CK1=[10,10,10,2]=[5,5,5]
h_1^[(S-1)^1]=[4,0,4,0]  h_2^[(S-1)^1]=[1,4,0,0]  h_3^[(S-1)^1]=[3,4,2,0]  
h_1^[(S-1)^2]=[3,1,4,0]  h_2^[(S-1)^2]=[3,1,4,0]  h_3^[(S-1)^2]=[2,4,1,0]  
h_1^[(S-1)^3]=[0,0,0,0]  h_2^[(S-1)^3]=[0,0,0,0]  h_3^[(S-1)^3]=[0,0,0,0]  
norm in K1/K of the component 1 of CK1:[0,0,0,0]
norm in K1/K of the component 2 of CK1:[0,0,0,0]
norm in K1/K of the component 3 of CK1:[0,0,0,0]
norm in K1/K of the component 4 of CK1:[0,0,0,0]
Complete capitulation, m(K1)=3, e(K1)=1
\end{verbatim}\ns}
\end{example}

\begin{example}{\rm
We consider the cubic field of conductor $f=20887$ with $p=2$ and $\ell = 17$
totally split:

\ft\begin{verbatim}
p=2  Nn=2  f=20887  PK=x^3+x^2-6962*x-225889  CK0=[4,4,2,2]  ell=17  r=3
CK1=[8,8,2,2]
h_1^[(S-1)^1]=[0,0,0,0]   h_2^[(S-1)^1]=[0,0,0,0]
h_3^[(S-1)^1]=[0,0,0,0]   h_4^[(S-1)^1]=[0,0,0,0]
norm in K1/K of the component 1 of CK1:[2,0,0,0]
norm in K1/K of the component 2 of CK1:[0,2,0,0]
norm in K1/K of the component 3 of CK1:[0,0,0,0]
norm in K1/K of the component 4 of CK1:[0,0,0,0]
Incomplete capitulation, m(K1)=1, e(K1)=3
\end{verbatim}\ns

We note that $\Norm_{K_1/K}(h_j) \ne 1$ for $j=3,4$, otherwise
$\Norm_{K_1/K}(\CH_{K_1}) = \CH_K$ would be of $2$-rank $2$ (absurd).
Since $m(\Kk_1)=1$ (all classes are invariant), Theorem \ref{main1} applies 
non-trivially for the classes $h_3, h_4$ of order~$2$ ($m=1$, $s=0$, $e \in [1, 1]$, 
which is indeed the case).
Let's give the complete data checking the capitulation of the two classes of $K$
of order $2$; the instruction ${\sf CK0=K.clgp}$ gives:

\ft\begin{verbatim}
[64,[4,4,2,2],[[2897,2889,2081;0,1,0; 0,0,1],[2897,825,2889;0,1,0;0,0,1],
[17,16,13;0,1,0;0,0,1],[53,36,44;0,1,0; 0,0,1]]]
\end{verbatim}\ns
\noindent
it describes $\CH_K$ with $4$ representative ideals of generating classes; that 
of order $2$ are ${\sf {\mathfrak a}_3 = [17, 16, 13; 0, 1, 0; 0, 0, 1],  
{\mathfrak a}_4 = [53, 36, 44; 0, 1, 0; 0, 0, 1]}$; the following $6$ large coefficients 
on the integral basis give integers $\alpha_i \in L^\times$ with the relations 
$({\mathfrak a}_i)_L = (\alpha_i)$:

\ft\begin{verbatim}
[[0,0,0,0],[4482450896,-1173749328,81969609,69123722,7646555,39729395]]
norm in K1/K of the component 3 of CK1:[0,0,0,0]
[[0,0,0,0],[-4877380814,1968946273,-1411818,102996743,38571732,40207952]]
norm in K1/K of the component 4 of CK1:[0,0,0,0]
\end{verbatim}\ns

At the level $n=2$, the result is similar, but shows that the classes of order $4$
of $\CH_K$ never capitulate:

\ft\begin{verbatim}
CK2=[16,16,2,2]
h_1^[(S-1)^1]=[8,0,0,0] h_2^[(S-1)^1]=[0,8,0,0] 
h_3^[(S-1)^1]=[0,0,0,0] h_4^[(S-1)^1]=[0,0,0,0]
h_1^[(S-1)^2]=[0,0,0,0] h_2^[(S-1)^2]=[0,0,0,0] 
h_3^[(S-1)^2]=[0,0,0,0] h_4^[(S-1)^2]=[0,0,0,0]
norm in K2/K of the component 1 of CK2:[4,0,0,0]
norm in K2/K of the component 2 of CK2:[0,4,0,0]
norm in K2/K of the component 3 of CK2:[0,0,0,0]
norm in K2/K of the component 4 of CK2:[0,0,0,0]
Incomplete capitulation, m(K2)=2, e(K2)=4
\end{verbatim}\ns}
\end{example}

\section{Arithmetic invariants that do not capitulate}

The non capitulation of $p$-class groups $\CH_K$ in cyclic $p$-extensions $L/K$
implies necessarily, as we have seen, that the structure of $\CH_L$ does not allow the 
previous use of the algebraic norm (Theorem \ref{main1}) and the complexity is not 
smooth according to Definition \ref{smooth}; that is to say, either $m(\Kk_n) \geq p^n$ 
or else $m(\Kk_n) \in [p^s, p^{s+1}-1]$ for $s \in [0, n-1]$ but in that case, $e(\Kk_n) >n-s$. 
This may be checked by means of a wider framework as follows:

\subsection{Injective transfers and arithmetic consequences}

Let $p$ be any prime number and let $\CK$ be a family of number fields $k$
stable by taking subfields.

\begin{definition} 
Let $\{\CX_k\}_{k \in \CK}$ be a family of finite invariants of $p$-power order, indexed 
by the set $\CK$, fulfilling the following conditions for $k, k' \in \CK$, $k \subseteq k'$:

(i) For any Galois extension $k'/k$, of Galois group $G$, $\CX_{k'}$
is a $\Z_p[G]$-module.

(ii) There exist norms $\Norm_{k'/k} : \CX_{k'} \to \CX_k$ and transfers 
$\J_{k'/k} : \CX_k \to \CX_{k'}$, such that $\J_{k'/k} \circ \Norm_{k'/k} = \Nu_{\!k'/k}$.

(iii) If $G=\Gal(k'/k)$ is a cyclic $p$-group, we define the associated filtration 
$\big \{\CX_{k'}^i  \big \}_{i \geq 0}$ by $\CX_{k'}^{i+1} / \CX_{k'}^i := 
(\CX_{k'}/\CX_{k'}^i)^G,\, \forall i \geq 0$.
\end{definition}

Thus, for a cyclic $p$-extension $L/K$, $L, K \in \CK$, let $m(\Ll)$ be the length of 
the filtration; the condition $e(\Ll) \in [1, N-s(\Ll)]$ if $m(\Ll) \in [p^{s(\Ll)}, p^{s(\Ll)+1}-1]$ 
for $s(\Ll) \in [0,N-1]$, of Theorem \ref{main1}, applies in the same way, independently of 
the fact of being able to calculate the orders of the $\CX_L^i$'s by means of a suitable algorithm. 

\begin{theorem}\label{nocap}
Let $L/K$, $K, L \in \CK$, be a cyclic $p$-extension of degree $p^\N$, 
let $K_n$ be the subfield of $L$ of degree $p^n$ over $K$, $n \in [0, N]$
and to simplify, put $\CX_n := \CX_{K_n}$. We assume that, for all $n \in [0, N-1]$, 
the arithmetic norms $\CX_{n+1} \to \CX_n$ are surjective 
and that the transfer maps $\CX_n \to \CX_{n+1}$ are injective. 
Let $p^{e_n}$ be the exponent of $\CX_n$.

Then $\order \CX_{n+h}\geq \order \CX_n \cdot \order \CX_n[p^h] \geq  
\order \CX_n  \cdot p^{\min(e_n,h)}$, for all $n \in [0, N]$ and all $h \in [0, N-n]$, 
where $\CX_n[p^h] := \{x \in \CX_n,\ \, x^{p^h} = 1\}$.
\end{theorem}

\begin{proof}
Put $G_n^{n+h} := \Gal(k_{n+h}/k_n)$ and in the same way for $\Norm_n^{n+h}$, 
$\J_n^{n+h}$. From the exact sequence
$1 \to  \J_n^{n+h}\CX_n \to \CX_{n+h} \to \CX_{n+h}/ \J_n^{n+h} \CX_n \to 1$, we get:
\begin{equation*}
\begin{aligned}
1  \to  \CX_{n+h}^{G_n^{n+h}} / \J_n^{n+h} \CX_n & \to (\CX_{n+h}/ \J_n^{n+h} \CX_n)^{G_n^{n+h}} \\
& \to  {\rm H}^1(G_n^{n+h}, \J_n^{n+h} \CX_n) \to  {\rm H}^1(G_n^{n+h}, \CX_{n+h}),
\end{aligned}
\end{equation*}

\noindent
where ${\rm H}^1(G_n^{n+h}, \J_n^{n+h} \CX_n) = (\J_n^{n+h} \CX_n) [p^h] 
\simeq \CX_n[p^h]$ (injectivity of $\J_n^{n+h}$), 
$$\order {\rm H}^1(G_n^{n+h}, \CX_{n+h}) = \order {\rm H}^2(G_n^{n+h}, \CX_{n+h})
= \order (\CX_{n+h}^{G_n^{n+h}} / \J_n^{n+h} \CX_n), $$
since $\Nu_n^{n+h}\CX_{n+h} =  \J_n^{n+h} \circ \Norm_n^{n+h} \CX_{n+h} =
\J_n^{n+h} \CX_n$ (surjectivity of $\Norm_n^{n+h}$), giving an exact sequence of the form:
$$1\to A \to (\CX_{n+h}/ \J_n^{n+h} \CX_n)^{G_n^{n+h}}  
\to \CX_n[p^h] \to A',\ \hbox{with $\order A' = \order A$} . $$ 

We then obtain the inequality $\order \CX_{n+h}\geq \order \CX_n \cdot \order \CX_n[p^h]$,
whence the results.
\end{proof}

\begin{corollary} The $n$-sequence $\order \CX_n$ stabilizes from 
some $n_0 \in [0, N-1]$ if and only if $\CX_n= 1$, for all $n \in [0, N]$. 

\noindent
In an Iwasawa's theory context with $\mu = 0$ and $\lambda > 0$, 
let $p^{e_n}$ (resp. $r_n$) be the exponent (resp. the $p$-rank) 
of $\CX_n$, $n \in [0, N]$. Then $r_n$ is a constant $r$,\, $\forall n \gg 0$ 
and $e_n \to \infty$ with $n$; in particular, if $\lambda = \mu = 0$, then 
$\CX_n=1$, for all $n \geq 0$.
\end{corollary}

\begin{proof}
The stability from $n_0$ means $\CX_{n_0}[p] = 1$, then 
$\CX_{n_0} = 1$ and $\CX_n =1$, for all $n \in [n_0, N]$; the
surjectivity of the norms implies $\CX_n =1$, for all $n \in [0, n_0]$.

If $\mu = 0$ in the formula $\order \CX_n = p^{\lambda\, n + \mu\,p^n + \nu}$, for 
$n \gg 0$, the relation $\order \CX_{n+1}\geq \order \CX_n \cdot \CX_n[p]$ 
implies $r_n \leq \lambda$; since $r_n$ is increasing (injectivity of the transfers), 
$r_n = r,\, \forall n \gg 0$; since $\order \CX_n \leq p^{r\,e_n}$, one gets
$\lambda n + \nu \leq r e_n$ 
proving that $e_n \to \infty$ with $n$. If $\lambda = \mu = 0$, then $\order \CX_n$ is 
constant for $n \gg 0$, whence $\CX_n=1,\, \forall n \geq 0$.
\end{proof}

We remark that if $r_n$ is unbounded, necessarily $\mu >0$. But all 
of these results are based on strong assumptions avoiding any capitulation. So,
Theorem \ref{main1} does not apply since the complexity of the $\CX_n$'s 
crucially increases with $n$. We shall give two examples of such families.

\subsection{\texorpdfstring{$p$}{Lg}-class groups of imaginary quadratic field}
If $K$ is an imaginary quadratic field and $L = L_0 K$, $L_0/\Q$ real cyclic 
of degree $p^\N$, we know that there is never capitulation of $\CH_K \ne 1$.
Let's give two numerical examples, and use Program \ref{quadratic} given further:

\begin{example}{\rm
Consider $K=\Q(\sqrt{-199})$, $p=3$, $\ell = 19$, inert in $K$:

\ft\begin{verbatim}
p=3  Nn=2  PK=x^2+199  CK0=[9]  ell=19  r=1
CK1=[513]=[27]
h_1^[(S-1)^1]=[9]       h_1^[(S-1)^2]=[0] 
norm in K1/K of the component 1 of CK1:[3]
No capitulation, m(K1)=2, e(K1)=3
\end{verbatim}\ns
\ft\begin{verbatim}
CK2=[749493,19,19]=[81]
h_1^[(S-1)^1]=[45,0,0]  h_1^[(S-1)^2]=[0,0,0]   
norm in K2/K of the component 1 of CK2:[9]
No capitulation, m(K2)=2, e(K2)=4
\end{verbatim}\ns

(i) Case $n=1$. We have $\CH_{K_1} \simeq \Z/3^3\Z$, $m(\Kk_1)=2$ 
($s(\Kk_1) = 1$) and  one obtains  $e(\Kk_1)=3 > n - s(\Kk_1) = 0$.
We get the equality $\order \CH_{K_1} = \order \CH_K \cdot \order \CH_K[3]$.

(ii) Case $n=2$. 
Then $\CH_{K_2} \simeq \Z/3^4\Z$, $m(\Kk_2) = 2$ ($s(\Kk_2) = 1$) 
and one obtains $e(\Kk_2)=4 > n - s(\Kk_2) = 1$.
We have $\order \CH_{K_2} = \order \CH_{K_1} \cdot \order \CH_{K_1}[3] =
\order \CH_K \cdot \order \CH_K[3^2]$.}
\end{example}

\begin{example}{\rm
We consider $K=\Q(\sqrt{-199})$, $\ell = 37$, inert in $K$.

\ft\begin{verbatim}
p=3  Nn=2  PK=x^2+199  CK0=[9]  ell=37  r=1 
CK1=[54,6,3]=[27,3,3]
h_1^[(S-1)^1]=[21,1,0] h_2^[(S-1)^1]=[18,0,1] h_3^[(S-1)^1]=[18,0,0]
h_1^[(S-1)^2]=[0,0,1]  h_2^[(S-1)^2]=[18,0,0] h_3^[(S-1)^2]=[0,0,0]
h_1^[(S-1)^3]=[18,0,0] h_2^[(S-1)^3]=[0,0,0]  h_3^[(S-1)^3]=[0,0,0]
h_1^[(S-1)^4]=[0,0,0]  h_2^[(S-1)^4]=[0,0,0]  h_3^[(S-1)^4]=[0,0,0]
norm in K1/K of the component 1 of CK1:[12,0,1]
norm in K1/K of the component 2 of CK1:[18,0,0]
norm in K1/K of the component 3 of CK1:[0,0,0]
No capitulation, m(K1)=4, e(K1)=3 
CK2=[42442542,18,9]=[81,9,9]
No capitulation, m(K2)=4, e(K2)=4
\end{verbatim}\ns

(i) Case $n=1$. 
In this case, $\CH_{K_1} \simeq \Z/3^3\Z \times \Z/3\Z \times \Z/3\Z$; the above 
data shows that $m(\Kk_1) = 4$ for $e(\Kk_1)=3$, and a more complex structure, 
since $s(\Kk_1)=1$, but $e(\Kk_1) = 3 > n - s(\Kk_1) = 0$. Here,
$\order \CH_{K_1} > \order \CH_K \cdot \order \CH_K[3]$.

(ii) Case $n=2$. 
Then $\CH_{K_2} \simeq \Z/3^4\Z \times \Z/3^2\Z \times \Z/3^2\Z$,
$m(\Kk_2) \geq 4$, $s(\Kk_2) \geq 1$ and $e(\Kk_2)=4 > n - s(\Kk_2)$.
One has $\order \CH_{K_2} = \order \CH_{K_1} \cdot \order \CH_{K_1}[3]$.}
\end{example}

\subsection{Torsion groups of abelian \texorpdfstring{$p$}{Lg}-ramified 
\texorpdfstring{$p$}{Lg}-extensions}

We will evoke the case of the torsion group $\CT_K$ of the Galois group of the 
maximal abelian $p$-ramified pro-$p$-extension of a number field $K$; then,
under Leopoldt's conjecture, {\it the transfer map is always injective, whatever 
the extensions of number fields $L/K$ considered} (see, e.g., for more information 
and explicit results, our Crelle's papers (1982), Nguyen Quang Do \cite{NQD1986} (1986), 
Jaulent \cite{Jaul1986} (1986), Movahhedi \cite{Mova1988} (1988), all written in french, 
collected in our book \cite[Chapter IV]{Gras2005}). 

This has some consequences because of the formula:
\begin{equation}\label{orderT}
\order \CT_K = \order \wt \CH_K \cdot  \order \CR_K \cdot \order \CW_K,
\end{equation}

\noindent
where $\CW_K$ is a canonical invariant built on the groups of (local and global) 
roots of unity of $p$-power order of $K$, $\CR_K$ is the normalized $p$-adic 
regulator and $\wt \CH_K$ a sub-group of $\CH_K$ (see, for instance 
\cite[Theorem IV.2.1]{Gras2005}, \cite[Diagram \S\,3 and \S\,5]{Gras2018}, 
\cite[Section 2]{Gras2021$^a$}). For the Bertrandias--Payan module $\CT_K^\bp$ 
of $K$ (isomorphic to $\CT_K/\CW_K$, from \cite{BerPay} about 
the embedding problem), the transfers $\J_{L/K}$ are injective, except few 
special cases discussed in Gras--Jaulent--Nguyen Quang Do \cite{GJN2016}.

Thus, as we have seen, in any cyclic $p$-extension $L/K$ of Galois group 
$G$, the complexity of the invariants $\CT_{K_n} \ne 1$ is never smooth and is
increasing with $n$.

\subsubsection{Quadratic fields}
We use our program \cite[Corollary 2.2, Program I, \S\,3.2]{Gras2019$^a$}
computing the group structure of the $\CT_{K_n}$'s for quadratic fields 
$K = \Q(\sqrt {s m})$, $s=\pm1$, $p=2$, $K_n \subseteq L \subset K(\mu_\ell^{})$ 
(so the number $r_2+1$ of independent $\Z_p$-extensions is $1$ for $s=1$ and 
$2^n+1$ for $s=-1$). One must chose an arbitrary constant $E$, ``assuming'' 
$E > e_n + 1$, to be controlled a posteriori; but taking $E$ very large does 
not essentially modify the computer calculation time:

\ft\begin{verbatim}
MAIN PROGRAM COMPUTING THE STRUCTURE OF TKn (quadratic fields).
Set s=1 (resp. s=-1) for real (resp. imaginary) quadratic fields:
{s=1;p=2;ell=257;Nn=4;E=16;for(m=2,150,if(core(m)!=m,next);
PK=x^2-s*m;print();print("p=",p," ell=",ell," PK=",PK);
for(n=0,Nn,r2=1;if(s==-1,r2=r2+2^n);Qn=polsubcyclo(ell,p^n);
Pn=polcompositum(PK,Qn)[1];Kn=bnfinit(Pn,1);Knmod=bnrinit(Kn,p^E);
CKnmod=Knmod.cyc;TKn=List;d=matsize(CKnmod)[2];for(j=1,d-r2,c=CKnmod[d-j+1];
w=valuation(c,p);if(w>0,listinsert(TKn,p^w,1)));print("TK",n,"=",TKn)))}
\end{verbatim}\ns

In a very simple context ($K = \Q(\sqrt m)$, $p=2$, $L \subset K(\mu_{257}^{})$), the 
complexity of the torsion groups $\CT_{K_n}$ is growing dramatically 
(for the $p^k$-ranks as well as the exponents) as shown by the following excerpts:

\ft\begin{verbatim}
p=2  ell=257  PK=x^2-2
TK0=[]
TK1=[8,8]
TK2=[16,16,4,4,2,2]
TK3=[32,32,8,8,4,4,4,2,2,2,2]
TK4=[64,64,16,16,4,4,4,4,4,4,4,4,4,4,4,2,2,2,2]
\end{verbatim}\ns
\ft\begin{verbatim}
p=2  ell=257  PK=x^2-73
TK0=[2]
TK1=[64,8,2,2]
TK2=[128,16,8,4,4,2,2,2]
TK3=[256,32,16,16,8,8,8,4,2,2,2,2,2,2,2,2]
TK4=[512,64,32,32,16,16,16,8,2,2,2,2,2,2,2,2,2,2,2,2,2,2,2,2,2,2,2,2,2,2,2,2]
\end{verbatim}\ns
\ft\begin{verbatim}
p=2  ell=257  PK=x^2-105
TK0=[2,2]
TK1=[16,8,2,2]
TK2=[32,16,8,4,2,2,2,2]
TK3=[64,32,8,8,8,8,8,2,2,2,2,2,2,2,2,2]
TK4=[128,64,16,16,16,16,16,2,2,2,2,2,2,2,2,2,2,2,2,2,2,2,2,2,2,2,2,2,2,2,2,2]
\end{verbatim}\ns
\ft\begin{verbatim}
p=2  ell=257  PK=x^2-113
TK0=[4]
TK1=[128,16,4]
TK2=[256,32,8,4,4,4,2]
TK3=[512,64,16,8,8,8,2,2,2,2,2,2,2,2,2]
TK4=[1024,128,32,16,16,16,2,2,2,2,2,2,2,2,2,2,2,2,2,2,2,2,2,2,2,2,2,2,2,2,2]
\end{verbatim}\ns

In the context $L \subset K(\mu_\ell^{})$, the analogue of the Chevalley--Herbrand 
formula is $\order \CT_{K_n}^{G_n} = \order \CT_K\! \cdot p^{r\, n}$,
where $r$ is the number of primes ${\mathfrak l} \mid \ell$
in $K$ \cite[Theorem IV.3.3, Exercise 3.3.1]{Gras2005}; unfortunately, 
we do not know formulas, similar to that of \eqref{filtration}, 
for the orders of the $\CT_{K_n}^{i+1}/\CT_{K_n}^i$ for $i \geq 1$, hence 
$m(\Kk_n)$ is unknown. 

Consider, for instance, the above case of $m = 113$, for $K_4$, 
$\order \CT_{K_4} = 2^{59}$, $r=2$, $\Nu_{\!{K_4}/K}(\CT_{K_4})\simeq \Z/4\Z$.
Since $\order(\CT_{K_4}^{i+1}/\CT_{K_4}^i) \leq \order \CT_{K_4}^{G_4} 
= 2^{10}$, this gives $m(\Kk_4) \geq 6$ ($s(\Kk_4) \geq 2$). The conditions 
$e(\Kk_4) =10 \leq 4-s(\Kk_4)$ can not be satisfied.

Taking imaginary quadratic fields ($s=-1$) does not modify the behavior of 
the $\CT_{K_n}$'s since, for all $n$, $\J_{K_n/K}$ is still injective and 
$\Norm_{K_n/K}$ surjective; but in the imaginary quadratic case, the 
normalized regulator is trivial and $\CT_{K_n}/\CW_{K_n}$ is isomorphic to
a subgroup of $\CH_{K_n}$, for which we know the non-smooth complexity.
For instance, for $K=\Q(\sqrt{-2})$, since $2$ ramifies in $K$ and totally splits 
in $K_4/K$, we have $\CW_{K_n} \simeq (\Z/2\Z)^{2^n-1}$, $n \in [0, 4]$
($(\Z/2\Z)^{15}$ for $n=4$), and the following $\CH_{K_n}$:

\ft\begin{verbatim}
p=2  ell=257  PK=x^2+2
CK0=[] 
CK1=[24]=[8]
CK2=[240,4,2]=[16,4,2]
CK3=[19680,8,2,2,2,2,2]=[32,8,2,2,2,2,2]
CK4=[11375040,272,2,2,2,2,2,2,2,2,2,2,2,2,2]=[64,16,2,2,2,2,2,2,2,2,2,2,2,2,2]
\end{verbatim}\ns

\noindent
that we may compare with the structure of $\CT_{K_n}$ for  $K=\Q(\sqrt{-2})$:

\ft\begin{verbatim}
p=2  ell=257  PK=x^2+2
TK0=[]                                    WK0=[]
TK1=[16,2]                                WK1=[2]
TK2=[16,4,4,4,4]                          WK2=[2,2,2]
TK3=[32,8,4,4,4,4,4,4,4]                  WK3=[2,2,2,2,2,2,2,]
TK4=[64,16,4,4,4,4,4,4,4,4,4,4,4,4,4,4,4] WK4=[2,2,2,2,2,2,2,2,2,2,2,2,2,2,2]
\end{verbatim}\ns

\noindent
showing that in the formula \eqref{orderT}, 
$\order \CR_{\!K_4} = 2^2 \cdot [\wt K_4 \cap H_{K_4}^\nr : K_4]$.

Similar analysis may be done with the following fields $K$:
\ft\begin{verbatim}
p=2  ell=257  PK=x^2+3
TK0=[]
TK1=[16]
TK2=[32,8,4]
TK3=[64,16,4,4,4,4,4]
TK4=[128,32,4,4,4,4,4,4,4,4,4,4,4,4,4]
\end{verbatim}\ns
\ft\begin{verbatim}
p=2  ell=257  PK=x^2+7
TK0=[2]
TK1=[8,2,2]
TK2=[16,4,2,2,2,2,2]
TK3=[32,8,2,2,2,2,2,2,2,2,2,2,2,2,2]
TK4=[64,16,2,2,2,2,2,2,2,2,2,2,2,2,2,2,2,2,2,2,2,2,2,2,2,2,2,2,2,2,2]
\end{verbatim}\ns

This leads to the interesting problem of estimating the maximal unramified
sub-extension of the compositum $\wt {K_n}$ of the $\Z_p$-extensions of the $K_n$'s.

\subsubsection{Cyclic cubic fields}
For cyclic cubic fields, $p=2$, we obtain analogous computations
by means of the following program:

\ft\begin{verbatim}
MAIN PROGRAM COMPUTING THE STRUCTURE OF TKn (cyclic cubic fields):
{p=2;ell=257;Nn=3;E=16;bf=7;Bf=10^3;
for(f=bf,Bf,h=valuation(f,3);if(h!=0 & h!=2,next);F=f/3^h;
if(core(F)!=F,next);F=factor(F);Div=component(F,1);d=matsize(F)[1];
for(j=1,d,D=Div[j];if(Mod(D,3)!=1,break));for(b=1,sqrt(4*f/27),
if(h==2 & Mod(b,3)==0,next);A=4*f-27*b^2;if(issquare(A,&a)==1,
if(h==0,if(Mod(a,3)==1,a=-a);PK=x^3+x^2+(1-f)/3*x+(f*(a-3)+1)/27);
if(h==2,if(Mod(a,9)==3,a=-a);PK=x^3-f/3*x-f*a/27);print("p=",p," f=",f,
" PK=",PK," ell=",ell);for(n=0,Nn,Qn=polsubcyclo(ell,p^n);
Pn=polcompositum(PK,Qn)[1];Kn=bnfinit(Pn,1);Knmod=bnrinit(Kn,p^E);
CKnmod=Knmod.cyc;TKn=List;d=matsize(CKnmod)[2];for(j=1,d-1,c=CKnmod[d-j+1];
w=valuation(c,p);if(w>0,listinsert(TKn,p^w,1)));print("TK",n,"=",TKn)))))}
\end{verbatim}\ns
\ft\begin{verbatim}
p=2  ell=257  f=31  PK=x^3+x^2-10*x-8
TK0=[2,2]
TK1=[8,2,2,2,2]
TK2=[16,4,2,2,2,2,2,2,2,2,2]
TK3=[32,8,2,2,2,2,2,2,2,2,2,2,2,2,2,2,2,2,2,2,2,2,2]
\end{verbatim}\ns
\ft\begin{verbatim}
p=2  ell=257  f=43  PK=x^3+x^2-14*x+8 
TK0=[2,2]
TK1=[16,16,8,2,2]
TK2=[32,32,16,4,4,4,2,2,2,2,2]
TK3=[64,64,32,8,8,8,2,2,2,2,2,2,2,2,2,2,2,2,2,2,2,2,2]
\end{verbatim}\ns
\ft\begin{verbatim}
p=2  ell=257  f=171  PK=x^3-57*x-152 
TK0=[8,8]
TK1=[16,16,8,2,2]
TK2=[32,32,16,4,2,2,2,2,2,2,2]
TK3=[64,64,32,8,2,2,2,2,2,2,2,2,2,2,2,2,2,2,2,2,2,2,2]
\end{verbatim}\ns
\ft\begin{verbatim}
p=2  ell=257  f=277  PK=x^3+x^2-92*x+236
TK0=[4,4]
TK1=[8,4,4,4,4]
TK2=[16,8,8,4,4,4,4,4,2,2,2]
TK3=[32,16,16,8,8,8,8,8,2,2,2,2,2,2,2,2,2,2,2,2,2,2,2]
\end{verbatim}\ns

To conclude about the spectacular increasing of the $p$-ranks, recall 
that the $p$-rank of the Tate--Chafarevich group
${\rm III}^2_L := {\rm Ker} \Big [{\rm H}^2(\wt {\mathcal G}_L,\F_p) \too
\bigoplus_{v \mid p}{\rm H}^2(\wt {\mathcal G}_{L_v},\F_p) \Big]$
($\wt {\mathcal G}_L=$  Galois group of the maximal $p$-ramified
pro-$p$-extension of $L$), is that of $\CT_L$.

\subsection{Remarks on \texorpdfstring{$\rk_p(\CX_n)$}{Lg}-ranks}

In the framework of Theorem \ref{nocap}, about the family $\{\CX_n\}_{n \geq 0}$
in a tower $L/K$, 
let $\CY_n:= \CX_n/\CX_n^p$, and assume that the transfer maps 
$\CY_n \to \CY_{n+1}$ are also injective; so, $\order \CY_{n+1} 
\geq \order \CY_n\! \cdot\! \order \CY_n[p]$ $= (\order \CY_n)^2$, whence 
$\rk_p(\CX_{n+1}) \geq 2\,\rk_p(\CX_n)$, for all $n \in [0, N-1]$. 
This ``doubling'' of the $p$-ranks does not seem exceptional; for instance,
for $\CX = \CT$, $K = \Q(\sqrt{105})$, $p=2$, $\ell = 257$, computed above:

\ft\begin{verbatim}
p=2  ell=257  PK=x^2-105  r=2
TK0=[2,2]
TK1=[16,8,2,2]
TK2=[32,16,8,4,2,2,2,2]
TK3=[64,32,8,8,8,8,8,2,2,2,2,2,2,2,2,2]
TK4=[128,64,16,16,16,16,16,2,2,2,2,2,2,2,2,2,2,2,2,2,2,2,2,2,2,2,2,2,2,2,2,2]
\end{verbatim}\ns

\noindent
we obtain precisely, for $r_n := \rk_p(\CT_n)$, $r_0=2$, $r_1=4$, 
$r_2=8$, $r_3=16$, $r_4=32$, and in the other similar examples, some 
irregularities appear.

The case $\CX = \CH$ and $K$ imaginary quadratic may give similar results:

\ft\begin{verbatim}
p=2  PK=x^2+5  CK0=[2]  ell=257  r=1
CK1=[6,6]=[2,2]
CK2=[174,6,2,2]=[2,2,2,2]
CK3=[19662,6,6,6,2,2,2,2]=[2,2,2,2,2,2,2,2]
CK4=[5353549698,42,6,6,2,2,2,2,2,2,2,2,2,2,2,2]=[2,2,2,2,2,2,2,2,2,2,2,2,2,2,2,2]
\end{verbatim}\ns

\noindent
for which $r_n = 2^n,\, \forall n \leq 4$. 

Note that, even if the maps $\CX_n \to \CX_{n+1}$ are injective, 
$\CY_n \to \CY_{n+1}$ may be non injective, meaning that some $x_n \in \CX_n$ 
become $p$th powers in $\CX_{n+1}$, which ``explains'' that the exponents $p^{e_n}$
increase with $n$ in many of the numerical examples and that the rule
$\rk_{n+1} \geq 2 \rk_n$ is not always fulfilled. One may hope that any pair 
$(p^{e_n}, r_n)$, compatible with Galois action, does exist.

This study suggests that classical lower bounds, given by genus theory in $L/K$ 
in the form of the genus exact sequence (e.g., Angl\`es--Jaulent \cite[Th\'eor\`eme 2.2.9]
{AnJau2000}, \cite[Th\'eor\`eme III.2.7]{Jaul1986}, Gras \cite[Corollary IV.4.5.1]{Gras2005}, 
Maire \cite[Th\-eorem 2.2]{Maire2018}, then Kl\"uners--Wang \cite{KluWan2022}, Liu 
\cite[Theorem 6.5]{Liu2022}), may be much largely exceeded and that the upper 
bounds may be reached.

\section{Capitulation in \texorpdfstring{$\Z_p$}{Lg}-extensions}

The problem of capitulations in a $\Z_p$-extension $\wt K = \bigcup_{n \geq 0} K_n$ 
of $K$ may be considered as a similar case of the previous tame context, taking
``$\ell = p$''. This problem has a long history from Iwasawa pioneering works showing, 
for instance, that the capitulation kernels $\Ker(\J_{\wt K/K_n})$ have a bounded order 
as $n \to \infty$ (Iwasawa \cite[Theorem 10, \S\,5]{Iwas1973}). One may refer for 
instance to Grandet--Jaulent \cite{GrJa1985}, Bandini--Caldarola 
\cite{BaCa2016,Cald2020} for classical context of $p$-class 
groups, to Kolster--Movahhedi \cite{KoMo2000}, Validire \cite{Vali2008} 
for wild kernels, and Jaulent \cite{Jaul2016,Jaul2019$^a$} for logarithmic class groups.

Let $\X_{\wt K} := \ds \limproj \CH_{K_n}$ (for the arithmetic norms) and
let $\ds \CH_{\wt K} := \limind \CH_{K_n}$ (for the transfer maps);
$\X_{\wt K}$ is isomorphic to the Galois group of the maximal unramified 
abelian pro-$p$-extension of $\wt K$ and $\CH_{\wt K}$ is the 
$p$-class group of $\wt K$.

\subsection{Known results under the assumption \texorpdfstring{$\mu = 0$}{Lg}}
If $\mu=0$, in the writing $\order \CH_{K_n} = p^{\lambda n + \mu p^n + \nu}$ 
for $n \gg 0$, the following properties are proved in Grandet--Jaulent
\cite[Th\'eor\`eme, p. 214]{GrJa1985}:

 $\bullet$  $\X_{\wt K} \simeq \T \plus \Z_p^\lambda$,  
where $\T$ is a finite $p$-group, 

 $\bullet$  $\Norm_{\wt K/K_n}\! : \X_{\wt K} \to \CH_{K_n}$
induces the isomorphisms $\T \simeq \Ker(\J_{\wt K/K_n}),\, \forall n \gg 0$,

 $\bullet$  $\CH_{K_n} \simeq \Ker(\J_{\wt K/K_n}) \bigoplus \J_{\wt K/K_n}(\CH_{K_n})
\simeq  \Ker(\J_{\wt K/K_n}) \bigoplus_{i=1}^\lambda \Z/p^{n+\alpha_i}\Z,\, \forall n \gg 0$,
with some relative integers $\alpha_i$.

In Validire \cite[Th\'eor\`eme 3.2.5]{Vali2008} is proved analogous results for even groups of the 
${\bf K}$-theory of rings of integers of number fields, after similar results as that
of Kolster--Movahhedi \cite{KoMo2000}.

From now on, we take the base field $K_{n_0}$, $n_0$ large enough,
in such a way that $\wt K/K_{n_0}$ is totally ramified and such that all the above 
properties are fulfilled from $n_0$. By abuse of notation, we write $K$ instead of
$K_{n_0}$ and $K_n$ now denotes $K_{n_0+n}$. So, the $\Z_p$-extension
$\wt K/K$ has Iwasawa invariants ($\lambda \geq 0$, $\mu = 0$, 
$\nu+\lambda n_0 \geq 0$) that we still denote ($\lambda$, $\nu$).
Thus, $\order \CH_K = p^\nu$, $\order \CH_{K_n} = p^{\lambda n + \nu}$ 
and, for all $n \geq 0$:
\begin{equation}\label{iso}
\left\{\begin{aligned}
\CH_{K_n} & \simeq \Ker(\J_{\wt K/K_n}) \oplus \oplus_{i=1}^\lambda \Z/p^{n+\alpha_i}\Z, \ \,
\alpha_i \geq 0,\ \Ker(\J_{\wt K/K_n}) \simeq \T\  \\
\CH_K &\simeq \Ker(\J_{\wt K/K}) \oplus \oplus_{i=1}^\lambda \Z/p^{\alpha_i}\Z,\ \, \alpha_i \geq 0. 
\end{aligned}\right.
\end{equation}

\begin{proposition}\label{cap}
Under the above choice of the base field $K$ in the $\Z_p$-exten\-sion $\wt K$ and 
assuming $\mu = 0$, the capitulation of $\CH_K$ in $\wt K$ is equivalent 
to the isomorphism $\CH_{K_n} \simeq \CH_K \oplus \J_{\wt K/K_n}(\CH_{K_n})
\simeq \CH_K \oplus \big( \Z/p^n\Z\big)^\lambda,\, \forall n \geq 0$.
\end{proposition}

\begin{proof}
If $\CH_K$ capitulates in $\wt K$, then $\Ker(\J_{\wt K/K}) = \CH_K$, whence
$\alpha_i=0$, for all  $i \in [1, \lambda]$, from \eqref{iso}, and
$\CH_{K_n} \simeq \CH_K \oplus \big( \Z/p^n\Z\big)^\lambda$ for all $n \geq 0$,
since each capitulation kernel $\Ker(\J_{\wt K/K_n})$ is isomorphic to $\T$, thus 
isomorphic to $\Ker(\J_{\wt K/K}) = \CH_K$. 

Reciprocally, assume that $\CH_{K_n} \simeq \CH_K \oplus 
\big( \Z/p^n\Z\big)^\lambda,\, \forall n \geq 0$; then, from \eqref{iso},
$\CH_{K_n} = \Ker(\J_{\wt K/K_n}) \oplus \plus_{i=1}^\lambda \Z/p^{n+\alpha_i}\Z 
\simeq \CH_K \oplus \big( \Z/p^n\Z\big)^\lambda$. 
Comparing the structures for $n$ large enough gives $\alpha_i = 0$, for all 
$i \in [1, \lambda]$ and $\Ker(\J_{\wt K/K_n}) \simeq \CH_K$, for all $n \geq 0$,
whence the capitulation of $\CH_K$ in $\wt K$. 
\end{proof}

\subsection{Case of the cyclotomic \texorpdfstring{$\Z_p$}{Lg}-extension 
of \texorpdfstring{$K$}{Lg}}

Assume that $K$ is totally real and let $K^\cyc  = \bigcup_{n \geq 0} K_n$ 
be the cyclotomic $\Z_p$-extension of $K$, assuming the previous choice
of the base field $K$ in $K^\cyc$ ($K$ is still real with same cyclotomic 
$\Z_p$-extension). One may also assume that Greenberg's conjecture \cite{Gree1976} 
($\lambda = \mu = 0$ for $K^\cyc$) is equivalent to the stability of the 
$\order \CH_{K_n}$'s from $K$, giving capitulations of all the class groups 
in $K^\cyc$ from $n=0$, then $\CH_{K_n} = \CH_{K_n}^{G_n}
\ds \mathop{\simeq}^{\hbox{\tiny$\Norm_{K_n/K}$}} \CH_K$, for all $n \geq 0$; 
thus, $m(\Kk_n)=1$ ($s(\Kk_n)=0$) with $e(\Kk_n) = e(\Kk)$, which is exactly 
the limit case of application of Theorem \ref{main1}, for $n \geq e(\Kk)$. 

In Kraft--Schoof--Pagani \cite{KrSch1995,Paga2022} such properties of stability are 
used to check the conjecture by means of analytic formulas. 

In Jaulent \cite{Jaul2016,Jaul2019$^a$,Jaul2019$^b$}, it is proved that Greenberg's conjecture 
is equivalent to the capitulation of the logarithmic class group $\CH_K^\lgm$ in 
$K^\cyc$;\,\footnote{This invariant, also denoted $\wt {C\!\ell}_K$ or $\wt \CT_K$, was
defined in \cite{Jaul1994} and is, in the class field theory viewpoint, isomorphic 
to $\Gal(H_K^\lc/K^\cyc)$, where $H_K^\lc$ is the maximal abelian locally cyclotomic 
pro-$p$-extension of $K$. In the sequel we will denote $\CH_K^\lgm$ this group since it 
behaves more like a class group rather than a torsion group $\CT_{\!K}$ which never capitulates.}
this may be effective if, by chance, a capitulation occurs in the firsts layers; indeed, this 
criterion is probably the only one giving an algorithmic test (using Belabas--Jaulent \cite{BeJa2016},
Diazy Diaz--Jaulent--Pauli--Pohst--Soriano-Gafiuk \cite{DJPPS2005}) from the base field. We will 
see, Section~\ref{clog} that $\CH^\lgm_K$ may capitulate in real towers $L/K$.

\subsubsection{Analysis of the Chevalley--Herbrand formula in $K^\cyc/K$}

\begin{hypothesis}\label{hypothesis}
Taking in the sequel, as totally real base field $K$, a suitable layer $K_{n_0}$ in $K^\cyc$, 
we may assume the following properties of $K^\cyc/K$:

(i) $p$ is totally ramified in $K^\cyc/K$;

(ii) $\order \CH_{K_n} = p^{\lambda n + \mu p^n + \nu}$ for all $n \geq 0$ with new 
non-negative invariants of the form ($\lambda$, $\mu p^{n_0}$, $\nu+\lambda n_0$) 
from that of $K$, still denoted ($\lambda$, $\mu$, $\nu$). Thus, $\order \CH_K = p^{\mu+\nu}$.
\end{hypothesis}

Then, as for the ``tame casee'', formulas \eqref{filtration} hold with 
$r = \order \{{\mathfrak p}, \  {\mathfrak p} \!\mid\! p \  \hbox{in $K$}\}$
and the filtration still depends on the class and norm factors. However, 
the norm factors can be interpreted as divisor of the normalized $p$-adic 
regulator of $K$, as follows from class field theory (under Leopoldt's conjecture): 

\begin{definitions}\label{diagram0}
(i) Consider $H_{K_n}^\gen$, the genus field of $K_n$ (i.e., the subfield of the
$p$-Hilbert class field $H_{K_n}^\nr$, abelian over $K$ and maximal, whence the 
subfield of $\Gal(H_{K_n}^\nr/K)$ fixed by the image of $\CH_{K_n}^{\sigma_n - 1}$), 
and $K^\cyc H_{K_n}^\gen$ which is abelian $p$-ramified over $K$; then 
$K^\cyc H_{K_n}^\gen \subseteq H_K^\pr$, the maximal $p$-ramified abelian 
pro-$p$-extension of $K$. So $\CT_K := \Gal(H_K^\pr/K^\cyc)$ is finite under 
Leopoldt's conjecture.

We define $H_{K^\cyc}^\gen := \bigcup_n K^\cyc H_{K_n}^\gen$ and put 
$\CG_K := \Gal(H_{K^\cyc}^\gen/K^\cyc)$.

We denote by $K_{n_1}$, $n_1 \geq 0$, the minimal layer such that 
$K^\cyc H_{K_{n_1}}^\gen = H_{K^\cyc}^\gen$ (even with the above 
Hypothesis \ref{hypothesis}, $K_{n_1}$ may be distinct from $K$).

(ii) Let $H_K^\bp$ be the Bertrandias--Payan field fixed by 
$\CW_K \simeq \big(\oplus_{v \mid p} \mu_{K_v}^{} \big) \big/\mu_K^{}$,
where $K_v$ is the $v$-completion of $K$ and $\mu_k^{}$  the group of 
$p$th-roots of unity of the field $k$ (local or global); if $U_v$ is the group of
principal units of $K_v$, then $\mu_{K_v}^{} = \tor_{\Z_p}^{}(U_v)$.

(iii) Let $\iota \CE_K$ be the image of $\CE_K\! = \BE_K \otimes \Z_p$ in 
$U_K\! := \prod_{v \mid p}U_v$ and let $I_v(H_K^\pr/K^\cyc)$
be the inertia groups of $v$ in $H_K^\pr/K^\cyc$; then
$I_v(H_K^\pr/K^\cyc) \simeq \tor_{\Z_p}^{}(U_v /\iota \CE_K \cap U_v )$ 
and  the subgroup of $\CT_K$ generated by these inertia groups fixes $H_{K^\cyc}^\gen$.

(iv) Let $\CR_K^\nr := \Gal(H_{K^\cyc}^\gen / K^\cyc H_K^\nr)\  \&\ 
\CR_K^\ram := \Gal(H_K^\bp/H_{K^\cyc}^\gen)$, where $\CR_K := 
\Gal(H_K^\bp / K^\cyc H_K^\nr)$ is the normalized $p$-adic regulator
that we define in \cite[\S\,5]{Gras2018}.

These definitions may be summarized by the following diagram 
\cite[\S\,2]{Gras2021$^a$}:
\unitlength=0.94cm
$$\vbox{\hbox{\hspace{-2.7cm} 
\begin{picture}(11.5,1.6) 
\bezier{550}(0.8,0.8)(7.0,1.8)(13.2,0.8)
\put(7.0,1.4){\ft $\CT_K$ \ns}
\bezier{350}(4.4,0.2)(7.9,-0.3)(11.1,0.2)
\put(8.0,-0.35){\ft $ \CR_K$ \ns}
\put(5.8,0.65){\ft $ \CR_K^\nr$ \ns}
\put(9.0,0.65){\ft $ \CR_K^\ram$ \ns}
\put(2.4,0.2){\ft $ \CH_K$ \ns}
% horizontales
\put(5.1,0.50){\line(1,0){2.0}}
\put(7.15,0.4){\ft $ H_{K^\cyc}^\gen$ \ns}
\put(7.8,0.50){\line(1,0){3.0}}
\put(11.5,0.50){\line(1,0){1.5}}
\put(10.85,0.4){\ft $ H_K^\bp$ \ns}
\put(12.0,0.2){\ft $ \CW_K$ \ns}
\put(1.2,0.50){\line(1,0){2.5}}
\put(3.8,0.4){\ft $ K^\cyc H_K^\nr$ \ns}
\put(0.6,0.4){\ft $ K^\cyc$ \ns}
\put(13.1,0.4){\ft $ H_K^\pr$ \ns}
\bezier{350}(0.9,0.2)(4.0,-0.7)(7.1,0.2)
\put(3.7,-0.52){\ft $\CG_K$ \ns}
\end{picture}   }} $$
\vspace{0.1cm}
\unitlength=1.0cm

\noindent
where $\order \CG_K = \order \CH_{K_n}^{G_n}$ for all $n \geq n_1$.
\end{definitions} 

Recall, under the above Hypothesis \ref{hypothesis}, some results that we 
have given in \cite{Gras2019$^b$}, generalizing some particular results of 
Taya \cite{Taya1996,Taya1999,Taya2000}, using $p$-adic $L$-functions:

\begin{proposition}\label{genus}
For all $n \geq 0$, the norm factor $\ffrac{p^{n \cdot (r -1)}}{\omega_{K_n/K}(\BE_K)}$ 
divides $\order \CR_K^\nr$ and there is equality for  all $n \geq n_1$. Whence
$\order \CH_{K_n}^{G_n} = 
 \order \CH_K \cdot \ffrac{p^{n \cdot (r -1)}}{\omega_{K_n/K}(\BE_K)}$ divides
$\order \CG_K = \order \CH_K \cdot \order \CR_K^\nr$, for all $n \geq 0$ with
equality for all $n \geq n_1$.
\end{proposition}

Recall that $m(\Kk_n)$ is the length of the filtration for $K_n$, with 
$m(\Kk_n) = 1$ if $\CH_{K_n} \ne 1$ ($m(\Kk_n) = 0$ if $\CH_{K_n} = 1$):

\begin{proposition}\label{inequalities}
Let $v_p$ be the $p$-adic valuation. Under Hypothesis \ref{hypothesis}, we have:
$$m(\Kk_n) \leq \lambda \cdot n + \mu \cdot p^n + \nu \leq 
v_p(\order \CH_K \cdot \order \CR_K^\nr) \cdot m(\Kk_n), \ \, \forall n \geq 0. $$ 
If $\CG_K=1$, then Greenberg's conjecture holds with $\CH_{K_n}=1$
for all $n$. If $\CG_K \ne 1$, thenGreenberg's conjecture holds if and only if
$m(\Kk_n)$ is bounded regarding $n$.
\end{proposition}

From these recalls and Hypothesis \ref{hypothesis}, we can deduce:

\begin{theorem}\label{main3}
Greenberg's conjecture is equivalent to $m(\Kk_n) = 1$ (resp. $=0$) for all $n \geq 0$,
if $\CH_K \ne 1$ (resp. $=1$), and $\CR_K^\nr = 1$, whence Greenberg's conjecture 
is equivalent to $\CH_{K_n} = \CH_{K_n}^{G_n} 
\ds \mathop{\simeq}^{\hbox{\tiny$\Norm_{K_n/K}$}} \CH_K$ for all $n \geq 0$. 
\end{theorem}

\begin{proof}
If $\lambda = \mu =0$, the stability holds from $n = 0$ and $\CH_{K_n} = \CH_{K_n}^{G_n} 
\ds \mathop{\simeq}^{\hbox{\tiny$\Norm_{K_n/K}$}} \CH_K$, for all $n \geq 0$
(whence $m(\Kk_n) \in \{0, 1\}$, depending on $\order \CH_K=1$ or not). 
Since $\order \CH_{K_n}^{G_n} = \order \CG_K = \order \CH_K,\, 
\forall n \geq n_1$, it follows that $\CR_K^\nr = 1$ (Proposition \ref{genus}).
Reciprocally, if $\CR_K^\nr = 1$ and $m(\Kk_n) \in \{0, 1\}$, then 
Proposition \ref{inequalities} implies $\lambda = \mu = 0$ if $\CH_K \ne 1$ 
(or $\lambda = \mu = \nu = 0$ if $\CH_K = 1$).
\end{proof}

\subsubsection{Particular case $\mu = 0$}
We can wonder, due to the Propositions \ref{cap} and \ref{inequalities}, if Greenberg's conjecture 
may be equivalent to $\Nu_{\!K_n/K}(\CH_{K_n}) = 1$, for all $n \geq e(\Kk)$, obtained with the 
stronger particular conditions $m(\Kk_n) \in \{0,1\}$ (i.e., $s(\Kk_n)=0$) and $e(\Kk_n) = e(\Kk)$
for all $n \geq 0$. 

\noindent
Indeed, Greenberg's conjecture implies the capitulation of $\CH_K$ and
$m(\Kk_n) = 1$ for all $n \geq 0$ (Theorem \ref{main3}); if so, one obtains
$\CH_{K_n} = \CH_{K_n}^{G_n} \ds \mathop{\simeq}^{\hbox{\tiny$\Norm_{K_n/K}$}} \CH_K$, 
for all $n \geq 0$; in some sense, the maximal smoothness complexity.

In the practice these phenomenon (if any) holds from an unknown level in $K^\cyc$.

This possibility may be suggested by the following example of the cyclic cubic field of 
conductor $f=2689$, of $2$-class group $\BZ/2\BZ$ and its cyclotomic $\Z_2$-extension, 
giving $\CH_{K_1} \simeq \BZ/4\BZ$, $\CH_{K_2} \simeq \BZ/8\BZ$, and $\CH_{K_3} 
\simeq \BZ/8\BZ$:

\ft\begin{verbatim}
p=2  f=2689  PK=x^3+x^2-896*x+5876  CK0=[2,2] 
\end{verbatim}\ns
\ft\begin{verbatim}
CK1=[28,4]=[4,4]
h_1^[(S-1)^1]=[0,0]    h_2^[(S-1)^1]=[0,0]
norm in K1/K of the component 1 of CK1:[2,0]
norm in K1/K of the component 2 of CK1:[0,2]
No capitulation,  m(K1)=1, e(K1)=2
CK2=[56,8]=[8,8]
h_1^[(S-1)^1]=[0,0]    h_2^[(S-1)^1]=[0,0]
norm in K2/K of the component 1 of CK2:[4,0]
norm in K2/K of the component 2 of CK2:[0,4]
No capitulation, m(K2)=1, e(K2)=3
CK3=[56,8]=[8,8]
h_1^[(S-1)^1]=[0,0]    h_2^[(S-1)^1]=[0,0]
norm in K3/K of the component 1 of CK3:[0,0]
norm in K3/K of the component 2 of CK3:[0,0]
Complete capitulation, m(K3)=1, e(K3)=3
\end{verbatim}\ns

At any layer, $m(\Kk_n)=1$ and capitulations in $K^\cyc$ hold for all $n$. 
Stability occurs from $K_2$ giving a checking of Greenberg's conjecture.

\begin{remark} 
It is interesting to check other examples given by Fukuda \cite{Fuku1994} for real 
quadratic fields and $p=3$ with a simplified program and suitable polynomials 
defining the layer $K_1$; for all, $\CH_K \simeq \Z/9\Z$ and 
$\Norm_{K_1/K}(\CH_{K_1}) = \CH_K^3$:

\ft\begin{verbatim}
{p=3;Lm=List([3137,3719,4409,6809,7226,9998]);for(i=1,6,PK=x^2-Lm[i];
K=bnfinit(PK,1);CK0=K.clgp;P1=polcompositum(PK,polsubcyclo(p^2,p))[1];
K1=bnfinit(P1,1);CK1=K1.clgp;print();G=nfgaloisconj(K1);Id=x;for(k=1,2*p,
Z=G[k];ks=1;while(Z!=Id,Z=nfgaloisapply(K1,G[k],Z);ks=ks+1);if(ks==p,
S=G[k];break));A0=CK1[3][1];A=1;for(t=1,p,As=nfgaloisapply(K1,S,A);
A=idealmul(K1,A0,As));C=bnfisprincipal(K1,A)[1];print("PK=",PK,
" CK0=",CK0[2]," CK1=",CK1[2]," Norm of CK1=",C))}
PK=x^2-3137  CK1=[9]  Norm of CK1=[3]    PK=x^2-6809  CK1=[9]  Norm of CK1=[3]
PK=x^2-3719  CK1=[9]  Norm of CK1=[3]    PK=x^2-7226  CK1=[18] Norm of CK1=[3]
PK=x^2-4409  CK1=[9]  Norm of CK1=[3]    PK=x^2-9998  CK1=[9]  Norm of CK1=[3]
\end{verbatim}\ns
\noindent
The stability from $K$ implies Greenberg's conjecture and capitulation
of $\CH_K$ in $K_2$, with an incomplete capitulation in $K_1$ and
$m(\Kk_n)=1$ for all $n$.
\end{remark}

\section{Isotopic components and capitulation}
Consider a real cyclic field $K$ of prime-to-$p$ degree $d$ and 
$L = L_0 K$ with $L_0/\Q$ real cyclic of degree $p^\N$, $N \geq 1$. 
Then $L/\Q$ is cyclic of degree $d p^\N$ with Galois group
$\Gamma = g \oplus G$ where $g = \Gal(L/L_0)$ and $G = \Gal(L/K)$.
The field $L$ is associated to an irreducible rational character $\chi$,
sum of irreducible $p$-adic characters $\varphi$ of ``order'' 
the order $d p^\N$ of any $\psi \mid \varphi$ of degree $1$.

This non semi-simple context is problematic for the definition of isotopic 
$p$-adic components of the form $\CH_{L,\varphi}$ and $\CH_{K_n,\varphi_n}$ 
for the subfields $K_n$ of $L$ with corresponding rational and $p$-adic characters 
$\chi_n$ and $\varphi_n \mid \chi_n$; this is extensively developed in the english
translation \cite{Gras2021$^b$} of our original paper (1978)  in french
\url{https://doi.org/10.5802/pmb.a-10}. 
So we just recall the definitions and explain how the phenomenon of capitulation 
gives rise to difficulties about the classical algebraic definition of the literature, 
compared to the arithmetic one that we have introduced to state the Main 
Conjecture in the non semi-simple case.

Indeed, classical works deal with an algebraic definition of the $\varphi$-components 
of $p$-class groups, which presents an inconsistency regarding analytic formulas; 
this definition is, for $\Gamma$ cyclic of order $d p^\N$ and for all $\varphi \mid \chi$:
$$\CH^\alg_{L,\varphi} := \CH_L \ \hbox {$\bigotimes^{}_{\Z_p[\Gamma]}$} \ 
\Z_p[\mu_{d p^\N}^{}], $$

\noindent
with the $\Z_p[\mu_{d p^\N}^{}]$-action $\tau \in \Gamma \mapsto \psi (\tau)$, with
$\psi \mid \varphi$ of order $d p^\N$ (see Solomon \cite[II, \S\,1]{Solo1990} or 
Greither \cite[Definition, p. 451]{Grei1992}). Let $P_\chi$ be the $d p^\N$th cyclotomic 
polynomial and let $\sigma_\chi$ be a generator of $\Gal(L/\Q)$; then,
let $P_\varphi \mid P_\chi$ be the corresponding local cyclotomic polynomial
associated to the above action $\tau \in \Gamma \mapsto \psi (\tau)$, with
$\psi \mid \varphi$. We have defined the notions of $\chi$ and $\varphi$-objects,
then proved in \cite[\S\,3.2.4, Theorem 3.7, Definition 3.11]{Gras2021$^b$}, the 
following interpretations:
\begin{equation*}
 \hspace{0.55cm}\left \{\begin{aligned}
\CH^\alg_{L,\chi} =  & \ \{x \in \CH_L, \ P_\chi (\sigma_\chi) \cdot x = 1 \} =
 \{x \in \CH_L, \ \Nu_{\!L/k}(x) = 1,\, \forall \, k \varsubsetneqq L \} , \\
\CH^\alg_{L,\varphi} = &\ \{x \in \CH_L, \ P_\varphi (\sigma_\chi) \cdot x = 1\}, 
\end{aligned}\right.
\end{equation*}

\noindent
which gives rise to our corresponding arithmetic definitions:
\begin{equation*}
 \hspace{-0.55cm}\left \{\begin{aligned}
\CH^\ar_{L,\chi} := &\  \{x \in \CH_L,\ \,\Norm_{L/k}(x) = 1,\, \forall \, k \varsubsetneqq L \}, \\
\CH^\ar_{L,\varphi} := &\  \{x \in \CH_L, \ \, P_\varphi (\sigma) \cdot x = 1\ \ \ \& \ \ \
\Norm_{L/k}(x) = 1,\, \forall \, k \varsubsetneqq L \} 
\end{aligned}\right.
\end{equation*}

\noindent
(in other words, $\CH^\ar_{L,\varphi} = \CH^\ar_{L,\chi} \cap \CH^\alg_{L,\varphi}$).

We then have, since $L/K$ is totally ramified (\cite[Theorem 3.15]{Gras2021$^b$}):
\begin{equation}\label{isotopic}
\order \CH^\ar_{L} = \prd_{\chi \in {\bf R}_L} \order \CH^\ar_{L,\chi} ,
\end{equation}
where ${\bf R}_L$ is the set of irreducible rational 
characters of $L$. More precisely, $\chi$ is of the form $\chi_0^{} \, \chi_n^{}$ 
for a rational characters $\chi_0^{}$ of $K$, of order a divisor od $d$, and 
the rational characters $\chi_n^{}$ 
of $L_0$ of order $p^n$, $n \in [1, N]$; then $\CH^\ar_{L,\chi} = 
\CH^\ar_{K_{\chi_0^{}} L_{0,\chi_n^{}},\chi}$,
where $K_{\chi_0^{}} \subseteq K$ (resp. $L_{0,\chi_n^{}} \subseteq L_0$) correspond 
to $\chi_0^{}$ (resp. $\chi_n^{}$).

These definitions and results lead to an unexpected semi-simplicity, especially in 
accordance with analytic formulas, which enforces the Main Conjecture in that~case
\cite[Theorem 4.5]{Gras2021$^b$}; it writes, for all $\chi \in {\bf R}_L$:
\begin{equation}\label{isotopicphi}
\CH^\ar_{L,\chi} = \plus_{\varphi \mid \chi} \CH^\ar_{L,\varphi},.
\end{equation}

The $\Z_p[\Gamma]$-modules of the form $\CH^\ar_{L,\chi}$ (resp. 
$\CH^\ar_{L,\varphi}$), annihilated by all the arithmetic norms $\Norm_{L/k}$,
are called arithmetic $\chi$-objects (resp. $\varphi$-objects).

We have $\CH^\ar_{L,\varphi} = \CH_{L,\varphi}^\alg$ as soon as the $\J_{L/k}$'s 
are injective for all $k \varsubsetneqq L$, but as we have seen, this does
not hold in general when $K \subseteq k \varsubsetneqq L$ since
there is often partial capitulation. One can even say that the classic 
admitted definition is non canonical and imperfect for real 
$p$-class groups (see Examples below).

\begin{remark}\label{remaconj}
Let $\chi$ be the rational character associated to $L$.
Our Main Conjecture \cite{Gras1977} (not yet proven in the non semi-simple 
case contrary to some claims) requires that the equality of orders of 
$\chi$-objects (see \cite[Theorem 7.5\,(i)]{Gras2021$^b$}):
$\order \CH^\ar_{L,\chi^{}} = \order (\CE_{L} /\CE^0_{L} \cdot \CF_{\!L})$,  
be valid for the $\varphi$-components,
for all $\varphi \mid \chi$; in these formula, $\CE^0_{L}$ is the 
subgroup of $\CE_{L}$ generated by the units of the strict subfields of $L$ 
and $\CF_{\!L}$ is the group of classical Leopoldt's cyclotomic units;
the fact that $\CE_{L} /\CE^0_{L} \cdot \CF_{\!L}$ be a $\chi$-object is
obvious since $\Norm_{L/K}(\CE_L) \subseteq \CE_L^0$ by definition
(see \cite[Examples 3.12, 3.13]{Gras2021$^b$}).
\end{remark}

In the case of cubic fields with $p=2$ or in the case of real quadratic fields,
$\chi = \varphi$, so that the Main Conjecture is trivial, but not the definition
of arithmetic $\varphi$-objects regarding the algebraic ones.
Let's give numerical examples showing the consequences of capitulation
for the non-arithmetic definitions:

\begin{example}\label{Ex1}
{\rm  Consider $K = \Q(\sqrt{4409})$, $p=3$, $\ell = 19$ split in $K$ and 
$K_2 \subset K(\mu_\ell^{})$.
We will see that $\CH_{K_n} \simeq \Z/9\Z,\, \forall n \leq 2$ (stability);
Program \ref{quadratic} gives:

\ft\begin{verbatim}
p=3  PK=x^2-4409  CK0=[9]  ell=19  r=2
CK1=[9]                            CK2=[9] 
norm in K1/K of CK1:[3]            norm in K2/K of CK2:[0]
Incomplete capitulation in K1      Complete capitulation in K2
\end{verbatim}\ns

The capitulation is complete in $K_2$ as expected from Theorem \ref{main2}
(stability from~$K$ giving  $\CH_{K_n} = \CH_{K_n}^{G_n} 
\ds \mathop{\simeq}^{\hbox{\tiny$\Norm_{K_n/K}$}} \CH_K$, 
for $n \in \{1, 2 \}$).

We use obvious notations for the characters defining the fields $K_n$.
Since arithmetic norms are isomorphisms, the above computations prove that:
\begin{equation*}
\hspace{-1.3cm}\left\{\begin{aligned}
\Nu_{\!K_1/K}(\CH_{K_1}) &= (\CH_{K_1})^{1+\sigma_1+\sigma_1^2}
= (\CH_{K_1})^3  \simeq \Z/3\Z, \\
\Nu_{\!K_2/K_1}(\CH_{K_2}) &= (\CH_{K_2})^{1+\sigma_2^3+\sigma_2^6}
= (\CH_{K_2})^3  \simeq \Z/3\Z, 
\end{aligned}\right.
\end{equation*}

\noindent
whence:
\begin{equation*}
 \hspace{-0.6cm}\left\{\begin{aligned}
\CH^\ar_{\chi^{}_2} & = \{x \in \CH_{K_2},\ \Norm_{K_2/K_1}(x)=1 \} = 1, \\
\CH^\alg_{\chi^{}_2} & = \{x \in \CH_{K_2},\  \Nu_{\!K_2/K_1}(x)=1\} = 
\CH_{K_2}[3] \simeq \Z/3\Z. \\
\CH^\ar_{\chi^{}_1} & = \{x \in \CH_{K_1},\ \Norm_{K_1/K}(x)=1\} =1, \\
\CH^\alg_{\chi^{}_1} & = \{x \in \CH_{K_1},\  \Nu_{\!K_1/K}(x)=1\} = \CH_{K_1}[3] \simeq \Z/3\Z.
\end{aligned}\right.
\end{equation*}

Formula \eqref{isotopic} (or \eqref{isotopicphi} since $\chi = \varphi$), gives the product 
of orders of the $\chi$-components $\CH^\ar_{\chi_n^{}}$ (since $\order \CH^\ar_{\chi_0^{}} 
= \order \CH_{K^{}} = 9$):
\begin{equation*}
\hspace{-2.55cm}\left\{\begin{aligned}
\order \CH_{K_2} &= \order \CH^\ar_{\chi_0^{}} \cdot \order \CH^\ar_{\chi_1^{}} 
\cdot \order \CH^\ar_{\chi_2^{}} = 9 \cdot 1 \cdot 1 = 3^2, \\
\order \CH_{K_1} &= \order \CH^\ar_{\chi_0^{}} \cdot \order \CH^\ar_{\chi_1^{}}
= 9 \cdot 1 = 3^2,
\end{aligned}\right.
\end{equation*}

These formulas are not fulfilled in the algebraic sense, because:
\begin{equation*}
\hspace{-2.8cm}\left\{\begin{aligned}
&\order \CH^\alg_{\chi^{}_0} \cdot \order \CH^\alg_{\chi^{}_1} \cdot 
\order \CH^\alg_{\chi^{}_2} = 9 \cdot 3 \cdot 3 =3^4, \\
&\order \CH^\alg_{\chi^{}_0} \cdot \order \CH^\alg_{\chi^{}_1} = 9 \cdot 3 = 3^3.
\end{aligned}\right.
\end{equation*}}
\end{example}

\begin{example}\label{Ex2}
{\rm This example, for the cyclic cubic field of conductor $ f=1951$ with $p=2$, 
$\ell = 17$ split in $K$, is analogous except that capitulation takes place in $K_1$
without any stability (see complete data in Example \ref{1951}):

\ft\begin{verbatim}
p=2  f=1951  PK=x^3+x^2-650*x-289  CK0=[2,2]  ell=17 
CK1=[4,4,2,2]   
h_1^[(S-1)^1]=[2,0,0,0]   h_2^[(S-1)^1]=[0,2,0,0]
h_3^[(S-1)^1]=[0,0,0,0]   h_4^[(S-1)^1]=[0,0,0,0]
CK2=[4,4,4,4]    
h_1^[(S-1)^1]=[2,0,0,0]   h_2^[(S-1)^1]=[0,2,0,0] 
h_3^[(S-1)^1]=[2,0,2,0]   h_4^[(S-1)^1]=[0,2,0,2]
h_1^[(S-1)^2]=[0,0,0,0]   h_2^[(S-1)^2]=[0,0,0,0] 
h_3^[(S-1)^2]=[0,0,0,0]   h_4^[(S-1)^2]=[0,0,0,0]
CK3=[8,8,4,4]
h_1^[(S-1)^1]=[0,0,2,0]   h_2^[(S-1)^1]=[2,2,2,0]
h_3^[(S-1)^1]=[2,0,0,0]   h_4^[(S-1)^1]=[6,2,0,2]
h_1^[(S-1)^2]=[4,0,0,0]   h_2^[(S-1)^2]=[0,4,0,0]
h_3^[(S-1)^2]=[0,0,0,0]   h_4^[(S-1)^2]=[0,0,0,0]
\end{verbatim}\ns

Numerical data of the form ${\sf h_j^{[(S-1)^i]} = [e_1,e_2,e_3, e_4]}$ give, 
with $\BZ = \Z_2[\exp(\frac{2 i \pi}{3})]$:
\begin{equation*}
\hspace{-2.2cm} \left\{
\begin{aligned}
\Nu_{\!K_3/K_2}(\CH_{K_3}) &= \CH_{K_3}^{1+\sigma_3^4} = \CH_{K_3}^2 
\simeq \BZ/4\BZ \times \BZ/2\BZ, \\
\Nu_{\!K_2/K_1}(\CH_{K_2}) &=  \CH_{K_2}^{1+\sigma_2^2} =
\CH_{K_2}^2 \simeq \BZ/2\BZ \times \BZ/2\BZ, \\
\Nu_{\!K_1/K}(\CH_{K_1}) & =  \CH_{K_1}^{1+\sigma_1}= \CH_{K_1}^4 = 1.
\end{aligned}\right.
\end{equation*}

Whence:
\begin{equation*}
\left\{\begin{aligned}
\CH^\ar_{\chi^{}_3} & = \{x \in \CH_{K_3},\ \Norm_{K_3/K_2}(x)=1 \} 
\simeq \BZ/2\BZ, \\
\CH^\alg_{\chi^{}_3} & = \{x \in \CH_{K_3},\ \Nu_{\!K_3/K_2}(x)=1\} = \CH_{K_3}[2] 
\simeq \BZ/2\BZ \times \BZ/2\BZ,\\
\CH^\ar_{\chi^{}_2} & = \{x \in \CH_{K_2},\ \Norm_{K_2/K_1}(x)=1 \} 
\simeq \BZ/2\BZ, \\
\CH^\alg_{\chi^{}_2} & = \{x \in \CH_{K_2},\ \Nu_{\!K_2/K_1}(x)=1\} = \CH_{K_2}[2] 
\simeq \BZ/2\BZ \times \BZ/2\BZ,\\
\CH^\ar_{\chi^{}_1} & = \{x \in \CH_{K_1},\ \Norm_{K_1/K}(x)=1\} 
\simeq \BZ/2\BZ \times \BZ/2\BZ \ {\rm or}\ \BZ/4\BZ, \\
\CH^\alg_{\chi^{}_1} & = \{x \in \CH_{K_1},\  \Nu_{\!K_1/K}(x)=1\} = \CH_{K_1}
\simeq \BZ/4\BZ\! \times\! \BZ/2\BZ.
\end{aligned}\right.
\end{equation*}

Which gives, noting that $\order (\BZ/2^k\BZ) = 4^k$:
\begin{equation*}
\hspace{-0.1cm}\left\{\begin{aligned}
&\order \CH_{K_3} = \order \CH^\ar_{\chi_0^{}} \cdot \order \CH^\ar_{\chi_1^{}} 
\cdot \order \CH^\ar_{\chi_2^{}} \cdot \order \CH^\ar_{\chi_3^{}} 
= 2^2 \cdot 4^2 \cdot 2^2 \cdot 2^2 = 2^{10}, \\
&\order \CH_{K_2} = \order \CH^\ar_{\chi_0^{}} \cdot \order \CH^\ar_{\chi_1^{}} 
\cdot \order \CH^\ar_{\chi_2^{}} = 2^2 \cdot 4^2 \cdot 2^2 = 2^8,  \\
&\order \CH_{K_1} = \order \CH^\ar_{\chi_0^{}} \cdot \order \CH^\ar_{\chi_1^{}} 
= 2^2 \cdot 2^4 = 2^6,
\end{aligned}\right.
\end{equation*}

\noindent
contrary to:
\begin{equation*}
\hspace{-1.1cm}\left\{\begin{aligned}
&\order \CH^\alg_{\chi_0^{}} \cdot \order \CH^\alg_{\chi_1^{}} 
\cdot \order \CH^\alg_{\chi_2^{}} \cdot \order \CH^\alg_{\chi_3^{}} 
= 2^2 \cdot 2^6 \cdot 2^4 \cdot 2^4  = 2^{16}, \\
&\order \CH^\alg_{\chi_0^{}} \cdot \order \CH^\alg_{\chi_1^{}} 
\cdot \order \CH^\alg_{\chi_2^{}}  = 2^2 \cdot 2^6 \cdot 2^4 = 2^{12}, \\
&\order \CH^\alg_{\chi_0^{}} \cdot \order \CH^\alg_{\chi_1^{}}  = 2^2 \cdot 2^6 = 2^8.
\end{aligned}\right.
\end{equation*}

We will illustrate, in $K_1$, the analytic equality, discussed in Remark \ref{remaconj}: 
$$\order \CH^\ar_{\chi_1^{}} = \order (\CE_{K_1}/\CE^0_{K_1} \!\cdot \CF_{\!K_1}). $$
 
Let $k := \Q(\sqrt{17})$, $\Gal(K_1/k) =: \{1, \tau, \tau^2\}$ and $G_1= \{1, \sigma\}$.
Since $17$ splits in $K$ and $1951$ splits in $k$, the generating cyclotomic 
unit $\eta$ of $K_1$ (of conductor $17 \cdot 1951$) is of norm~$1$, both in 
$K_1/k$ and $K_1/K$, so it generates $\CF_{\!\chi_1^{}}$ that we may write 
$\langle \eta^\tau, \eta^{\tau^2} \rangle_\Z$ since
$\eta^\sigma = \eta^{-1}$ and $\eta = \eta^{-\tau-\tau^2}$. 
Computing in $\Q(\mu_{17 \cdot 1951})/K_1$
gives (taking logarithms for convenience):
\begin{equation*}
\hspace{-0.75cm}\left\{\begin{aligned}
\log(\eta^\tau) &= + 32.072728696925313868267792411432213485, \\
\log(\eta^{\tau^2}) &=-27.025540940115152089120603990746892715, \\
\log(\eta) &= -5.0471877568101617791471884206853207885.
\end{aligned}\right.
\end{equation*}

The group $\CE_{K_1}$ given by  PARI is of the form
$\CE_{K_1} = \CE_k \oplus \CE_K  \oplus  \langle e_4, e_5 \rangle_\Z$, 
where $\Norm_{K_1/K}(e_4) = \Norm_{K_1/K}(e_5) = 1$,
$\Norm_{K_1/k}(e_4) = \Norm_{K_1/k}(e_5) = 1$,  
$\CE_{\chi_1^{}} = \langle e_4, e_5 \rangle_\Z$, with:

\begin{equation*}
\hspace{-0.8cm}\left\{\begin{aligned}
\log(e_4) &= -8.0181821742313284670669481028585732345, \\
\log(e_5) &= + 6.7563852350287880222801509976867231766,
\end{aligned}\right.
\end{equation*}

\noindent
yielding immediately $\eta^\tau = e_4^{-4}$,  $\eta^{\tau^2} = e_5^{-4}$.
Thus, in this example:
$$\CE_{K_1}/\CE^0_{K_1}\! \!\cdot \!\CF_{\!K_1} = \CE_{\chi_1^{}}/ \CF_{\!\chi_1^{}} =
\langle  e_4, e_5 \rangle_\Z/\langle \eta^\tau, \eta^{\tau^2} \rangle_\Z \simeq \BZ/4\BZ,$$ 
of order $16$. 
So, we have $\order \CH^\ar_{\chi_1^{}} = 
\order (\CE_{\chi_1^{}}/\CF_{\!\chi_1^{}})$ (possibly with different structures),
which relativizes the interest of algebraic definitions, regarding analytic formulas, 
since $\CH^\alg_{\chi^{}_1} 
= \CH_{K_1} \simeq \BZ/4\BZ\! \times\! \BZ/2\BZ$ in that example.}
\end{example}

\begin{remark}
More generally, Galois cohomology groups are based on algebraic definitions of the 
norms, so that results strongly depend on capitulation phenomena. For instance, 
let $L/K$ be a cyclic $p$-extension of Galois group $G$ such that all the $r$ 
prime ideals of $K$ ramified in $L/K$, are totally ramified; thus,
$$\Hom^1(G,\CH_L) = \Ker_{\CH_L}(\Nu_{\!L/K})/\CH_L^{\sigma - 1} \ \ \& \ \ 
\Hom^2(G,\CH_L) = \CH_L^G/ \Nu_{\!L/K}(\CH_L), $$ 
are of same order $\ffrac{\order \CH_L^G}{\order \J_{L/K}(\CH_K)} = \ffrac{\order \CH_K}
{\order \J_{L/K}(\CH_K)} \times \ffrac{p^{N\,(r-1)}}{\order \omega_{L/K}(\BE_K)}$.
So, if $\Ker(\J_{L/K})=1$, the order is $\ffrac{p^{N\,(r-1)}}{\order \omega_{L/K}(\BE_K)}$;
if $\CH_K$ capitulates, the order becomes $\order\CH_K \times \ffrac{p^{N\,(r-1)}}
{\order \omega_{L/K}(\BE_K)}$ and any intermediate situation does exist. 
\end{remark}

\section{Tables for cubic fields and \texorpdfstring{$p=2$}{Lg}}\label{cubic}

We consider various totally ramified cyclic $p$-extensions $L/K$, where
$K$ is a cyclic cubic field and $L = K(\mu_\ell^{})$, $\ell \equiv 1 \pmod {2 p^\N}$. 
The program eliminates the cases of stability $\order \CH_{K_1}
 = \order \CH_K$ since capitulation holds in a suitable layer if $e(\Kk) \leq N$.
Then ${\sf vHK}$ defines the minimal $p$-adic valuation of the $\order \CH_K$'s 
to be considered (it may be chosen at will) then $r \in \{1,3\}$ is the number of 
prime ideals above $\ell$ in $K$.

The submodules $\Nu_{K_n/K}(\CH_{K_n}) 
= \J_{K_n/K}(\CH_K)$ are computed for $n\leq 2$ and given under the 
form ${\sf h_j^{[(S-1)^i]} = [e_1,\ldots, e_{rKn}]}$, where ${\sf rKn}$ is
the $\Z$-rank of $\BH_{K_n}$. 

\subsection{Case \texorpdfstring{$\ell = 17$}{Lg}}\label{cubic17}
We give an excerpt of the various cases obtained (all these examples show the 
randomness of the structures and of the capitulations, complete or incomplete).
We indicate if $\CH_K$ capitulates in $K_3$ (not computed) which holds as soon 
as $\order \CH_{K_2} = \order \CH_{K_1}$ (stability from $K_1$) and $e(\Kk) \leq 2$
(Theorem \ref{main2}\,(i)):

\ft\begin{verbatim}
MAIN PROGRAM FOR CYCLIC CUBIC FIELDS
{p=2;Nn=2;bf=7;Bf=10^4;vHK=2;ell=17;mKn=2;
for(f=bf,Bf,h=valuation(f,3);if(h!=0 & h!=2,next);F=f/3^h;
if(core(F)!=F,next);F=factor(F);Div=component(F,1);d=matsize(F)[1];
for(j=1,d,D=Div[j];if(Mod(D,3)!=1,break));for(b=1,sqrt(4*f/27),
if(h==2 & Mod(b,3)==0,next);A=4*f-27*b^2;if(issquare(A,&a)==1,
if(h==0,if(Mod(a,3)==1,a=-a);PK=x^3+x^2+(1-f)/3*x+(f*(a-3)+1)/27);
if(h==2,if(Mod(a,9)==3,a=-a);PK=x^3-f/3*x-f*a/27);
K=bnfinit(PK,1);r=matsize(idealfactor(K,ell))[1];
\\Testing the order of the p-class group of K compared to vHK:
HK=K.no;if(valuation(HK,p)<vHK,next);CK0=K.clgp;
for(n=1,Nn,Qn=polsubcyclo(ell,p^n);Pn=polcompositum(PK,Qn)[1];
Kn=bnfinit(Pn,1);HKn=Kn.no;dn=poldegree(Pn);
\\Test for elimination of the stability from K:
if(n==1 & valuation(HKn,p)==valuation(HK,p),break);
if(n==1,print("f=",f," PK=",PK," CK0=",CK0[2]," ell=",ell," r=",r));
CKn=Kn.clgp;print("CK",n,"=",CKn[2]);rKn=matsize(CKn[2])[2];
\\Search of a generator S of Gal(Kn/K):
G=nfgaloisconj(Kn);Id=x;for(k=1,dn,Z=G[k];ks=1;while(Z!=Id,
Z=nfgaloisapply(Kn,G[k],Z);ks=ks+1);if(ks==p^n,S=G[k];break));
\\Computation of the filtration:
for(j=1,rKn,X=CKn[3][j];Y=X;for(i=1,mKn,YS=nfgaloisapply(Kn,S,Y);
T=idealpow(Kn,Y,-1);Y=idealmul(Kn,YS,T);B=bnfisprincipal(Kn,Y)[1];
Ehij=List;for(j=1,rKn,c=B[j];w=valuation(CKn[2][j],p);c=lift(Mod(c,p^w)); 
listput(Ehij,c,j));print("h_",j,"^[","(S-1)^",i,"]=",Ehij)));
\\Computation of the algebraic norms of the rKn generators h_j:
for(j=1,rKn,A0=CKn[3][j];A=1;for(t=1,p^n,As=nfgaloisapply(Kn,S,A);
A=idealmul(Kn,A0,As));B=bnfisprincipal(Kn,A)[1];
\\Reduction modulo suitable p-powers of the exponents: 
Enu=List;for(j=1,rKn,c=B[j];w=valuation(CKn[2][j],p);
c=lift(Mod(c,p^w));listput(Enu,c,j));
print("norm in K",n,"/K of the component ",j," of CK",n,":",Enu))))))}
\end{verbatim}\ns
\ft\begin{verbatim}
p=2  f=607  PK=x^3+x^2-202*x-1169  CK0=[2,2]  ell=17  r=1
CK1=[2,2,2,2]
h_1^[(S-1)^1]=[1,0,0,1]   h_2^[(S-1)^1]=[0,1,1,1] 
h_3^[(S-1)^1]=[1,1,1,0]   h_4^[(S-1)^1]=[1,0,0,1]
h_1^[(S-1)^2]=[0,0,0,0]   h_2^[(S-1)^2]=[0,0,0,0] 
h_3^[(S-1)^2]=[0,0,0,0]   h_4^[(S-1)^2]=[0,0,0,0]
norm in K1/K of the component 1 of CK1:[1,0,0,1]
norm in K1/K of the component 2 of CK1:[0,1,1,1]
norm in K1/K of the component 3 of CK1:[1,1,1,0]
norm in K1/K of the component 4 of CK1:[1,0,0,1]
No capitulation, m(K1)=2, e(K1)=1
CK2=[2,2,2,2]
h_1^[(S-1)^1]=[1,0,1,0]   h_2^[(S-1)^1]=[0,1,0,1] 
h_3^[(S-1)^1]=[1,0,1,0]   h_4^[(S-1)^1]=[0,1,0,1]
h_1^[(S-1)^2]=[0,0,0,0]   h_2^[(S-1)^2]=[0,0,0,0] 
h_3^[(S-1)^2]=[0,0,0,0]   h_4^[(S-1)^2]=[0,0,0,0]
norm in K2/K of the component 1 of CK2:[0,0,0,0]
norm in K2/K of the component 2 of CK2:[0,0,0,0]
norm in K2/K of the component 3 of CK2:[0,0,0,0]
norm in K2/K of the component 4 of CK2:[0,0,0,0]
Complete capitulation, m(K2)=2, e(K2)=1
\end{verbatim}\ns
\ft\begin{verbatim}
p=2  f=1009  PK=x^3+x^2-336*x-1719  CK0=[2,2]  ell=17  r=1
CK1=[28,4]=[4,4]
h_1^[(S-1)^1]=[0,2]   h_2^[(S-1)^1]=[2,2]
h_1^[(S-1)^2]=[0,0]   h_2^[(S-1)^2]=[0,0]
norm in K1/K of the component 1 of CK1:[2,2]
norm in K1/K of the component 2 of CK1:[2,0]
No capitulation, m(K1)=2, e(K1)=2
CK2=[28,4]=[4,4]
h_1^[(S-1)^1]=[0,2]   h_2^[(S-1)^1]=[2,2]
h_1^[(S-1)^2]=[0,0]   h_2^[(S-1)^2]=[0,0]
norm in K2/K of the component 1 of CK2:[0,0]
norm in K2/K of the component 2 of CK2:[0,0]
Complete capitulation, m(K2)=2, e(K2)=2
\end{verbatim}\ns
\ft\begin{verbatim}
p=2  f=1789  PK=x^3+x^2-596*x-5632  CK0=[2,2]  ell=17  r=1
CK1=[24,8]=[8,8]
h_1^[(S-1)^1]=[2,0]   h_2^[(S-1)^1]=[0,2]
h_1^[(S-1)^2]=[4,0]   h_2^[(S-1)^2]=[0,4]
norm in K1/K of the component 1 of CK1:[4,0]
norm in K1/K of the component 2 of CK1:[0,4]
No capitulation, m(K1)=3, e(K1)=3
CK2=[312,8]=[8,8]
h_1^[(S-1)^1]=[2,0]   h_2^[(S-1)^1]=[0,2]
h_1^[(S-1)^2]=[4,0]   h_2^[(S-1)^2]=[0,4]
norm in K2/K of the component 1 of CK2:[0,0]
norm in K2/K of the component 2 of CK2:[0,0]
Complete capitulation, m(K2)=3, e(K2)=3
\end{verbatim}\ns
\ft\begin{verbatim}
p=2  f=2077  PK=x^3+x^2-692*x-7231  CK0=[6,2]  ell=17  r=1
CK1=[6,2,2,2]=[2,2,2,2]
h_1^[(S-1)^1]=[1,1,1,0] h_2^[(S-1)^1]=[0,0,1,1] 
h_3^[(S-1)^1]=[1,1,0,1] h_4^[(S-1)^1]=[1,1,0,1]
h_1^[(S-1)^2]=[0,0,0,0] h_2^[(S-1)^2]=[0,0,0,0] 
h_3^[(S-1)^2]=[0,0,0,0] h_4^[(S-1)^2]=[0,0,0,0]
norm in K1/K of the component 1 of CK1:[1,1,1,0]
norm in K1/K of the component 2 of CK1:[0,0,1,1]
norm in K1/K of the component 3 of CK1:[1,1,0,1]
norm in K1/K of the component 4 of CK1:[1,1,0,1]
No capitulation, m(K1)=2, e(K1)=1
CK2=[6,2,2,2,2,2]=[2,2,2,2,2,2]
h_1^[(S-1)^1]=[1,0,0,0,1,0] h_2^[(S-1)^1]=[0,1,0,0,0,1] 
h_3^[(S-1)^1]=[1,1,0,1,0,1] h_4^[(S-1)^1]=[1,0,0,1,0,0] 
h_5^[(S-1)^1]=[1,0,0,1,0,0] h_6^[(S-1)^1]=[0,0,1,1,0,1]
h_1^[(S-1)^2]=[0,0,0,1,1,0] h_2^[(S-1)^2]=[0,1,1,1,0,0] 
h_3^[(S-1)^2]=[0,1,1,0,1,0] h_4^[(S-1)^2]=[0,0,0,1,1,0] 
h_5^[(S-1)^2]=[0,0,0,1,1,0] h_6^[(S-1)^2]=[0,1,1,1,0,0]
norm in K2/K of the component 1 of CK2:[0,0,0,0,0,0]
norm in K2/K of the component 2 of CK2:[0,0,0,0,0,0]
norm in K2/K of the component 3 of CK2:[0,0,0,0,0,0]
norm in K2/K of the component 4 of CK2:[0,0,0,0,0,0]
norm in K2/K of the component 5 of CK2:[0,0,0,0,0,0]
norm in K2/K of the component 6 of CK2:[0,0,0,0,0,0]
Complete capitulation, m(K2)=3, e(K2)=1
\end{verbatim}\ns
\ft\begin{verbatim}
p=2  f=2817  PK=x^3-939*x+6886  CK0=[12,4]  ell=17  r=1
CK1=[84,4]=[4,4]
h_1^[(S-1)^1]=[0,0]   h_2^[(S-1)^1]=[0,0]
h_1^[(S-1)^2]=[0,0]   h_2^[(S-1)^2]=[0,0]
norm in K1/K of the component 1 of CK1:[2,0]
norm in K1/K of the component 2 of CK1:[0,2]
Incomplete capitulation, m(K1)=1, e(K1)=2
CK2=[84,4]=[4,4]
h_1^[(S-1)^1]=[0,0]   h_2^[(S-1)^1]=[0,0]
h_1^[(S-1)^2]=[0,0]   h_2^[(S-1)^2]=[0,0]
norm in K2/K of the component 1 of CK2:[0,0]
norm in K2/K of the component 2 of CK2:[0,0]
Complete capitulation, m(K2)=1, e(K2)=2
\end{verbatim}\ns
\ft\begin{verbatim}
p=2  f=3357  PK=x^3-1119*x+9325  CK0=[6,2]  ell=17  r=3
CK1=[6,2,2,2]=[2,2,2,2]
h_1^[(S-1)^1]=[0,0,0,0]   h_2^[(S-1)^1]=[0,0,0,0]
h_3^[(S-1)^1]=[0,0,0,0]   h_4^[(S-1)^1]=[0,0,0,0]
norm in K1/K of the component 1 of CK1:[0,0,0,0]
norm in K1/K of the component 2 of CK1:[0,0,0,0]
norm in K1/K of the component 3 of CK1:[0,0,0,0]
norm in K1/K of the component 4 of CK1:[0,0,0,0]
Complete capitulation, m(K1)=1, e(K1)=1
CK2=[12,4,2,2]=[4,4,2,2]
h_1^[(S-1)^1]=[2,0,0,0] h_2^[(S-1)^1]=[0,0,0,0] 
h_3^[(S-1)^1]=[2,0,0,0] h_4^[(S-1)^1]=[0,2,0,0]
h_1^[(S-1)^2]=[0,0,0,0] h_2^[(S-1)^2]=[0,0,0,0] 
h_3^[(S-1)^2]=[0,0,0,0] h_4^[(S-1)^2]=[0,0,0,0]
norm in K2/K of the component 1 of CK2:[0,0,0,0]
norm in K2/K of the component 2 of CK2:[0,0,0,0]
norm in K2/K of the component 3 of CK2:[0,0,0,0]
norm in K2/K of the component 4 of CK2:[0,0,0,0]
Complete capitulation, m(K2)=2, e(K2)=2
\end{verbatim}\ns
\ft\begin{verbatim}
p=2  f=5479  PK=x^3+x^2-1826*x+13799  CK0=[2,2]  ell=17  r=1
CK1=[2,2,2,2]
h_1^[(S-1)^1]=[1,0,0,1] h_2^[(S-1)^1]=[1,1,1,0] 
h_3^[(S-1)^1]=[0,1,1,1] h_4^[(S-1)^1]=[1,0,0,1]
h_1^[(S-1)^2]=[0,0,0,0] h_2^[(S-1)^2]=[0,0,0,0] 
h_3^[(S-1)^2]=[0,0,0,0] h_4^[(S-1)^2]=[0,0,0,0]
norm in K1/K of the component 1 of CK1:[1,0,0,1]
norm in K1/K of the component 2 of CK1:[1,1,1,0]
norm in K1/K of the component 3 of CK1:[0,1,1,1]
norm in K1/K of the component 4 of CK1:[1,0,0,1]
No capitulation, m(K1)=2, e(K1)=1
CK2=[4,4,4,4]
h_1^[(S-1)^1]=[3,1,3,1] h_2^[(S-1)^1]=[3,3,0,3] 
h_3^[(S-1)^1]=[0,2,2,2] h_4^[(S-1)^1]=[0,2,3,0]
h_1^[(S-1)^2]=[0,2,2,0] h_2^[(S-1)^2]=[2,2,2,0] 
h_3^[(S-1)^2]=[2,2,2,2] h_4^[(S-1)^2]=[2,0,2,0]
norm in K2/K of the component 1 of CK2:[0,0,2,0]
norm in K2/K of the component 2 of CK2:[2,2,2,2]
norm in K2/K of the component 3 of CK2:[0,0,0,0]
norm in K2/K of the component 4 of CK2:[2,2,0,2]
No capitulation, m(K2)=4, e(K2)=2
\end{verbatim}\ns
\ft\begin{verbatim}
p=2  f=6247  PK=x^3+x^2-2082*x-35631  CK0=[4,4]  ell=17  r=1
CK1=[24,8,2,2]=[8,8,2,2]
h_1^[(S-1)^1]=[0,2,0,1] h_2^[(S-1)^1]=[6,6,1,0] 
h_3^[(S-1)^1]=[0,4,0,0] h_4^[(S-1)^1]=[4,0,0,0]
h_1^[(S-1)^2]=[0,4,0,0] h_2^[(S-1)^2]=[4,4,0,0] 
h_3^[(S-1)^2]=[0,0,0,0] h_4^[(S-1)^2]=[0,0,0,0]
norm in K1/K of the component 1 of CK1:[2,2,0,1]
norm in K1/K of the component 2 of CK1:[6,0,1,0]
norm in K1/K of the component 3 of CK1:[0,4,0,0]
norm in K1/K of the component 4 of CK1:[4,0,0,0]
No capitulation, m(K1)=4, e(K1)=3
CK2=[24,8,2,2]=[8,8,2,2]
h_1^[(S-1)^1]=[0,6,1,1] h_2^[(S-1)^1]=[2,6,0,1] 
h_3^[(S-1)^1]=[4,4,0,0] h_4^[(S-1)^1]=[0,4,0,0]
h_1^[(S-1)^2]=[0,4,0,0] h_2^[(S-1)^2]=[4,4,0,0] 
h_3^[(S-1)^2]=[0,0,0,0] h_4^[(S-1)^2]=[0,0,0,0]
norm in K2/K of the component 1 of CK2:[4,4,0,0]
norm in K2/K of the component 2 of CK2:[4,0,0,0]
norm in K2/K of the component 3 of CK2:[0,0,0,0]
norm in K2/K of the component 4 of CK2:[0,0,0,0]
Incomplete capitulation, m(K2)=3, e(K2)=3
Complete capitulation in K3 (stability from K1)
\end{verbatim}\ns
\ft\begin{verbatim}
p=2  f=9247  PK=x^3+x^2-3082*x-27056  CK0=[12,4]  ell=17  r=3
CK1=[24,8]=[8,8]
h_1^[(S-1)^1]=[0,0]   h_2^[(S-1)^1]=[0,0]
norm in K1/K of the component 1 of CK1:[2,0]
norm in K1/K of the component 2 of CK1:[0,2]
No capitulation, m(K1)=1, e(K1)=3
CK2=[48,16]=[16,16]
h_1^[(S-1)^1]=[0,0]   h_2^[(S-1)^1]=[0,0]
norm in K2/K of the component 1 of CK2:[4,0]
norm in K2/K of the component 2 of CK2:[0,4]
No capitulation, m(K2)=1, e(K2)=4
\end{verbatim}\ns
\ft\begin{verbatim}
p=2  f=20887  PK=x^3+x^2-6962*x-225889  CK0=[4,4,2,2]  ell=17  r=3
CK1=[8,8,2,2]
h_1^[(S-1)^1]=[0,0,0,0]   h_2^[(S-1)^1]=[0,0,0,0]
h_3^[(S-1)^1]=[0,0,0,0]   h_4^[(S-1)^1]=[0,0,0,0]
norm in K1/K of the component 1 of CK1:[2,0,0,0]
norm in K1/K of the component 2 of CK1:[0,2,0,0]
norm in K1/K of the component 3 of CK1:[0,0,0,0]
norm in K1/K of the component 4 of CK1:[0,0,0,0]
Incomplete capitulation, m(K1)=1, e(K1)=3
CK2=[16,16,2,2]
h_1^[(S-1)^1]=[8,0,0,0] h_2^[(S-1)^1]=[0,8,0,0] 
h_3^[(S-1)^1]=[0,0,0,0] h_4^[(S-1)^1]=[0,0,0,0]
h_1^[(S-1)^2]=[0,0,0,0] h_2^[(S-1)^2]=[0,0,0,0] 
h_3^[(S-1)^2]=[0,0,0,0] h_4^[(S-1)^2]=[0,0,0,0]
norm in K2/K of the component 1 of CK2:[4,0,0,0]
norm in K2/K of the component 2 of CK2:[0,4,0,0]
norm in K2/K of the component 3 of CK2:[0,0,0,0]
norm in K2/K of the component 4 of CK2:[0,0,0,0]
Incomplete capitulation, m(K2)=2, e(K2)=4
\end{verbatim}\ns
\ft\begin{verbatim}
p=2  f=48769  PK=x^3+x^2-16256*x-7225  CK0=[24,8]  ell=17  r=3
CK1=[48,16]=[16,16]
h_1^[(S-1)^1]=[0,0]   h_2^[(S-1)^1]=[0,0]
norm in K1/K of the component 1 of CK1:[2,0]
norm in K1/K of the component 2 of CK1:[0,2]
No capitulation,  m(K1)=1, e(K1)=4
CK2=[48,16]=[16,16]
h_1^[(S-1)^1]=[0,0]   h_2^[(S-1)^1]=[0,0]
norm in K2/K of the component 1 of CK2:[4,0]
norm in K2/K of the component 2 of CK2:[0,4]
Incomplete capitulation, m(K2)=1, e(K2)=4
\end{verbatim}\ns
\ft\begin{verbatim}
p=2  f=55609  PK=x^3+x^2-18536*x-823837  CK0=[4,4,2,2]  ell=17  r=3
CK1=[56,8,2,2]=[8,8,2,2]
h_1^[(S-1)^1]=[4,0,0,0] h_2^[(S-1)^1]=[0,0,0,0] 
h_3^[(S-1)^1]=[4,0,0,0] h_4^[(S-1)^1]=[0,4,0,0]
h_1^[(S-1)^2]=[0,0,0,0] h_2^[(S-1)^2]=[0,0,0,0] 
h_3^[(S-1)^2]=[0,0,0,0] h_4^[(S-1)^2]=[0,0,0,0]
norm in K1/K of the component 1 of CK1:[6,0,0,0]
norm in K1/K of the component 2 of CK1:[0,2,0,0]
norm in K1/K of the component 3 of CK1:[4,0,0,0]
norm in K1/K of the component 4 of CK1:[0,4,0,0]
Incomplete capitulation, m(K1)=2, e(K1)=3
CK2=[56,8,2,2]=[8,8,2,2]
h_1^[(S-1)^1]=[0,0,0,0] h_2^[(S-1)^1]=[0,0,0,0] 
h_3^[(S-1)^1]=[0,4,0,0] h_4^[(S-1)^1]=[4,4,0,0]
h_1^[(S-1)^2]=[0,0,0,0] h_2^[(S-1)^2]=[0,0,0,0] 
h_3^[(S-1)^2]=[0,0,0,0] h_4^[(S-1)^2]=[0,0,0,0]
norm in K2/K of the component 1 of CK2:[4,0,0,0]
norm in K2/K of the component 2 of CK2:[0,4,0,0]
norm in K2/K of the component 3 of CK2:[0,0,0,0]
norm in K2/K of the component 4 of CK2:[0,0,0,0]
Complete capitulation, m(K2)=2, e(K2)=3
\end{verbatim}\ns

\subsection{Case  \texorpdfstring{$\ell=17$}{Lg} and
\texorpdfstring{$f \in \{9283, 7687, 44857\}$}{Lg}}
Let's give some comments on interesting examples given by the program:

\ft\begin{verbatim}
p=2  f=9283  PK=x^3+x^2-3094*x-5501  CK0=[2,2]  ell=17  r=1
CK1=[48,16]=[16,16]
h_1^[(S-1)^1]=[6,0]  h_2^[(S-1)^1]=[0,6]
h_1^[(S-1)^2]=[4,0]  h_2^[(S-1)^2]=[0,4]
h_1^[(S-1)^3]=[8,0]  h_2^[(S-1)^3]=[0,8]
h_1^[(S-1)^4]=[0,0]  h_2^[(S-1)^4]=[0,0]
norm in K1/K of the component 1 of CK1:[8,0]
norm in K1/K of the component 2 of CK1:[0,8]
No capitulation, m(K1)=4, e(K1)=4
CK2=[48,16]=[16,16]
h_1^[(S-1)^1]=[6,0]  h_2^[(S-1)^1]=[0,6]
h_1^[(S-1)^2]=[4,0]  h_2^[(S-1)^2]=[0,4]
h_1^[(S-1)^3]=[8,0]  h_2^[(S-1)^3]=[0,8]
h_1^[(S-1)^4]=[0,0]  h_2^[(S-1)^4]=[0,0]
norm in K2/K of the component 1 of CK2:[0,0]
norm in K2/K of the component 2 of CK2:[0,0]
Complete capitulation, m(K2)=4, e(K2)=4
\end{verbatim}\ns

There is complete capitulation, even if conditions of Theorem \ref{main1}
are not fulfilled for the $K_n/K$'s (for $n=2$, $m(\Kk_2)=4$, $s(\Kk_2)=2$, 
$e(\Kk_2)=4$, $n-s(\Kk_2)=0$). 
Moreover, the exponent of $\CH_{K_1}$ is $2^4$ giving a larger complexity in $K_1/K$, 
but in $K_n$, $n \geq 2$, the exponent is still $2^4$ (no increasing of the complexity). 
Some other examples are:

\ft\begin{verbatim}
p=2  f=7687  PK=x^3+x^2-2562*x-48969  CK0=[2,2,2,2]  ell=17  r=1
CK1=[4,4,2,2]
h_1^[(S-1)^1]=[0,2,0,0] h_2^[(S-1)^1]=[0,0,0,0] 
h_3^[(S-1)^1]=[2,0,0,0] h_4^[(S-1)^1]=[0,2,0,0]
h_1^[(S-1)^2]=[0,0,0,0] h_2^[(S-1)^2]=[0,0,0,0] 
h_3^[(S-1)^2]=[0,0,0,0] h_4^[(S-1)^2]=[0,0,0,0]
norm in K1/K of the component 1 of CK1:[2,2,0,0]
norm in K1/K of the component 2 of CK1:[0,2,0,0]
norm in K1/K of the component 3 of CK1:[2,0,0,0]
norm in K1/K of the component 4 of CK1:[0,2,0,0]
Incomplete capitulation, m(K1)=2, e(K1)=2
CK2=[4,4,2,2]
h_1^[(S-1)^1]=[2,2,0,0] h_2^[(S-1)^1]=[2,0,0,0] 
h_3^[(S-1)^1]=[2,2,0,0] h_4^[(S-1)^1]=[2,0,0,0]
h_1^[(S-1)^2]=[0,0,0,0] h_2^[(S-1)^2]=[0,0,0,0] 
h_3^[(S-1)^2]=[0,0,0,0] h_4^[(S-1)^2]=[0,0,0,0]
norm in K2/K of the component 1 of CK2:[0,0,0,0]
norm in K2/K of the component 2 of CK2:[0,0,0,0]
norm in K2/K of the component 3 of CK2:[0,0,0,0]
norm in K2/K of the component 4 of CK2:[0,0,0,0]
Complete capitulation, m(K2)=2, e(K2)=2
\end{verbatim}\ns
\ft\begin{verbatim}
p=2  f=44857  PK=x^3+x^2-14952*x-704421  CK0=[6,2,2,2]  ell=17  r=3
CK1=[12,12,2,2]=[4,4,2,2]
h_1^[(S-1)^1]=[0,0,0,0] h_2^[(S-1)^1]=[0,0,0,0] 
h_3^[(S-1)^1]=[2,0,0,0] h_4^[(S-1)^1]=[2,2,0,0]
h_1^[(S-1)^2]=[0,0,0,0] h_2^[(S-1)^2]=[0,0,0,0] 
h_3^[(S-1)^2]=[0,0,0,0] h_4^[(S-1)^2]=[0,0,0,0]
norm in K1/K of the component 1 of CK1:[2,0,0,0]
norm in K1/K of the component 2 of CK1:[0,2,0,0]
norm in K1/K of the component 3 of CK1:[2,0,0,0]
norm in K1/K of the component 4 of CK1:[2,2,0,0]
Incomplete capitulation, m(K1)=2, e(K1)=2
CK2=[12,12,2,2]=[4,4,2,2]
h_1^[(S-1)^1]=[0,0,0,0] h_2^[(S-1)^1]=[2,2,0,0] 
h_3^[(S-1)^1]=[2,2,0,0] h_4^[(S-1)^1]=[0,2,0,0]
h_1^[(S-1)^2]=[0,0,0,0] h_2^[(S-1)^2]=[0,0,0,0] 
h_3^[(S-1)^2]=[0,0,0,0] h_4^[(S-1)^2]=[0,0,0,0]
norm in K2/K of the component 1 of CK2:[0,0,0,0]
norm in K2/K of the component 2 of CK2:[0,0,0,0]
norm in K2/K of the component 3 of CK2:[0,0,0,0]
norm in K2/K of the component 4 of CK2:[0,0,0,0]
Complete capitulation, m(K2)=2, e(K2)=2
\end{verbatim}\ns

\noindent
these examples suggest that the order of magnitude of the $p$-ranks of the
$\CH_K$'s is not an obstruction to a capitulation in such cyclic $p$-extensions 
$L \subset K(\mu_\ell^{})$. In the above cases, the capitulation is obtained by 
means of a stability in larger layers.

\subsection{Case \texorpdfstring{$\ell = 97$}{Lg}}
Similarly, we give a table for $\ell = 97$ allowing capitulations up to $K_4$. 
One finds much more cases of capitulation (not given in the table since they 
are very numerous); it seems clearly that a larger value of $N$ may intervene in 
the phenomenon of capitulation:

\ft\begin{verbatim}
p=2  f=349  PK=x^3+x^2-116*x-517  CK0=[2,2]  ell=97  r=1
CK1=[4,4]
h_1^[(S-1)^1]=[2,2]   h_2^[(S-1)^1]=[2,0]
h_1^[(S-1)^2]=[0,0]   h_2^[(S-1)^2]=[0,0]
norm in K1/K of the component 1 of CK1:[0,2]
norm in K1/K of the component 2 of CK1:[2,2]
Incomplete capitulation, m(K1)=2, e(K1)=2
CK2=[4,4]
h_1^[(S-1)^1]=[2,2]   h_2^[(S-1)^1]=[2,0]
h_1^[(S-1)^2]=[0,0]   h_2^[(S-1)^2]=[0,0]
norm in K2/K of the component 1 of CK2:[0,0]
norm in K2/K of the component 2 of CK2:[0,0]
Complete capitulation, m(K2)=2, e(K2)=2
\end{verbatim}\ns
\ft\begin{verbatim}
p=2  f=607  PK=x^3+x^2-202*x-1169  CK0=[2,2]  ell=97  r=1
CK1=[8,8]
h_1^[(S-1)^1]=[6,4]   h_2^[(S-1)^1]=[4,2]
h_1^[(S-1)^2]=[4,0]   h_2^[(S-1)^2]=[0,4]
norm in K1/K of the component 1 of CK1:[0,4]
norm in K1/K of the component 2 of CK1:[4,4]
Incomplete capitulation, m(K1)=3, e(K1)=3
CK2=[104,8]=[8,8]
h_1^[(S-1)^1]=[6,4]   h_2^[(S-1)^1]=[4,2]
h_1^[(S-1)^2]=[4,0]   h_2^[(S-1)^2]=[0,4]
norm in K2/K of the component 1 of CK2:[0,0]
norm in K2/K of the component 2 of CK2:[0,0]
Complete capitulation, m(K2)=3, e(K2)=3
\end{verbatim}\ns
\ft\begin{verbatim}
p=2  f=1957  PK=x^3+x^2-652*x+6016  CK0=[6,2]  ell=97  r=3
CK1=[12,4]=[4,4]
h_1^[(S-1)^1]=[0,0]   h_2^[(S-1)^1]=[0,0]
h_1^[(S-1)^2]=[0,0]   h_2^[(S-1)^2]=[0,0]
norm in K1/K of the component 1 of CK1:[2,0]
norm in K1/K of the component 2 of CK1:[0,2]
No capitulation, m(K1)=1, e(K1)=2
CK2=[24,8]=[8,8]
h_1^[(S-1)^1]=[0,0]   h_2^[(S-1)^1]=[0,0]
h_1^[(S-1)^2]=[0,0]   h_2^[(S-1)^2]=[0,0]
norm in K2/K of the component 1 of CK2:[4,0]
norm in K2/K of the component 2 of CK2:[0,4]
No capitulation, m(K2)=1, e(K2)=3
\end{verbatim}\ns
\ft\begin{verbatim}
p=2  f=4639  PK=x^3+x^2-1546*x+6529  CK0=[2,2]  ell=97  r=1
CK1=[4,4]
h_1^[(S-1)^1]=[2,0]   h_2^[(S-1)^1]=[0,2]
h_1^[(S-1)^2]=[0,0]   h_2^[(S-1)^2]=[0,0]
norm in K1/K of the component 1 of CK1:[0,0]
norm in K1/K of the component 2 of CK1:[0,0]
Complete capitulation, m(K1)=2, e(K1)=2
CK2=[4,4]
h_1^[(S-1)^1]=[2,0]   h_2^[(S-1)^1]=[0,2]
h_1^[(S-1)^2]=[0,0]   h_2^[(S-1)^2]=[0,0]
norm in K2/K of the component 1 of CK2:[0,0]
norm in K2/K of the component 2 of CK2:[0,0]
Complete capitulation, m(K2)=2, e(K2)=2
\end{verbatim}\ns
\ft\begin{verbatim}
p=2  f=9391  PK=x^3+x^2-3130*x-24347  CK0=[2,2]  ell=97  r=3
CK1=[4,4,2,2]
h_1^[(S-1)^1]=[2,0,0,0]   h_2^[(S-1)^1]=[0,2,0,0]
h_3^[(S-1)^1]=[2,0,0,0]   h_4^[(S-1)^1]=[2,2,0,0]
h_1^[(S-1)^2]=[0,0,0,0]   h_2^[(S-1)^2]=[0,0,0,0]
h_3^[(S-1)^2]=[0,0,0,0]   h_4^[(S-1)^2]=[0,0,0,0]
norm in K1/K of the component 1 of CK1:[0,0,0,0]
norm in K1/K of the component 2 of CK1:[0,0,0,0]
norm in K1/K of the component 3 of CK1:[2,0,0,0]
norm in K1/K of the component 4 of CK1:[2,2,0,0]
Incomplete capitulation, m(K1)=2, e(K1)=2
CK2=[4,4,4,4,2,2]
h_1^[(S-1)^1]=[0,0,0,2,0,1]   h_2^[(S-1)^1]=[0,0,2,0,1,0] 
h_3^[(S-1)^1]=[2,0,2,0,1,0]   h_4^[(S-1)^1]=[0,0,2,0,0,1]
h_5^[(S-1)^1]=[0,2,2,0,0,0]   h_6^[(S-1)^1]=[2,0,0,2,0,0]
h_1^[(S-1)^2]=[2,0,0,2,0,0]   h_2^[(S-1)^2]=[0,2,2,0,0,0]
h_3^[(S-1)^2]=[0,2,2,0,0,0]   h_4^[(S-1)^2]=[2,0,0,2,0,0]
h_5^[(S-1)^2]=[0,0,0,0,0,0]   h_6^[(S-1)^2]=[0,0,0,0,0,0]
norm in K2/K of the component 1 of CK2:[0,0,0,0,0,0]
norm in K2/K of the component 2 of CK2:[0,0,0,0,0,0]
norm in K2/K of the component 3 of CK2:[0,0,0,0,0,0]
norm in K2/K of the component 4 of CK2:[0,0,0,0,0,0]
norm in K2/K of the component 5 of CK2:[0,0,0,0,0,0]
norm in K2/K of the component 6 of CK2:[0,0,0,0,0,0]
Complete capitulation, m(K2)=3, e(K2)=2
\end{verbatim}\ns
\ft\begin{verbatim}
p=2  f=20419  PK=x^3+x^2-6806*x-3025  CK0=[42,2]  ell=97  r=3
CK1=[84,4,2,2]=[4,4,2,2]
h_1^[(S-1)^1]=[0,0,0,0]   h_2^[(S-1)^1]=[0,0,0,0]
h_3^[(S-1)^1]=[2,0,0,0]   h_4^[(S-1)^1]=[0,2,0,0]
h_1^[(S-1)^2]=[0,0,0,0]   h_2^[(S-1)^2]=[0,0,0,0]
h_3^[(S-1)^2]=[0,0,0,0]   h_4^[(S-1)^2]=[0,0,0,0]
norm in K1/K of the component 1 of CK1:[2,0,0,0]
norm in K1/K of the component 2 of CK1:[0,2,0,0]
norm in K1/K of the component 3 of CK1:[2,0,0,0]
norm in K1/K of the component 4 of CK1:[0,2,0,0]
No capitulation, m(K1)=2, e(K1)=2
CK2=[84,4,4,4,2,2]=[4,4,4,4,2,2]
h_1^[(S-1)^1]=[0,0,2,0,1,0]   h_2^[(S-1)^1]=[0,2,0,2,1,0]
h_3^[(S-1)^1]=[2,2,2,2,0,0]   h_4^[(S-1)^1]=[0,2,2,2,1,1]
h_5^[(S-1)^1]=[2,0,0,2,0,0]   h_6^[(S-1)^1]=[0,0,2,2,0,0]
h_1^[(S-1)^2]=[2,0,0,2,0,0]   h_2^[(S-1)^2]=[2,0,0,2,0,0]
h_3^[(S-1)^2]=[0,0,0,0,0,0]   h_4^[(S-1)^2]=[2,0,2,0,0,0]
h_5^[(S-1)^2]=[0,0,0,0,0,0]   h_6^[(S-1)^2]=[0,0,0,0,0,0]
norm in K2/K of the component 1 of CK2:[0,0,0,0,0,0]
norm in K2/K of the component 2 of CK2:[0,0,0,0,0,0]
norm in K2/K of the component 3 of CK2:[0,0,0,0,0,0]
norm in K2/K of the component 4 of CK2:[0,0,0,0,0,0]
norm in K2/K of the component 5 of CK2:[0,0,0,0,0,0]
norm in K2/K of the component 6 of CK2:[0,0,0,0,0,0]
Complete capitulation, m(K2)=3, e(K2)=2
\end{verbatim}\ns
\ft\begin{verbatim}
p=2  f=24589  PK=x^3+x^2-8196*x-33696  CK0=[6,2]  ell=97  r=3
CK1=[6,2,2,2]=[2,2,2,2]
h_1^[(S-1)^1]=[1,0,0,1]   h_2^[(S-1)^1]=[0,1,1,0]
h_3^[(S-1)^1]=[0,1,1,0]   h_4^[(S-1)^1]=[1,0,0,1]
h_1^[(S-1)^2]=[0,0,0,0]   h_2^[(S-1)^2]=[0,0,0,0]
h_3^[(S-1)^2]=[0,0,0,0]   h_4^[(S-1)^2]=[0,0,0,0]
norm in K1/K of the component 1 of CK1:[1,0,0,1]
norm in K1/K of the component 2 of CK1:[0,1,1,0]
norm in K1/K of the component 3 of CK1:[0,1,1,0]
norm in K1/K of the component 4 of CK1:[1,0,0,1]
No capitulation, m(K1)=2, e(K1)=1
CK2=[6,2,2,2,2,2]=[2,2,2,2,2,2]
h_1^[(S-1)^1]=[0,0,0,0,0,0]   h_2^[(S-1)^1]=[0,0,0,0,0,0]
h_3^[(S-1)^1]=[1,1,0,1,0,0]   h_4^[(S-1)^1]=[1,1,0,0,0,0]
h_5^[(S-1)^1]=[0,1,0,1,1,1]   h_6^[(S-1)^1]=[1,1,0,1,1,1]
h_1^[(S-1)^2]=[0,0,0,0,0,0]   h_2^[(S-1)^2]=[0,0,0,0,0,0]
h_3^[(S-1)^2]=[1,1,0,0,0,0]   h_4^[(S-1)^2]=[0,0,0,0,0,0]
h_5^[(S-1)^2]=[0,1,0,0,0,0]   h_6^[(S-1)^2]=[0,1,0,0,0,0]
norm in K2/K of the component 1 of CK2:[0,0,0,0,0,0]
norm in K2/K of the component 2 of CK2:[0,0,0,0,0,0]
norm in K2/K of the component 3 of CK2:[0,0,0,0,0,0]
norm in K2/K of the component 4 of CK2:[0,0,0,0,0,0]
norm in K2/K of the component 5 of CK2:[0,0,0,0,0,0]
norm in K2/K of the component 6 of CK2:[0,0,0,0,0,0]
Complete capitulation, m(K2)=3, e(K2)=1
\end{verbatim}\ns
\ft\begin{verbatim}
p=2  f=25171  PK=x^3+x^2-8390*x+273152  CK0=[14,2]  ell=97  r=3
CK1=[84,4,2,2]=[4,4,2,2]
h_1^[(S-1)^1]=[2,0,0,0]   h_2^[(S-1)^1]=[0,2,0,0]
h_3^[(S-1)^1]=[0,0,0,0]   h_4^[(S-1)^1]=[0,0,0,0]
h_1^[(S-1)^2]=[0,0,0,0]   h_2^[(S-1)^2]=[0,0,0,0]
h_3^[(S-1)^2]=[0,0,0,0]   h_4^[(S-1)^2]=[0,0,0,0]
norm in K1/K of the component 1 of CK1:[0,0,0,0]
norm in K1/K of the component 2 of CK1:[0,0,0,0]
norm in K1/K of the component 3 of CK1:[0,0,0,0]
norm in K1/K of the component 4 of CK1:[0,0,0,0]
Complete capitulation, m(K1)=2, e(K1)=2
CK2=[84,4,4,4]=[4,4,4,4]
h_1^[(S-1)^1]=[0,0,2,0]   h_2^[(S-1)^1]=[0,0,2,0]
h_3^[(S-1)^1]=[0,0,2,0]   h_4^[(S-1)^1]=[2,0,2,2]
h_1^[(S-1)^2]=[0,0,0,0]   h_2^[(S-1)^2]=[0,0,0,0]
h_3^[(S-1)^2]=[0,0,0,0]   h_4^[(S-1)^2]=[0,0,0,0]
norm in K2/K of the component 1 of CK2:[0,0,0,0]
norm in K2/K of the component 2 of CK2:[0,0,0,0]
norm in K2/K of the component 3 of CK2:[0,0,0,0]
norm in K2/K of the component 4 of CK2:[0,0,0,0]
Complete capitulation, m(K2)=2, e(K2)=2
\end{verbatim}\ns

\subsection{Capitulations for \texorpdfstring{$K$}{Lg} fixed and 
\texorpdfstring{$\ell$}{Lg} varying}
In this subsection, we fix a cyclic cubic field (given via its polynomial ${\sf PK}$)
and consider primes $\ell \equiv 1 \pmod {2 p^\N}$ with ${\sf N \geq Nell}$, and 
${\sf \ell \leq Bell}$; one must give ${\sf p}$, ${\sf Nell}$, ${\sf Bell}$:

\ft\begin{verbatim}
{p=2;Nell=3;Bell=10^3;Nn=2;f=20887;PK=x^3+x^2-6962*x-225889;K=bnfinit(PK,1);
CK0=K.clgp;forprime(ell=1,Bell,N=valuation(ell-1,p)-1;if(N<Nell,next);
r=matsize(idealfactor(K,ell))[1];for(n=1,Nn,Qn=polsubcyclo(ell,p^n);
Pn=polcompositum(PK,Qn)[1];Kn=bnfinit(Pn,1);HKn=Kn.no;dn=poldegree(Pn);
if(n==1,print();print("f=",f," PK=",PK," CK0=",CK0[2]," ell=",ell," N=",N," r=",r));
CKn=Kn.clgp;print("CK",n,"=",CKn[2]);rKn=matsize(CKn[2])[2];
G=nfgaloisconj(Kn);Id=x;for(k=1,dn,Z=G[k];ks=1;while(Z!=Id,
Z=nfgaloisapply(Kn,G[k],Z);ks=ks+1);if(ks==p^n,S=G[k];break));
for(j=1,rKn,A0=CKn[3][j];A=1;for(t=1,p^n,As=nfgaloisapply(Kn,S,A);
A=idealmul(Kn,A0,As));B=bnfisprincipal(Kn,A)[1];
Enu=List;for(j=1,rKn,c=B[j];w=valuation(CKn[2][j],p);c=lift(Mod(c,p^w)); 
listput(Enu,c,j));print("norm in K",n,"/K of the component ",j,
" of CK",n,":",Enu))))}
\end{verbatim}\ns

\subsubsection{Cubic field of conductor $f=1777$} Then $P_K=x^3+x^2-592*x+724$
and $\CH_K \simeq \BZ/4\BZ$; a complete capitulation in $K_2$ holds for the 
following primes $\ell \equiv 1 \pmod 8$ (taking ${\sf Nell=2}$):

\noindent
$\ell \in\! \{$$41$, $89$, $97$, $137$, $233$, $281$, $313$, $337$, $353$, $401$, $409$, 
$433$, $449$, $457$, $521$, $569$, $577$, $593$, $601$, $617$, $673$, $761$, 
$769$, $809$, $857$, $881$, $929$, $937$, $953$, $977$, $1009$, $1049$, $1097$, 
$1129$, $1153$, $1193$, $1201$, $1217$, $1249$, $1361$, $1409$, $1433$, $1489$, 
$1553$, $1601$, $1609$, $1657$, $1721$, $1777$, $1801$, $2017$, $2089$,\,$\ldots$$\}$;

\noindent
exceptions are $\ell \in \!\{$$17$, $73$, $113$, $193$, $241$, $257$, $641$, $1033$, 
$1289$, $1297$, $1321$, $1481$, $1697$, $1753$, $1873$, $1889$, $1913$, 
$1993$, $2081$,\,$\ldots$$\}$.

\subsubsection{Cubic field of conductor $f=20887$} \label{20887}
We get the more complex structure 
$\CH_K \simeq \BZ/4\BZ \times \BZ/2\BZ$; a complete capitulation in $K_1$ does not hold
since $e(\Kk) = 2$, but computations in $K_3$ are out of reach. However, taking ${\sf Nell=3}$,
the results allow to distinguish between incomplete capitulation in $K_2$ and possible 
capitulation in $K_n$, $n \geq 3$; we obtain the following matrices showing always an 
incomplete capitulation and some stabilities, up to $\ell = 449$ (the mention 
${\sf Im(Jn)=[a,...,z]}$ denotes the structure of the image $\J_{K_n/K}(\CH_K)$, to be compared 
with ${\sf CK0=[4,4,2,2]}$):

\ft\begin{verbatim}
p=2  f=20887  PK=x^3+x^2-6962*x-225889 
\end{verbatim}\ns
\ft\begin{verbatim}
ell=17  CK0=[4,4,2,2]  CK1=[8,8,2,2]  CK2=[16,16,2,2]  N=3  r=3
norm in K1/K of the component 1 of CK1:[2,0,0,0]
norm in K1/K of the component 2 of CK1:[0,2,0,0]
norm in K1/K of the component 3 of CK1:[0,0,0,0]
norm in K1/K of the component 4 of CK1:[0,0,0,0]
Im(J1)=[4,4]
norm in K2/K of the component 1 of CK2:[4,0,0,0]
norm in K2/K of the component 2 of CK2:[0,4,0,0]
norm in K2/K of the component 3 of CK2:[0,0,0,0]
norm in K2/K of the component 4 of CK2:[0,0,0,0]
Im(J2)=[4,4]
\end{verbatim}\ns
\ft\begin{verbatim}
ell=97  CK0=[4,4,2,2]  CK1=[8,8,2,2]  CK2=[8,8,2,2]  N=4  r=3
norm in K2/K of the component 1 of CK2:[4,0,0,0]
norm in K2/K of the component 2 of CK2:[0,4,0,0]
norm in K2/K of the component 3 of CK2:[0,0,0,0]
norm in K2/K of the component 4 of CK2:[0,0,0,0]
Im(J2)=[2,2], Capitulation in K3
\end{verbatim}\ns
\ft\begin{verbatim}
ell=353  CK0=[4,4,2,2]  CK1=[4,4,4,4,2,2]  CK2=[8,8,4,4,2,2]  N=4  r=3
norm in K1/K of the component 1 of CK1:[2,2,0,0,0,0]
norm in K1/K of the component 2 of CK1:[2,0,0,0,0,0]
norm in K1/K of the component 3 of CK1:[3,1,2,2,1,1]
norm in K1/K of the component 4 of CK1:[1,2,0,2,1,0]
norm in K1/K of the component 5 of CK1:[2,2,0,0,0,0]
norm in K1/K of the component 6 of CK1:[0,0,0,0,0,0]
Im(J1)=[4,4,2,2], no capitulation (J1 injective)
norm in K2/K of the component 1 of CK2:[0,0,2,2,0,0]
norm in K2/K of the component 2 of CK2:[0,0,0,2,0,0]
norm in K2/K of the component 3 of CK2:[0,0,0,0,0,0]
norm in K2/K of the component 4 of CK2:[0,0,0,0,0,0]
norm in K2/K of the component 5 of CK2:[0,0,0,0,0,0]
norm in K2/K of the component 6 of CK2:[0,0,0,0,0,0]
Im(J2)=[2,2]
\end{verbatim}\ns
\ft\begin{verbatim}
ell=433  CK0=[4,4,2,2]  CK1=[8,8,2,2,2,2]  CK2=[8,8,4,4,2,2]  N=3  r=1
norm in K2/K of the component 1 of CK2:[0,4,2,0,0,0]
norm in K2/K of the component 2 of CK2:[0,4,0,2,0,0]
norm in K2/K of the component 3 of CK2:[0,0,0,0,0,0]
norm in K2/K of the component 4 of CK2:[0,0,0,0,0,0]
norm in K2/K of the component 5 of CK2:[0,0,0,0,0,0]
norm in K2/K of the component 6 of CK2:[0,0,0,0,0,0]
Im(J2)=[2,2]
\end{verbatim}\ns
\ft\begin{verbatim}
ell=769  CK0=[4,4,2,2]  CK1=[4,4,2,2,2,2]  CK2=[4,4,4,4,4,4]  N=7  r=1
norm in K1/K of the component 1 of CK1:[2,0,0,0,0,0]
norm in K1/K of the component 2 of CK1:[0,2,1,1,1,0]
norm in K1/K of the component 3 of CK1:[0,0,0,1,0,0]
norm in K1/K of the component 4 of CK1:[0,0,0,0,0,0]
norm in K1/K of the component 5 of CK1:[0,0,0,1,0,0]
norm in K1/K of the component 6 of CK1:[0,0,1,0,1,0]
Im(J1)=[2,2,2,2]
norm in K2/K of the component 1 of CK2:[0,0,0,0,0,0]
norm in K2/K of the component 2 of CK2:[0,0,0,0,0,0]
norm in K2/K of the component 3 of CK2:[0,0,0,0,2,0]
norm in K2/K of the component 4 of CK2:[0,2,0,0,2,2]
norm in K2/K of the component 5 of CK2:[0,0,0,0,0,0]
norm in K2/K of the component 6 of CK2:[0,0,0,0,0,0]
Im(J2)=[2,2]
\end{verbatim}\ns
\ft\begin{verbatim}
ell=977  CK0=[4,4,2,2]  CK1=[8,8,4,4]  CK2=[16,16,4,4]  N=3  r=3
norm in K2/K of the component 1 of CK2:[4,8,0,0]
norm in K2/K of the component 2 of CK2:[8,4,0,0]
norm in K2/K of the component 3 of CK2:[0,8,0,0]
norm in K2/K of the component 4 of CK2:[8,0,0,0]
Im(J2)=[4,4]
\end{verbatim}\ns
\ft\begin{verbatim}
ell=1009  CK0=[4,4,2,2]  CK1=[8,8,4,4]  CK2=[8,8,4,4]  N=3  r=1
norm in K2/K of the component 1 of CK2:[4,0,0,0]
norm in K2/K of the component 2 of CK2:[0,0,0,0]
norm in K2/K of the component 3 of CK2:[4,0,0,0]
norm in K2/K of the component 4 of CK2:[4,4,0,0]
Im(J2)=[2,2], Capitulation in K3
\end{verbatim}\ns

\subsubsection{Remarks about the towers $K(\mu_\ell^{})/K$, $K$
of conductor $f=20887$}
In the above values of $\ell$, there is always partial capitulation when $r=1$ but 
never stability from $K$; in other words, if a stability from $K_{n_0}$ does exist, then 
$n_0 \geq 1$ (see Remark~\ref{nzero}).

This may be explained as follows and many generalizations are possible. 
For simplicity, put $\Norm_{K_1/K} =: \Norm$,
$\J_{K_1/K} =: \J$, $\CH_{K_1} =: \CH_1$, $\BE_{K_1} =: \BE_1$,
$\CH_K =: \CH$, $\BE_K =: \BE$.

\noindent
Stability from $K$ would imply $\CH_1^{G_1} = \CH_1 \ds \mathop
{\simeq}^{\hbox{\tiny$\Norm$}} \CH \simeq \BZ/4\BZ \times \BZ/2\BZ$ 
(Theorem \ref{main2}), Chevalley--Herbrand formula would imply 
$\BE / \BE \cap \Norm(K_1^\times) \simeq 1$ or $\BZ/2\BZ$ 
depending on $r=1$ or $r=3$, and $\CH_1^{G_1} = \CH_1$ 
would imply $\Norm(\CH_1^{G_1}) = \CH$. 

Let's examine the consequences according to the value of $r$:

(i) Case $r=3$. Since $\BE/ \BE^2 \simeq \BZ/2\BZ$, the condition 
$\BE/ \BE \cap \Norm(K_1^\times) \simeq \BZ/2\BZ$ implies 
$\BE \cap \Norm(K_1^\times) = \BE^2 = \Norm(\BE_1)$, whence 
$\CH_1^{G_1} = \J(\CH) \cdot \CH_1^\ram$, from exact sequence \ref{suite},
thus $\Norm(\CH_1^{G_1}) = \CH^2 \cdot \CH^\ram = \CH$; we deduce
from this, $\CH =  \CH^2 \cdot \CH^\ram = (\CH^2 \cdot \CH^\ram)^2
\cdot \CH^\ram =  \CH^4 \cdot \CH^\ram =  \CH^\ram$ of $2$-rank $\leq 2$ (absurd).
The case $\ell=353$ gives an example of injective transfer.

(ii) Case $r=1$. The norm factor is trivial, $\CH_1^\ram=1$ and
 exact sequence \ref{suite} becomes $1 \to \J(\CH) \to
\CH_1^{G_1} \to \BE/\Norm(\BE_1) \to 1$ with $\CH_1^{G_1} \ds 
\mathop{\simeq}^{\hbox{\tiny$\Norm$}} \CH$ and $\BE/\Norm(\BE_1)$ 
isomorphic to $1$, $\BZ/2\BZ$ or $\BZ/4\BZ$; but
Theorem \ref{nocap} implies $\J$ non injective (otherwise we get
$\order \CH_1 \geq \order \CH\cdot \order \CH[2] > \order \CH$), 
whence a partial capitulation in $K_1$ (so, $\J(\CH)$ is isomorphic 
to $\BZ/4\BZ$, $\BZ/2\BZ \times \BZ/2\BZ$ or $\BZ/2\BZ$); but 
$\Norm : \CH_1 = \CH_1^{G_1} \to \CH$ being an isomorphism, 
we have $\Norm(\J(\CH)) = \CH^2 \simeq \BZ/2\BZ$, then $\CH/\CH^2 
\simeq \BZ/2\BZ  \times \BZ/2\BZ \simeq  \BE/\Norm(\BE_1)$ (absurd).

This gives examples of values of $n_0=1$ for $\ell = 97$ and $1009$.

\section{Tables for Kummer fields and \texorpdfstring{$p=2$}{Lg}}
The purpose is to consider fields of the form $K=\Q(\sqrt[q]{R})$, $q \geq 3$,
with $L \subset K(\mu_\ell^{})$, $\ell \equiv 1 \pmod {2 p^\N}$;
although these fields are not totally real, it is known that capitulation may exist in 
compositum $L = KL_0$, with suitable abelian $p$-extensions $L_0/\Q$
(conjectured in \cite{Gras1997} (1997), proved by Bosca \cite{Bosc2009}
(2009)).

\subsection{Pure cubic fields, \texorpdfstring{$\ell = 17$}{Lg}}
We consider the set of pure cubic fields $K=\Q(\sqrt[3]{R})$; so
$L/\Q$ is not Galois, but, by chance, the instruction
${\sf G=nfgaloisconj(Kn)}$ of PARI computes the group of automorphisms,
whence $\Gal(K_n/K)$ in our case; this simplifies the search of ${\sf S}$
of order $p^n$. 

Taking ${\sf p=2}$, $\ell = 17$, ${\sf Nn=3}$ and restricting to fields $K$ such 
that $\order \CH_K \geq 4$, we obtain many capitulations (the program eliminates
the cases of stability from $K$):

\ft\begin{verbatim}
MAIN PROGRAM FOR PURE CUBIC FIELDS:
{p=2;Nn=3;vHK=2;ell=17;mKn=2;for(R=2,10^4,PK=x^3-R;
if(polisirreducible(PK)==0,next);
K=bnfinit(PK,1);r=matsize(idealfactor(K,ell))[1];
\\Testing the order of the p-class group of K compared to vHK:
HK=K.no;if(valuation(HK,p)<vHK,next);CK0=K.clgp;
for(n=1,Nn,Qn=polsubcyclo(ell,p^n);Pn=polcompositum(PK,Qn)[1];
Kn=bnfinit(Pn,1);HKn=Kn.no;
\\Test for elimination of the stability from K:
if(n==1 & valuation(HKn,p)==valuation(HK,p),break);
if(n==1,print();print("PK=",PK," CK0=",CK0[2]," ell=",ell," r=",r));
CKn=Kn.clgp;print("CK",n,"=",CKn[2]);rKn=matsize(CKn[2])[2];
G=nfgaloisconj(Kn);Id=x;for(k=1,p^n,Z=G[k];ks=1;while(Z!=Id,
Z=nfgaloisapply(Kn,G[k],Z);ks=ks+1);if(ks==p^n,S=G[k];break));
for(j=1,rKn,X=CKn[3][j];Y=X;for(i=1,mKn,YS=nfgaloisapply(Kn,S,Y);
T=idealpow(Kn,Y,-1);Y=idealmul(Kn,YS,T);B=bnfisprincipal(Kn,Y)[1];
Ehij=List;for(j=1,rKn,c=B[j];w=valuation(CKn[2][j],p);c=lift(Mod(c,p^w)); 
listput(Ehij,c,j));print("h_",j,"^[","(S-1)^",i,"]=",Ehij)));
for(j=1,rKn,A0=CKn[3][j];A=1;for(t=1,p^n,As=nfgaloisapply(Kn,S,A);
A=idealmul(Kn,A0,As));B=bnfisprincipal(Kn,A)[1];
Enu=List;for(j=1,rKn,c=B[j];w=valuation(CKn[2][j],p);c=lift(Mod(c,p^w)); 
listput(Enu,c,j));print("norm in K",n,"/K of the component ",j,
" of CK",n,":",Enu))))}
\end{verbatim}\ns
\ft\begin{verbatim}
p=2  PK=x^3-43  CK0=[12]  ell=17  r=2
CK1=[12,6]=[4,2]
h_1^[(S-1)^1]=[0,1]   h_2^[(S-1)^1]=[0,0]
h_1^[(S-1)^2]=[0,0]   h_2^[(S-1)^2]=[0,0]
norm in K1/K of the component 1 of CK1:[2,1]
norm in K1/K of the component 2 of CK1:[0,0]
Incomplete capitulation, m(K1)=2, e(K1)=2
CK2=[12,12]=[4,4]
h_1^[(S-1)^1]=[0,1]   h_2^[(S-1)^1]=[0,0]
h_1^[(S-1)^2]=[0,0]   h_2^[(S-1)^2]=[0,0]
norm in K2/K of the component 1 of CK2:[0,2]
norm in K2/K of the component 2 of CK2:[0,0]
Incomplete capitulation, m(K2)=2, e(K2)=2
CK3=[12,12]=[4,4]
h_1^[(S-1)^1]=[0,3]   h_2^[(S-1)^1]=[0,0]
h_1^[(S-1)^2]=[0,0]   h_2^[(S-1)^2]=[0,0]
norm in K3/K of the component 1 of CK3:[0,0]
norm in K3/K of the component 2 of CK3:[0,0]
Complete capitulation (stability from K2), m(K3)=2, e(K3)=2 
\end{verbatim}\ns
\ft\begin{verbatim}
p=2  PK=x^3-113  CK0=[2,2]  ell=17  r=2
CK1=[6,2,2]=[2,2,2]
h_1^[(S-1)^1]=[0,1,0]  h_2^[(S-1)^1]=[0,0,0]  h_3^[(S-1)^1]=[0,0,0]
h_1^[(S-1)^2]=[0,0,0]  h_2^[(S-1)^2]=[0,0,0]  h_3^[(S-1)^2]=[0,0,0]
norm in K1/K of the component 1 of CK1:[0,1,0]
norm in K1/K of the component 2 of CK1:[0,0,0]
norm in K1/K of the component 3 of CK1:[0,0,0]
Incomplete capitulation, m(K1)=2, e(K1)=1
CK2=[6,2,2,2,2]=[2,2,2,2,2]
h_1^[(S-1)^1]=[0,0,0,1,0] h_2^[(S-1)^1]=[0,1,1,1,0] 
h_3^[(S-1)^1]=[0,1,1,1,0] h_4^[(S-1)^1]=[1,1,0,0,0] 
h_5^[(S-1)^1]=[0,0,0,0,0]
h_1^[(S-1)^2]=[1,1,0,0,0] h_2^[(S-1)^2]=[1,1,0,0,0] 
h_3^[(S-1)^2]=[1,1,0,0,0] h_4^[(S-1)^2]=[0,1,1,0,0] 
h_5^[(S-1)^2]=[0,0,0,0,0]
norm in K2/K of the component 1 of CK2:[0,1,1,0,0]
norm in K2/K of the component 2 of CK2:[0,1,1,0,0]
norm in K2/K of the component 3 of CK2:[0,1,1,0,0]
norm in K2/K of the component 4 of CK2:[0,0,0,0,0]
norm in K2/K of the component 5 of CK2:[0,0,0,0,0]
Incomplete capitulation, m(K2)=3, e(K2)=1
CK3=[12,2,2,2,2]=[4,2,2,2,2]
h_1^[(S-1)^1]=[2,1,0,0,1] h_2^[(S-1)^1]=[0,0,1,1,0] 
h_3^[(S-1)^1]=[0,1,0,0,0] h_4^[(S-1)^1]=[2,0,1,1,1] 
h_5^[(S-1)^1]=[0,1,0,0,1]
h_1^[(S-1)^2]=[0,1,1,1,1] h_2^[(S-1)^2]=[2,1,1,1,1] 
h_3^[(S-1)^2]=[0,0,1,1,0] h_4^[(S-1)^2]=[2,0,1,1,0] 
h_5^[(S-1)^2]=[0,1,1,1,1]
norm in K3/K of the component 1 of CK3:[0,0,0,0,0]
norm in K3/K of the component 2 of CK3:[0,0,0,0,0]
norm in K3/K of the component 3 of CK3:[0,0,0,0,0]
norm in K3/K of the component 4 of CK3:[0,0,0,0,0]
norm in K3/K of the component 5 of CK3:[0,0,0,0,0]
Complete capitulation, m(K3)=4, e(K3)=2
\end{verbatim}\ns
\ft\begin{verbatim}
p=2  PK=x^3-122  CK0=[12]  ell=17  r=2
CK1=[12,4]=[4,4]
h_1^[(S-1)^1]=[0,1]      h_2^[(S-1)^1]=[0,2]
h_1^[(S-1)^2]=[0,2]      h_2^[(S-1)^2]=[0,0]
norm in K1/K of the component 1 of CK1:[2,1]
norm in K1/K of the component 2 of CK1:[0,0]
Incomplete capitulation, m(K1)=3, e(K1)=2
CK2=[12,4]=[4,4]
h_1^[(S-1)^1]=[0,3]   h_2^[(S-1)^1]=[0,2]
h_1^[(S-1)^2]=[0,2]   h_2^[(S-1)^2]=[0,0]
norm in K2/K of the component 1 of CK2:[0,2]
norm in K2/K of the component 2 of CK2:[0,0]
Incomplete capitulation, m(K2)=3, e(K2)=2
CK3=[12,4]=[4,4]
h_1^[(S-1)^1]=[0,0]      h_2^[(S-1)^1]=[3,2]
h_1^[(S-1)^2]=[0,0]      h_2^[(S-1)^2]=[2,0]
norm in K3/K of the component 1 of CK3:[0,0]
norm in K3/K of the component 2 of CK3:[0,0]
Complete capitulation (stability from K1), m(K3)=3, e(K3)=2
\end{verbatim}\ns
\ft\begin{verbatim}
p=2  PK=x^3-141  CK0=[4,2]  ell=17  r=2
CK1=[8,2]
h_1^[(S-1)^1]=[0,0]      h_2^[(S-1)^1]=[0,0]
norm in K1/K of the component 1 of CK1:[2,0]
norm in K1/K of the component 2 of CK1:[0,0]
Incomplete capitulation, m(K1)=1, e(K1)=3
CK2=[16,2]
h_1^[(S-1)^1]=[0,0]      h_2^[(S-1)^1]=[0,0]
norm in K2/K of the component 1 of CK2:[4,0]
norm in K2/K of the component 2 of CK2:[0,0]
Incomplete capitulation, m(K2)=1, e(K2)=4
CK3=[288,18]=[32,2]
h_1^[(S-1)^1]=[0,0]      h_2^[(S-1)^1]=[0,0]
norm in K3/K of the component 1 of CK3:[8,0]
norm in K3/K of the component 2 of CK3:[0,0]
Incomplete capitulation, m(K3)=1, e(K3)=5
\end{verbatim}\ns
\ft\begin{verbatim}
p=2  PK=x^3-174  CK0=[6,2]  ell=17  r=2
CK1=[12,6,2]=[4,2,2]
h_1^[(S-1)^1]=[2,1,0]  h_2^[(S-1)^1]=[0,0,0]  h_3^[(S-1)^1]=[0,0,0]
h_1^[(S-1)^2]=[0,0,0]  h_2^[(S-1)^2]=[0,0,0]  h_3^[(S-1)^2]=[0,0,0]
norm in K1/K of the component 1 of CK1:[0,1,0]
norm in K1/K of the component 2 of CK1:[0,0,0]
norm in K1/K of the component 3 of CK1:[0,0,0]
Incomplete capitulation, m(K1)=2, e(K1)=2
CK2=[2040,12,2,2]=[8,4,2,2]
h_1^[(S-1)^1]=[4,0,1,1]  h_2^[(S-1)^1]=[0,0,0,1]  
h_3^[(S-1)^1]=[0,2,0,1]  h_4^[(S-1)^1]=[0,2,0,0]
h_1^[(S-1)^2]=[0,0,0,1]  h_2^[(S-1)^2]=[0,2,0,0]  
h_3^[(S-1)^2]=[0,2,0,0]  h_4^[(S-1)^2]=[0,0,0,0]
norm in K2/K of the component 1 of CK2:[4,2,0,0]
norm in K2/K of the component 2 of CK2:[0,0,0,0]
norm in K2/K of the component 3 of CK2:[0,0,0,0]
norm in K2/K of the component 4 of CK2:[0,0,0,0]
Incomplete capitulation, m(K2)=4, e(K2)=3
CK3=[4080,12,2,2]=[16,4,2,2]
\end{verbatim}\ns

Unfortunately, the computations for $n=3$, in this last example, take too much time.
We are going to examine separately this field.

\subsection{Case of \texorpdfstring{$K\!=\!\Q(\sqrt[3]{174})$}{Lg}, 
\texorpdfstring{$\ell \equiv 1 \!\!\pmod {16}$}{Lg}}

Varying $\ell$, we find many capitulations in the layer $K_2$:

\ft\begin{verbatim}
p=2  PK=x^3-174  CK0=[6,2]  ell=193  r=1
CK1=[12,6]=[4,2]
h_1^[(S-1)^1]=[2,0]   h_2^[(S-1)^1]=[2,0]
h_1^[(S-1)^2]=[0,0]   h_2^[(S-1)^2]=[0,0]
norm in K1/K of the component 1 of CK1:[0,0]
norm in K1/K of the component 2 of CK1:[2,0]
Incomplete capitulation, m(K1)=2, e(K1)=2
CK2=[12,6]=[4,2]
h_1^[(S-1)^1]=[0,0]   h_2^[(S-1)^1]=[2,0]
h_1^[(S-1)^2]=[0,0]   h_2^[(S-1)^2]=[0,0]
norm in K2/K of the component 1 of CK2:[0,0]
norm in K2/K of the component 2 of CK2:[0,0]
Complete capitulation, m(K2)=2, e(K2)=2
\end{verbatim}\ns
\ft\begin{verbatim}
p=2  PK=x^3-174  CK0=[6,2]  ell=353  r=2
CK1=[48,6]=[16,2]
h_1^[(S-1)^1]=[6,0]   h_2^[(S-1)^1]=[8,0]
h_1^[(S-1)^2]=[4,0]   h_2^[(S-1)^2]=[0,0]
norm in K1/K of the component 1 of CK1:[8,0]
norm in K1/K of the component 2 of CK1:[8,0]
Incomplete capitulation, m(K1)=4, e(K1)=4
CK2=[48,6,3,3]=[16,2]
h_1^[(S-1)^1]=[6,0,0,0]  h_2^[(S-1)^1]=[8,0,0,0]  
h_1^[(S-1)^2]=[4,0,0,0]  h_2^[(S-1)^2]=[0,0,0,0]  
norm in K2/K of the component 1 of CK2:[0,0,0,0]
norm in K2/K of the component 2 of CK2:[0,0,0,0]
Complete capitulation, m(K2)=4, e(K2)=4
\end{verbatim}\ns
\ft\begin{verbatim}
p=2  PK=x^3-174  CK0=[6,2]  ell=577  r=1
CK1=[84,6,2]=[4,2,2]
h_1^[(S-1)^1]=[0,1,1]  h_2^[(S-1)^1]=[2,1,1]  h_3^[(S-1)^1]=[2,1,1]
h_1^[(S-1)^2]=[0,0,0]  h_2^[(S-1)^2]=[0,0,0]  h_3^[(S-1)^2]=[0,0,0]
norm in K1/K of the component 1 of CK1:[2,1,1]
norm in K1/K of the component 2 of CK1:[2,1,1]
norm in K1/K of the component 3 of CK1:[2,1,1]
Incomplete capitulation, m(K1)=2, e(K1)=2
CK2=[168,6,6,3]=[8,2,2]
h_1^[(S-1)^1]=[2,0,0,0] h_2^[(S-1)^1]=[4,0,0,0] h_3^[(S-1)^1]=[4,1,0,0]
h_1^[(S-1)^2]=[4,0,0,0] h_2^[(S-1)^2]=[0,0,0,0] h_3^[(S-1)^2]=[4,0,0,0]
h_1^[(S-1)^3]=[0,0,0,0] h_2^[(S-1)^3]=[0,0,0,0] h_3^[(S-1)^3]=[0,0,0,0]
norm in K2/K of the component 1 of CK2:[0,0,0,0]
norm in K2/K of the component 2 of CK2:[0,0,0,0]
norm in K2/K of the component 3 of CK2:[0,0,0,0]
Complete capitulation, m(K2)=3, e(K2)=3
\end{verbatim}\ns

The last case $\ell = 577$ shows that the complexity of the  
$\CH_{K_n}$'s is increasing, but nevertheless leads to 
complete capitulation in $K_2$; so conditions of Theorem \ref{main1} are not 
necessary for capitulation. Indeed, we obtain the following information on the structure 
of the $\CH_{K_n}$'s for $K=\Q(\sqrt[3]{174})$ and $\ell = 577$:

In $K_1$, $m(\Kk_1)=2$, $s(\Kk_1)=1$ but $e(\Kk_1)=2$. In $K_2$, $m(\Kk_2) = 3$, 
$s(\Kk_1)=1$, $e(\Kk_2)=3$; thus $\Nu_{\!K_2/K}$ acts like $4(\sigma-1)^2+6(\sigma-1)+4$. 
The above data show that this reduces to the annihilation by $A=6(\sigma-1)+4$; indeed, 
$h_1^A=h_1^{12}h_1^4=1$, $h_j^A=1$ for the other generators.

\subsection{Pure quintic fields, \texorpdfstring{$L \subset K(\mu_{17}^{})$}{Lg}}
Replacing, in the program, the polynomial ${\sf PK=x^3-R}$ by ${\sf PK=x^5-R}$ 
(or with any odd degree), we get analogous results:

\ft\begin{verbatim}
p=2  PK=x^5-13  CK0=[4]  ell=17  r=2
CK1=[40]=[8]
h_1^[(S-1)^1]=[0]  h_1^[(S-1)^2]=[0]
norm in K1/K of the component 1 of CK1:[2]
No capitulation, m(K1)=1, e(K1)=3
CK2=[40]=[8]
h_1^[(S-1)^1]=[0]  h_1^[(S-1)^2]=[0]
norm in K2/K of the component 1 of CK2:[4]
Incomplete capitulation, m(K2)=1, e(K2)=3
Complete capitulation in K3 (stability from K1)
\end{verbatim}\ns
\ft\begin{verbatim}
p=2  PK=x^5-122  CK0=[10,2]  ell=17  r=2
CK1=[10,2,2]=[2,2,2]
h_1^[(S-1)^1]=[0,0,0] h_2^[(S-1)^1]=[0,0,0] h_3^[(S-1)^1]=[1,1,0]
h_1^[(S-1)^2]=[0,0,0] h_2^[(S-1)^2]=[0,0,0] h_3^[(S-1)^2]=[0,0,0]
norm in K1/K of the component 1 of CK1:[0,0,0]
norm in K1/K of the component 2 of CK1:[0,0,0]
norm in K1/K of the component 3 of CK1:[1,1,0]
Incomplete capitulation, m(K1)=2, e(K1)=1
CK2=[20,2,2]=[4,2,2]
h_1^[(S-1)^1]=[2,0,1] h_2^[(S-1)^1]=[2,0,1] h_3^[(S-1)^1]=[2,0,0]
h_1^[(S-1)^2]=[2,0,0] h_2^[(S-1)^2]=[2,0,0] h_3^[(S-1)^2]=[0,0,0]
norm in K2/K of the component 1 of CK2:[0,0,0]
norm in K2/K of the component 2 of CK2:[0,0,0]
norm in K2/K of the component 3 of CK2:[0,0,0]
Complete capitulation, m(K2)=3, e(K2)=2
\end{verbatim}\ns

\section{Tables for real quadratic fields and \texorpdfstring{$p=3$}{Lg}}\label{quadratic}

We consider cyclic $p$-extensions $L/K$, where $K=\Q(\sqrt m)$,
$L \subset  K(\mu_{\ell}^{})$, $\ell \equiv 1 \pmod {2p^\N}$. Results are given 
for $p=3$, $\ell = 19$ ($N=2$), $109$ ($N=3$) and $163$ ($N=4$),
except cases of stability in $K_1/K$. The images 
$\J_{K_n/K}(\CH_K)$ are computed for $n=1$ and $n=2$; 
$r \in \{1, 2\}$ is the number of prime ideals above $\ell$ in $K$:

\ft\begin{verbatim}
MAIN PROGRAM FOR REAL QUADRATIC FIELDS
{p=3;Nn=2;bm=2;Bm=10^8;vHK=1;ell=109;mKn=2;
for(m=bm,Bm,if(core(m)!=m,next);PK=x^2-m;K=bnfinit(PK,1);
\\Testing the order of the p-class group of K compared to vHK:
HK=K.no;if(valuation(HK,p)<vHK,next);CK0=K.clgp;r=(kronecker(m,ell)+3)/2;
for(n=1,Nn,Qn=polsubcyclo(ell,p^n);Pn=polcompositum(PK,Qn)[1];
Kn=bnfinit(Pn,1);HKn=Kn.no;dn=poldegree(Pn);
\\Test for elimination of the stability from K:
if(n==1 & valuation(HKn,p)==valuation(HK,p),break);
if(n==1,print();print("PK=",PK," CK0=",CK0[2]," ell=",ell," r=",r));
CKn=Kn.clgp;print("CK",n,"=",CKn[2]);rKn=matsize(CKn[2])[2];
G=nfgaloisconj(Kn);Id=x;for(k=1,dn,Z=G[k];ks=1;while(Z!=Id,
Z=nfgaloisapply(Kn,G[k],Z);ks=ks+1);if(ks==p^n,S=G[k];break));
for(j=1,rKn,X=CKn[3][j];Y=X;for(i=1,mKn,YS=nfgaloisapply(Kn,S,Y);
T=idealpow(Kn,Y,-1);Y=idealmul(Kn,YS,T);B=bnfisprincipal(Kn,Y)[1];
Ehij=List;for(j=1,rKn,c=B[j];w=valuation(CKn[2][j],p);c=lift(Mod(c,p^w)); 
listput(Ehij,c,j));print("h_",j,"^[","(S-1)^",i,"]=",Ehij)));
for(j=1,rKn,A0=CKn[3][j];A=1;for(t=1,p^n,As=nfgaloisapply(Kn,S,A);
A=idealmul(Kn,A0,As));B=bnfisprincipal(Kn,A)[1];
Enu=List;for(j=1,rKn,c=B[j];w=valuation(CKn[2][j],p);c=lift(Mod(c,p^w)); 
listput(Enu,c,j));print("norm in K",n,"/K of the component ",j,
" of CK",n,":",Enu))))}
\end{verbatim}\ns

\subsection{Examples with \texorpdfstring{$\CH_K \simeq \Z/3\Z$}{Lg}}

This simplest case gives both capitulations or non-capitulations in $K_2$:

\ft\begin{verbatim}
PK=x^2-142  CK0=[3]  ell=19  r=2
CK1=[9]
h_1^[(S-1)^1]=[0]        h_1^[(S-1)^2]=[0]
norm in K1/K of the component 1 of CK1:[3]
No capitulation, m(K1)=1, e(K1)=2
CK2=[9]
h_1^[(S-1)^1]=[0]        h_1^[(S-1)^2]=[0]
norm in K2/K of the component 1 of CK2:[0]
Complete capitulation, m(K2)=1, e(K2)=2
\end{verbatim}\ns
\ft\begin{verbatim}
PK=x^2-359  CK0=[3]  ell=19  r=2
CK1=[9]
h_1^[(S-1)^1]=[0]        h_1^[(S-1)^2]=[0]
norm in K1/K of the component 1 of CK1:[3]
No capitulation, m(K1)=1, e(K1)=2
CK2=[27]
h_1^[(S-1)^1]=[0]        h_1^[(S-1)^2]=[0]
norm in K2/K of the component 1 of CK2:[9]
No capitulation, m(K2)=1, e(K2)=3
\end{verbatim}\ns
\ft\begin{verbatim}
p=3  PK=x^2-142  CK0=[3]  ell=109  r=1
CK1=[18,2]=[9]
h_1^[(S-1)^1]=[3,0]        h_1^[(S-1)^2]=[0,0]       
norm in K1/K of the component 1 of CK1:[3,0]
No capitulation, m(K1)=2, e(K1)=2
CK2=[54,2]=[27]
h_1^[(S-1)^1]=[24,0]      h_1^[(S-1)^2]=[9,0]        
norm in K2/K of the component 1 of CK2:[9,0]
No capitulation, m(K2)=3, e(K2)=3
\end{verbatim}\ns
\ft\begin{verbatim}
p=3  PK=x^2-223  CK0=[3]  ell=109  r=2
CK1=[9]
h_1^[(S-1)^1]=[6]         h_1^[(S-1)^2]=[0]
norm in K1/K of the component 1 of CK1:[3]
No capitulation, m(K1)=2, e(K1)=2
CK2=[9]
h_1^[(S-1)^1]=[3]         h_1^[(S-1)^2]=[0]
norm in K2/K of the component 1 of CK2:[0]
Complete capitulation, m(K2)=2, e(K2)=2
\end{verbatim}\ns
\ft\begin{verbatim}
p=3  PK=x^2-254  CK0=[3]  ell=109  r=2
CK1=[3,3]
h_1^[(S-1)^1]=[2,2]         h_2^[(S-1)^1]=[1,1]
h_1^[(S-1)^2]=[0,0]         h_2^[(S-1)^2]=[0,0]
norm in K1/K of the component 1 of CK1:[0,0]
norm in K1/K of the component 2 of CK1:[0,0]
Complete capitulation, m(K1)=2, e(K1)=1
CK2=[3,3]
h_1^[(S-1)^1]=[1,1]         h_2^[(S-1)^1]=[2,2]
h_1^[(S-1)^2]=[0,0]         h_2^[(S-1)^2]=[0,0]
norm in K2/K of the component 1 of CK2:[0,0]
norm in K2/K of the component 2 of CK2:[0,0]
Complete capitulation, m(K2)=2, e(K2)=1
\end{verbatim}\ns
\ft\begin{verbatim}
p=3  PK=x^2-79  CK0=[3]  ell=163  r=1
CK1=[18,2]=[9]
h_1^[(S-1)^1]=[3,0]         h_1^[(S-1)^2]=[0,0]         
norm in K1/K of the component 1 of CK1:[3,0]
norm in K1/K of the component 2 of CK1:[0,0]
No capitulation, m(K1)=2, e(K1)=2
CK2=[18,2]=[9]
h_1^[(S-1)^1]=[3,0]         h_1^[(S-1)^2]=[0,0]         
norm in K2/K of the component 1 of CK2:[0,0]
norm in K2/K of the component 2 of CK2:[0,0]
Complete capitulation, m(K2)=2, e(K2)=2
\end{verbatim}\ns
\ft\begin{verbatim}
p=3  PK=x^2-254  CK0=[3]  ell=163  r=2 
CK1=[18,2]=[9]
h_1^[(S-1)^1]=[0,0]        h_1^[(S-1)^2]=[0,0]        
norm in K1/K of the component 1 of CK1:[3,0]
norm in K1/K of the component 2 of CK1:[0,0]
No capitulation, m(K1)=1, e(K1)=2
CK2=[18,2]=[9]
h_1^[(S-1)^1]=[0,0]        h_1^[(S-1)^2]=[0,0]        
norm in K2/K of the component 1 of CK2:[0,0]
norm in K2/K of the component 2 of CK2:[0,0]
Complete capitulation, m(K2)=1, e(K2)=2
\end{verbatim}\ns

\subsection{Examples with \texorpdfstring{$\CH_K \simeq \Z/3\Z \times \Z/3\Z$}{Lg}}
Due to a very large calculation time for degrees $[K_2 : \Q] = 18$, we have 
only some results showing that, as for the case of cubic fields and $p=2$ (degrees 
$[K_2 : \Q] = 12$), capitulation may occur in $K_2$:

\ft\begin{verbatim}
p=3  PK=x^2-23659  CK0=[6,3]  ell=19  r=2
CK1=[18,3]=[9,3]
h_1^[(S-1)^1]=[0,0]        h_2^[(S-1)^1]=[3,0]
h_1^[(S-1)^2]=[0,0]        h_2^[(S-1)^2]=[0,0]
norm in K1/K of the component 1 of CK1:[3,0]
norm in K1/K of the component 2 of CK1:[0,0]
Incomplete capitulation, m(K1)=2 ,e(K1)=2
CK2=[18,3]=[9,3]
h_1^[(S-1)^1]=[0,0]        h_2^[(S-1)^1]=[3,0]
h_1^[(S-1)^2]=[0,0]        h_2^[(S-1)^2]=[0,0]
norm in K2/K of the component 1 of CK2:[0,0]
norm in K2/K of the component 2 of CK2:[0,0]
Complete capitulation, m(K2)=2, e(K2)=2
\end{verbatim}\ns
\ft\begin{verbatim}
p=3  PK=x^2-23659  CK0=[6,3]  ell=37  r=2
CK1=[18,3,3]=[9,3,3]
h_1^[(S-1)^1]=[0,0,0]  h_2^[(S-1)^1]=[6,0,1]  h_3^[(S-1)^1]=[6,0,0]
h_1^[(S-1)^2]=[0,0,0]  h_2^[(S-1)^2]=[6,0,0]  h_3^[(S-1)^2]=[0,0,0]
norm in K1/K of the component 1 of CK1:[0,0,0]
norm in K1/K of the component 2 of CK1:[3,0,0]
norm in K1/K of the component 3 of CK1:[3,0,0]
Incomplete capitulation, m(K1)=3 ,e(K1)=2
CK2=[18,3,3]=[9,3,3]
h_1^[(S-1)^1]=[0,0,0]  h_2^[(S-1)^1]=[3,1,1]  h_3^[(S-1)^1]=[0,2,2]
h_1^[(S-1)^2]=[0,0,0]  h_2^[(S-1)^2]=[3,0,0]  h_3^[(S-1)^2]=[6,0,0]
norm in K2/K of the component 1 of CK2:[0,0,0]
norm in K2/K of the component 2 of CK2:[0,0,0]
norm in K2/K of the component 3 of CK2:[0,0,0]
Complete capitulation, m(K2)=3, e(K2)=2
\end{verbatim}\ns

\section{Tables for the logarithmic class group}\label{clog} 

Questions of capitulation, in various $p$-extensions, of other arithmetic invariants, 
are at the origin of many papers (see, e.g., in a chronological order \cite{Maire1996, 
KhPr2000, KoMo2000, JauMi2006, Brig2007, Vali2008, Jaul2016, GJN2016, 
Jaul2019$^a$, Jaul2019$^b$, Jaul2022, Gras2022$^a$} and their references); 
they are related to generalized $p$-class groups with conditions of ramification 
and decomposition, to wild kernels of ${\bf K}$-theory, to torsion groups in 
$p$-ramification theory, to Tate--Chafarevich groups, Bertrandias--Payan 
modules, logarithmic class groups. 

The same techniques, using the algebraic norm, may be applied; the results 
essentially depend on the properties of the associated filtration (that we do not 
know for most of the above invariants), whence on the variation of the complexity 
in the $p$-extension $L/K$ considered.

We shall focus on transfers of the logarithmic class groups $\CH_K^\lgm$ in some
totally ramified cyclic $p$-extensions. Recall that $H_K^\lc$ is the maximal abelian locally 
cyclotomic pro-$p$-extension of $K$ and $K^\cyc$ its cyclotomic $\Z_p$-extension.
The tame places totally split in $H_K^\pr/K^\cyc$, so that $H_K^\lc$ 
is the subfield of $H_K^\pr$ fixed by the decomposition groups of the $p$-places.
In Jaulent \cite{Jaul2019$^a$} one finds the following diagram of the main invariants,
showing in particular that $\CH_K$ and $\CH^\lgm_K$ are isomorphic to quotients 
of $\CT_K$ and that $\Gal(H_K^\bp/K^\cyc H_K^\nr)$, $\Gal(H_K^\bp/ H_K^\lc)$ 
are suitable regulators of units (compare with the diagram in 
Definitions \ref{diagram0}):
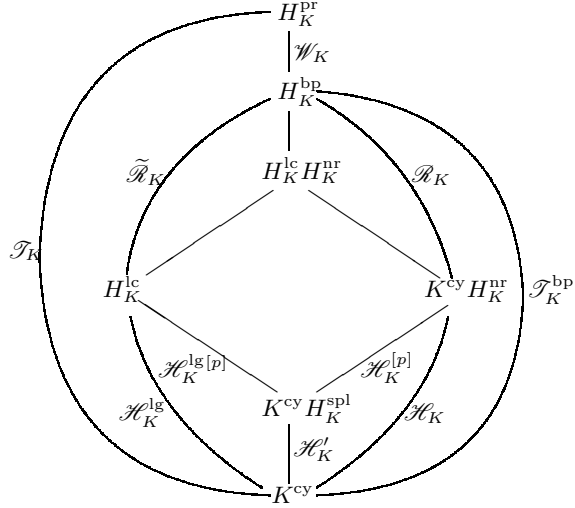
\begin{center}
\vspace{0.6cm}
\setlength{\unitlength}{1.5pt}
\begin{picture}(102,126)(0,0)
\thinlines 
\put(50,130){\ft $H_K^\pr$}
\put(54,120){\ft $\CW_K$}
\put(50,110){\ft $H_K^\bp$}
\put(46,90){\ft $H_K^\lc H_K^\nr$}
\put(6,60){\ft $H_K^\lc$}
\put(87,60){\ft $K^\cyc H_K^\nr$}
\put(46,30){\ft $K^\cyc H_K^{\rm spl}$}
\put(48.5,8){\ft $K^\cyc$}
\bezier{200}(13,55)(18,30)(46,12)
\bezier{200}(93,55)(88,30)(60,12)
\bezier{300}(60,112)(110,108)(112,60)
\bezier{300}(60,10)(110,12)(112,60)
\put(113,60){\ft $\CT^\bp_K$}
\bezier{300}(48,132)(-8,130)(-10,70)
\bezier{300}(48,10)(-8,12)(-10,70)
\put(-18,70){\ft $\CT_K$}
\bezier{300}(94,65)(86,95)(60,110)
\put(84,89){\ft $\CR_K$}
\bezier{300}(12,66)(16,95)(48,110)
\put(12,89){\ft $\wt\CR_K$}
\put(53,127){\line(0,-1){10}}
\put(53,107){\line(0,-1){10}}
\put(53,28){\line(0,-1){15}}
\put(49,87){\line(-3,-2){32}}
\put(58,87){\line(3,-2){33}}
\put(50,36){\line(-3,2){35}}
\put(60,36){\line(3,2){33}}
\put(71,40){\ft $\CH_K^{[p]}$}
\put(20,40){\ft $\CH^\lgm_K{}^{[p]}$}
\put(10,29){\ft $\CH^\lgm_K$}
\put(82,29){\ft $\CH_K$}
\put(54,20){\ft $\CH'_K$}
\end{picture}
\vspace{-0.2cm}
\end{center}

\noindent
In this diagram, $\CH^\lgm_K{}^{[p]}$ (resp. $\CH_K^{[p]}$) is the subgroup, 
of the logarithmic class group $\CH^\lgm_K$ (resp. of the $p$-class group 
$\CH_K$), generated by the classes of the primes dividing $p$. 
So $\CH'_K$ is the quotient 
$\CH_K/\CH_K^{[p]}$ where $H_K^{\rm spl}$ is the splitting field of $p$ in 
$H_K^\nr$, hence the subfield fixed by the image of $\CH_K^{[p]}$, noting
that in our case, $H_K^\nr \cap K^\cyc = K$.

For $L \subset K(\mu_\ell^{})$, $\ell \equiv 1\!\! \pmod {2p^\N}$, 
Theorem \ref{main1} applies to $\CH^\lgm_K$ and $\CH^\lgm_L$, computable 
using Diaz y Diaz--Jaulent--Pauli--Pohst--Soriano-Gafiuk \cite{DJPPS2005} or, 
with the instruction ${\sf bnflog(K,p)}$, Belabas--Jaulent \cite{BeJa2016}. Indeed, 
since $L/K$ is tamely and totally ramified at~$\ell$, then $L^\cyc = L K^\cyc$, 
thus $H_K^\lc$ is linearly disjoint from $L^\cyc$ and the norm 
$\Norm_{L/K} : \CH^\lgm_L = \Gal(H_L^\lc/L^\cyc) \to \CH^\lgm_K 
= \Gal(H_K^\lc/K^\cyc)$ is surjective.

\begin{remark}\label{tricky}
For logarithmic class groups in totally tamely ramified cyclic $p$-extensions 
(in the classical sense), the theory of stability does exist, essentially because the arithmetic
norms are {\it surjective} allowing the criterion of capitulation with $\Nu_{L/K} = \J_{L/K} 
\circ \Norm_{L/K}$, but in numerical applications, our extensions $L/K$ may be partially 
locally cyclotomic, say in $K_{n_0}/K$, which gives some logarithmic non-ramification 
(see \cite[Th\'eor\`eme 1.4]{Jaul1994}).
For instance, if $K$ is a cyclic cubic field, $p=2$ and $\ell = 17$, then $p$
splits in $K_1=K(\sqrt {17})$ and is inert in $L/K_1$; for $K$ quadratic real,
$p=3$, $\ell = 109$, $p$ is totally inert in $L/K$ leading to the classic reasoning.
\end{remark}

In the following program we restrict ourselves to {\it totally logarithmically ramified 
cyclic $p$-extensions}.
Since the invariants $m(\Kk_n)$ are unknown, we can only prove capitulation
from stabilities at some layer, but very probably, many capitulation hold in the 
forthcoming computations with no stability, especially if $N$ is large; we give  
excerpts of the results.

\subsection{Examples with real quadratic fields and \texorpdfstring{$p=3$}{Lg}}
We give examples with $K = \Q(\sqrt m)$ given, $p=3$. There are 
many stabilities, allowing to conclude the capitulation.
Recall that PARI gives the data $\big[\CH^\lgm_{K_n}, \CH^\lgm_{K_n}{}^{\![p]}, 
\CH'_{K_n} \big]$ in this order. One may vary ${\sf p, Nn, Nell, m}$:

\ft\begin{verbatim}
{p=3;Nn=2;Nell=2;m=67;PK=x^2-m;K=bnfinit(PK,1);ClogK0=bnflog(K,p);
forprime(ell=1,10^3,N=valuation(ell-1,p);if(N<Nell,next);
a=znorder(Mod(p,ell));if(valuation(a,p)!=N,next);
r=(kronecker(m,ell)+3)/2;for(n=1,Nn,Qn=polsubcyclo(ell,p^n);
Pn=polcompositum(PK,Qn)[1];Kn=bnfinit(Pn,1);if(n==1,print();
print("PK=",PK," ClogK0=",ClogK0," ell=",ell," N=",N," r=",r));
ClogKn= bnflog(Kn,p);print("ClogK",n,"=",ClogKn)))}
\end{verbatim}\ns
\ft\begin{verbatim}
PK=x^2-67  ClogK0=[[3],[3],[]] 
ell=19  N=2  r=1              ell=397  N=2  r=2
ClogK1=[[9],[9],[]]           ClogK1=[[9,3],[9],[3]]
ClogK2=[[9],[9],[]]           ClogK2=[[27,9],[27],[9]]

ell=37  N=2  r=2              ell=487  N=5  r=2
ClogK1=[[9],[9],[]]           ClogK1=[[9],[9],[]]
ClogK2=[[27],[27],[]]         ClogK2=[[27],[27],[]]

ell=109  N=3  r=1             ell=631  N=2  r=2
ClogK1=[[9],[9],[]]           ClogK1=[[9,3],[9],[3]]
ClogK2=[[9],[9],[]]           ClogK2=[[9,9],[9],[9]]

ell=163  N=4  r=1             ell=811  N=4  r=2
ClogK1=[[9],[9],[]]           ClogK1=[[9],[9],[]]
ClogK2=[[9],[9],[]]           ClogK2=[[27],[27],[]]
\end{verbatim}\ns

Fixing $\ell = 109$ and varying $K = \Q(\sqrt m)$, we get for $r=1$:

\ft\begin{verbatim}
PK=x^2-67  ClogK0=[[3],[3],[]]  r=1       PK=x^2-473  ClogK0=[[3],[],[3]]  r=1
ClogK1=[[9],[9],[]]                       ClogK1=[[3],[],[3]]  
ClogK2=[[9],[9],[]]                       ClogK2=[[3],[],[3]]
Complete capitulation in K2               Complete capitulation in K1
\end{verbatim}\ns
\ft\begin{verbatim}
PK=x^2-321 ClogK0=[[3],[],[3]] r=1        PK=x^2-610 ClogK0=[[3],[3],[]] r=1
ClogK1=[[9],[],[9]]                       ClogK1=[[9],[9],[]] 
ClogK2=[[9],[],[9]]                       ClogK2=[[27],[27],[]]
Complete capitulation in K2               No conclusion
\end{verbatim}\ns

The last example ($m=610$) is the only one of non stability in the interval considered.
Its $3$-class group is trivial, then $\CH'_K = 1$ and $\CH^\lgm_{K_n} = \CH^\lgm_{K_n}{}^{\!\![3]}$.

The following list deals with the case $r=2$. If $\CH^\lgm_{K_n}{}^{\![p]}=1$,
we can conclude, using the general Program \ref{quadratic} for $p$-class groups.

\ft\begin{verbatim}
PK=x^2-106  ClogK0=[[3],[3],[]]  r=2      PK=x^2-659  ClogK0=[[3],[],[3]]   r=2
ClogK1=[[3,3],[3],[3]]                    ClogK1=[[9],[],[9]]  
ClogK2=[[9,3],[9],[3]]                    ClogK2=[[9],[],[9]]
No conclusion                             Complete capitulation in K2
\end{verbatim}\ns
\ft\begin{verbatim}
PK=x^2-238  ClogK0=[[3],[3],[]]  r=2      PK=x^2-679  ClogK0=[[3],[3],[]]   r=2
ClogK1=[[9,3],[9],[3]]                    ClogK1=[[9],[9],[]] 
ClogK2=[[27,9],[27],[9]]                  ClogK2=[[27],[27],[]]
No conclusion                             No conclusion
\end{verbatim}\ns
\ft\begin{verbatim}
PK=x^2-253  ClogK0=[[3],[3],[]]  r=2      PK=x^2-727  ClogK0=[[3],[3],[]]   r=2
ClogK1=[[9],[9],[]]                       ClogK1=[[9],[9],[]]   
ClogK2=[[27],[27],[]]                     ClogK2=[[27],[27],[]]
No conclusion                             No conclusion
\end{verbatim}\ns
\ft\begin{verbatim}
PK=x^2-254  ClogK0=[[3],[],[3]]  r=2      PK=x^2-785  ClogK0=[[3],[],[3]]   r=2
ClogK1=[[3,3],[],[3,3]]                   ClogK1=[[9],[],[9]] 
ClogK2=[[3,3],[],[3,3]]                   ClogK2=[[27],[],[27]]
Complete capitulation in K2               No conclusion
\end{verbatim}\ns
\ft\begin{verbatim}
PK=x^2-326  ClogK0=[[3],[],[3]]  r=2      PK=x^2-790 ClogK0=[[3],[3],[]] r=2
ClogK1=[[3],[],[3]]                       ClogK1=[[3,3],[3],[3]] 
ClogK2=[[3],[],[3]]                       ClogK2=[[9,3],[3],[9]]
Complete capitulation in K1               No conclusion
\end{verbatim}\ns

\subsection{Examples with cyclic cubic fields and \texorpdfstring{$p=2$}{Lg}}
With an analogous program, $\ell = 17$ ($N=3$), one obtains the following 
results, taking into account Remark \ref{tricky}:

\ft\begin{verbatim}
PK=x^3+x^2-54*x-169                       PK=x^3+x^2-336*x-1719  
ClogK0=[[2,2],[],[2,2]]  r=3              ClogK0=[[2,2],[],[2,2]]  r=1
ClogK1=[[2,2],[],[2,2]]                   ClogK1=[[4,4],[],[4,4]]
ClogK2=[[2,2],[],[2,2]]                   ClogK2=[[4,4],[],[4,4]]
ClogK3=[[2,2],[],[2,2]]                   Complete capitulation in K2
Complete capitulation in K1                          
\end{verbatim}\ns
\ft\begin{verbatim}
PK=x^3+x^2-182*x-81                       PK=x^3+x^2-340*x+416
ClogK0=[[2,2],[],[2,2]]  r=1              ClogK0=[[2,2],[2,2],[]]  r=1
ClogK1=[[2,2],[],[2,2]]                   ClogK1=[[2,2],[2,2],[]]
ClogK2=[[2,2],[],[2,2]]                   ClogK2=[[2,2,2,2],[2,2,2,2],[]]
ClogK3=[[2,2],[],[2,2]]                   No conclusion (fake stability in K1/K)
Complete capitulation in K1                          
\end{verbatim}\ns
\ft\begin{verbatim}
PK=x^3-201*x+1072                         PK=x^3+x^2-386*x+1760
ClogK0=[[2,2],[2,2],[]]  r=1              ClogK0=[[4,4],[4,4],[]]  r=1
ClogK1=[[2,2],[2,2],[]]                   ClogK1=[[4,4],[4,4],[]]
ClogK2=[[4,4],[4,4],[]]                   ClogK2=[[8,8],[8,8],[]]
ClogK3=[[8,8],[8,8],[]]                   No conclusion (fake stability in K1/K)
No conclusion (fake stability in K1/K)                          
\end{verbatim}\ns
\ft\begin{verbatim}
PK=x^3+x^2-202*x-1169                     PK=x^3+x^2-486*x+2864
ClogK0=[[2,2],[],[2,2]]  r=1              ClogK0=[[2,2],[2,2],[]]  r=3
ClogK1=[[2,2,2,2],[],[2,2,2,2]]           ClogK1=[[2,2],[2,2],[]]
ClogK2=[[2,2,2,2],[],[2,2,2,2]]           ClogK2=[[4,4,2,2],[4,4],[2,2]]
ClogK3=[[2,2,2,2],[],[2,2,2,2]]           No conclusion (fake stability in K1/K)
Complete capitulation in K2                          
\end{verbatim}\ns
\ft\begin{verbatim}
PK=x^3+x^2-234*x-729                      PK=x^3+x^2-650*x-289
ClogK0=[[2,2],[],[2,2]]  r=1              ClogK0=[[2,2],[],[2,2]]  r=3
ClogK1=[[4,4],[],[4,4]]                   ClogK1=[[4,4,2,2],[],[4,4,2,2]]
ClogK2=[[8,8],[],[8,8]]                   ClogK2=[[4,4,4,4],[],[4,4,4,4]]
ClogK3=[[16,16],[],[16,16]]               No conclusion
No conclusion, probably no capitulation                          
PK=x^3-291*x-1358                         PK=x^3+x^2-692*x-7231
ClogK0=[[2,2],[2,2],[]]  r=1              ClogK0=[[2,2],[],[2,2]]  r=1
ClogK1=[[4,4],[4,4],[]]                   ClogK1=[[2,2,2,2],[],[2,2,2,2]]
ClogK2=[[4,4],[4,4],[]]                   ClogK2=[[2,2,2,2,2,2],[],[2,2,2,2,2,2]]
Complete capitulation in K2               No conclusion
\end{verbatim}\ns

To verify the capitulations, as for the $p$-class groups, it would be interesting 
to have available the logarithmic instructions replacing ${\sf K.clgp}$, once the 
field ${\sf K}$ is given as usual with ${\sf K=bnfinit(PK)}$ and the logarithmic 
class group by ${\sf bnflog(K,p)}$, then an instruction replacing 
${\sf bnfisprincipal(K,A)}$ for an ideal ${\sf A}$ in the logarithmic sense.

\subsection{Capitulation of \texorpdfstring{$\CH_K^\lgm$}{Lg} in 
\texorpdfstring{$K^\cyc$}{Lg} for \texorpdfstring{$K$}{Lg} quadratic}

The program eliminates cases of stability from $K$. The capitulation is obtained 
in $K_3/K$ for almost cases; we give an excerpt of the results  for $n \leq 3$ with 
$m$ up to $1250$. When the exponents increase, no conclusion is possible up to $K_3$
(e.g., successive exponents $4,8,16,32$):

\ft\begin{verbatim}
{p=2;Nn=3;bm=2;Bm=10^4;for(m=bm,Bm,if(core(m)!=m,next);if(Mod(m,8)==2,next);
PK=x^2-m;K=bnfinit(PK,1);r=matsize(idealfactor(K,2))[1];ClogK0=bnflog(K,p);
if(ClogK0==[[],[],[]],next);for(n=1,Nn,Qn=x;for(i=1,n,Qn=Qn^2-2);
Pn=polcompositum(PK,Qn)[1];Kn=bnfinit(Pn,1);if(n==1,print();
print("PK=",PK," ClogK0=",ClogK0," r=",r));ClogKn= bnflog(Kn,p);
if(n==1 & ClogKn==ClogK0,break);print("ClogK",n,"=",ClogKn)))}

PK=x^2-113  ClogK0=[[2],[2],[]]  r=2      PK=x^2-799  ClogK0=[[4],[],[4]]  r=1
ClogK1=[[2],[],[2]]                       ClogK1=[[4,2],[],[4,2]]
ClogK2=[[2],[],[2]]                       ClogK2=[[8,2,2],[],[8,2,2]]
ClogK3=[[2],[],[2]]                       ClogK3=[[8,2,2],[],[8,2,2]]
\end{verbatim}\ns
\ft\begin{verbatim}
PK=x^2-119  ClogK0=[[2],[],[2]]  r=1      PK=x^2-805  ClogK0=[[2],[],[2]]  r=1
ClogK1=[[2,2],[],[2,2]]                   ClogK1=[[2,2],[],[2,2]]
ClogK2=[[2,2],[],[2,2]]                   ClogK2=[[8,4],[],[8,4]]
ClogK3=[[2,2],[],[2,2]]                   ClogK3=[[8,4],[],[8,4]]
\end{verbatim}\ns
\ft\begin{verbatim}
PK=x^2-161  ClogK0=[[4],[4],[]]  r=2      PK=x^2-1023  ClogK0=[[4],[],[4]]  r=1
ClogK1=[[4],[2],[2]]                      ClogK1=[[8],[],[8]]
ClogK2=[[4],[],[4]]                       ClogK2=[[16],[],[16]]
ClogK3=[[4],[],[4]]                       ClogK3=[[32],[],[32]]
\end{verbatim}\ns
\ft\begin{verbatim}
PK=x^2-221  ClogK0=[[2],[],[2]]  r=1      PK=x^2-1067  ClogK0=[[2],[],[2]]  r=1
ClogK1=[[4],[],[4]]                       ClogK1=[[8],[],[8]]
ClogK2=[[4],[],[4]]                       ClogK2=[[16],[],[16]]
ClogK3=[[4],[],[4]]                       ClogK3=[[32],[],[32]]
\end{verbatim}\ns
\ft\begin{verbatim}
PK=x^2-255  ClogK0=[[2],[],[2]]  r=1      PK=x^2-1217  ClogK0=[[8],[8],[]]  r=2
ClogK1=[[2,2],[],[2,2]]                   ClogK1=[[16,2],[8],[4]]
ClogK2=[[4,2],[],[4,2]]                   ClogK2=[[16,2],[4],[4,2]]
ClogK3=[[4,2],[],[4,2]]                   ClogK3=[[16,2],[2],[8,2]]
\end{verbatim}\ns
\ft\begin{verbatim}
PK=x^2-323  ClogK0=[[2],[],[2]]  r=1      PK=x^2-1221  ClogK0=[[4],[],[4]]  r=1
ClogK1=[[8],[],[8]]                       ClogK1=[[8],[],[8]]
ClogK2=[[8],[],[8]]                       ClogK2=[[16],[],[16]]
ClogK3=[[8],[],[8]]                       ClogK3=[[16],[],[16]]
\end{verbatim}\ns
\ft\begin{verbatim}
PK=x^2-357  ClogK0=[[2],[],[2]]  r=1      PK=x^2-1243  ClogK0=[[2],[],[2]]  r=1
ClogK1=[[2,2],[],[2,2]]                   ClogK1=[[2,2],[],[2,2]]
ClogK2=[[4,2,2],[],[4,2,2]]               ClogK2=[[8,4],[],[8,4]]
ClogK3=[[4,2,2],[],[4,2,2]]               ClogK3=[[8,8],[],[8,8]]
\end{verbatim}\ns
\ft\begin{verbatim}
PK=x^2-627  ClogK0=[[2,2],[],[2,2]]  r=1  PK=x^2-1249  ClogK0=[[4],[4],[]]  r=2
ClogK1=[[4,4],[],[4,4]]                   ClogK1=[[8],[4],[2]]
ClogK2=[[8,8],[],[8,8]]                   ClogK2=[[8],[2],[4]]
ClogK3=[[16,8],[],[16,8]]                 ClogK3=[[8],[],[8]]
No conclusion up to K3
\end{verbatim}\ns

\subsection{Conclusion about \texorpdfstring{$\CH^\lgm_K$}{Lg}}
As a conclusion, one can say, from the above examples, that the logarithmic 
class group of a real field $K$ may capitulate in the simplest cyclic
$p$-extensions $L/K$, $L \subset K(\mu_\ell^{})$, as for $p$-class groups; 
this was not so obvious, but in Jaulent \cite{Jaul2019$^b$} is proved
the existence (as for $p$-class groups with Bosca techniques \cite{Bosc2009}) 
of abelian extensions $L_0/\Q$ such that $L =L_0K$ is a capitulation field for 
$\CH^\lgm_K$ (some more general conditions of signature may be assumed for $K$).

Clearly, for imaginary quadratic fields, the fact that, probably, $\CH^\lgm_K$ never capitulates 
in $L$ seems plausible, because of a systematic non-smooth increasing complexity ($p$-rank 
and/or exponent) as shown by the following excerpt:

\ft\begin{verbatim}
PK=x^2+14  ell=109   r=1             PK=x^2+74  ell=109   r=2
ClogK0=[[3],[3],[]]                  ClogK0=[[9],[9],[]] 
ClogK1=[[9],[9],[]]                  ClogK1=[[27,3],[27],[3]]    
ClogK2=[[27],[27],[]]                ClogK2=[[81,9],[81],[9]]
\end{verbatim}\ns
\ft\begin{verbatim}
PK=x^2+41  ell=109   r=1             PK=x^2+107  ell=109   r=1
ClogK0=[[27],[27],[]]                ClogK0=[[9],[9],[]] 
ClogK1=[[81],[81],[]]                ClogK1=[[27,9,3],[27],[9,3]]    
ClogK2=[[243],[243],[]]              ClogK2=[[81,27,9],[81],[27,9]]
\end{verbatim}\ns

The case of capitulation of $\CH^\lgm_K$ in the cyclotomic $\Z_p$-extension of a 
totally real number field (equivalent to Greenberg's conjecture) for which no proof does exist,  
is made very plausible due to a general principle of capitulation of $\CH^\lgm_K$.

\section{Conclusions and prospects}
{\bf a)} We have conjectured in Conjecture \ref{conjcap} that, varying $\ell \equiv 1 
\pmod {2p^\N}$, $N$ large enough, there are infinitely many cases of stability 
from a suitable layer in $K(\mu_\ell^{})$, yielding capitulation of $\CH_K$
(Theorem \ref{main2}\,(i)), which is stronger than the more general capitulation 
conjecture for infinitely many $\ell$'s; this would be coherent with Greenberg's 
conjecture, equivalent to the stability in $K^\cyc$. In other words, our conjecture
may be seen as a ``tame version'' of Greenberg's one, it being understood that 
the towers are finite, so that capitulation needs large $N$'s.

Furthermore, the particular criterion of Theorem \ref{main1}, using the 
algebraic norm $\Nu_{L/K}$ by means of the invariants $m(\Ll)$ and $e(\Ll)$, 
yields capitulation without there necessarily being stability; it shows the link 
between capitulation and complexity (in the meaning of Definition \ref{smooth}), 
of the filtration of $\CH_L$, likely to be governed by natural density 
results (Conjecture \ref{conjprobas}). It is reasonable to think that, restricting 
to primes $\ell \equiv 1 \pmod {2 p^\N}$ with $N \to \infty$, $N-s(\Ll)$ 
becomes larger than $e(\Ll)$ taking into account that $s(\Ll) = 
\Big[\ffrac{\log(m(\Ll))}{\log(p)} \Big]$ is logarithmic regarding $m(\Ll)$ which 
basically depends on the algorithm defining $\CH_L^{i+1}$ from $\CH_L^i$.
This would say that, in huge towers, stabilization occurs at some layer with 
an increasing probability regarding $N$, or, at least, a smooth complexity. 

But in our computations, we were limited to testing with few values of $\ell$ 
(among infinitely many !) and only for the levels $n \leq 3$. Similarly, we were limited to 
small primes $p$ because of the degrees $[K_n : \Q] = [K : \Q]\,p^n$ for PARI calculations.

When capitulation is, on the contrary, structurally impossible (e.g., case of minus 
$p$-class groups of CM-fields or case of torsion groups $\CT_K$ of 
$p$-ramification theory), the complexity of the corresponding invariants necessarily 
increases in any totally ramified cyclic $p$-tower.

{\bf b)}  Because capitulation of $p$-class groups, in a totally ramified cyclic 
$p$-extension, is in connection with its class group complexity, one may 
wonder if this has some repercussion (or not) on the very numerous heuristics on 
repartition of $p$-class groups $\CH_k$ when $\Gal(k/\Q)$ is of order divisible by $p$ 
(see, e.g., A. Bartel--Johnston--Lenstra \cite{BarLen2020,BarJoLe2022} dealing 
with some difficulties about the classical heuristics of Cohen--Lenstra--Martinet--Malle 
\cite{CohLen1983, CohMart1990, CohMart1994,AdMal2015}, and giving attempts 
to modify them), or on the numerous probabilistic works and $\epsilon$-conjectures, like 
Ellenberg--Venkatesh \cite[Conjecture 1,\,\S\,1.2]{EllVen2007}, then \cite{EPW2017, 
PTBW2020, Wan2020, Pier2022, MatWood2023} and \cite{FLN2022} replacing the
class groups of abelian extensions fields $K/k$ by their genus groups ${\mathfrak g}_{K/k}$
of order essentially given by a ``norm factor'' $\ffrac{\prod_v e_v(K/k)}
{(\BE_k : \BE_k \cap {\mathcal N}_{K/k})}$, where ${\mathcal N}_{K/k}$ is the group
of local norms in $K/k$ \cite[Theorem IV.4.2]{Gras2005}.
Then one may cite from \cite[Section 5.1]{Pier2022}: 

\smallskip \noindent
``{\it the truth of the $\ell$-torsion Conjecture} [= $p$-torsion $\epsilon$-Conjecture] 
{\it is implied by the truth of the Cohen--Lenstra--Martinet heuristics on the distribution 
of class groups}.

\smallskip\noindent
showing the similarities of these conjectures, all within a complex analytic frame. 
Thus the only remaining question is: by means of which parameters of $k$, 
$p$-class group heuristics may be defined (structure of $\Gal(k/\Q)$, 
of the prescribed groups $\CH_k$, signature $(r_1(k), r_2(k))$, discriminant $D_k$,
complex analytic formulas ?); all this been insufficient in our opinion since behavior
of $p$-class groups in a $p$-tower is more $p$-adic in nature than
archimedean, as suggested in \cite{Gras2017$^b$}, and experiments with 
extensions $K(\mu_\ell^{})/K$, $K$ fixed, show very sensitive results depending on 
the chosen ramification $\ell$ and not of the orders of magnitude of the parameters.

In other words, is capitulation a governing principle for complexity, or, on the contrary, 
is complexity a governing principle for capitulation~? This is a difficult question all the 
more that numerical examples show that if the tower $L/K$, of degree $d p^N$ with $N$ 
large enough, complexity may become as smooth as possible (for 
instance, stability from some layer, Greenberg's conjecture in $K^\cyc$, etc.), while the 
discriminants become oversized in the towers.

The huge discriminant values are going in the right direction for {\it the $p$-ranks 
$\epsilon$-conjectures}, as proved in Kl\"uners--Wang in some $p$-extensions 
\cite[Theorems 2.1, 2.3]{KluWan2022} and us in \cite{Gras2020, Gras2022} for 
arbitrary towers of $p$-cyclic extensions. 
The much more difficult case of the $p$-torsion $\epsilon$-conjecture $\order \CH_k 
\ll_{\epsilon,p,[k : \Q]} D_k^\epsilon$ remains plausible even if this depends 
on the $p$-adic complexity given, in cyclic $p$-extensions, by the length of the 
filtration and the exponent of the class group.

{\bf c)}  The remarkable circumstance of capitulations in these simplest tamely ramified
cyclic $p$-extensions $L/K$, is certainly a basic principle for many arithmetic properties, 
as the following ones:

\quad (i) The real abelian Main Conjecture, whose proof becomes trivial in the 
semi-simple case as soon as $\ell$ is inert in $K/\Q$ and $\CH_K$ capitulates 
in $L$, because of the general relations $(\CE_{K,\varphi} : \CF_{K,\varphi}) = 
\big (\Norm_{L/K} (\CE_{L,\varphi}) : \CF_{K,\varphi}\big) \times 
\ffrac{\order \CH_{K,\varphi}}{\order \J_{L/K}(\CH_{K,\varphi})}$, implying 
$(\CE_{K,\varphi} : \CF_{K,\varphi}) \geq \order \CH_{K,\varphi}$ 
for all $\varphi \in \Phi_K$ under capitulation \cite[Theorem 4.6]{Gras2022$^b$}.

\quad (ii) In the Iwasawa theory context, we refer for instance to 
Assim--Mazigh--Oukhaba \cite{AsMaOu2017,Mazigh2016,Mazigh2017} about 
generalized Main Conjectures associated to Euler systems built over Stark units 
replacing Leopoldt's cyclotomic units, hoping that capitulation phenomena may give 
new insights in these theories using auxiliary cyclic $p$-extensions $L/K$ in which 
filtration and exact sequence \eqref{suite} are identical.

\quad (iii) Capitulations prevent to get efficient algebraic definitions of $p$-adic isotopic 
components $\CX_\varphi^\alg$ of arithmetic invariants $\CX$ in the non semi-simple case; 
which suggests to replace algebraic norms by arithmetic ones in the definitions and 
thus to use instead the arithmetic $\varphi$-objects $\CX_\varphi^\ar$.

\end{document}